%% file: main.tex
\documentclass{article}

\usepackage{hyperref}

\usepackage{nomencl}
\usepackage{chngcntr}
\usepackage{amsmath}
\usepackage{amsthm}
\usepackage{amscd}
\usepackage{bbm}
\usepackage{enumerate}
\usepackage{amssymb}
\usepackage{palatino}
\usepackage{amscd}
\usepackage{verbatim}
\usepackage{mathrsfs}
\usepackage{esint}
\usepackage[dvipsnames]{xcolor}
\usepackage{hyperref}
\usepackage{tikzsymbols}
\usepackage{xypic}
\usepackage{arabtex}

\usepackage{scalerel}
\usepackage{ifnextok}

\usepackage{hwemoji}

\usepackage[T2A, T1]{fontenc}
\usepackage[utf8]{inputenc}

\usepackage{amsfonts, amsmath, amsthm, amssymb, algorithmic}

\usepackage{xcolor}

\let\HH\H

\newtheorem{thm}{Theorem}[section]
\newtheorem{lem}[thm]{Lemma}
\newtheorem{prop}[thm]{Proposition}
\newtheorem{prin}[thm]{Principle}
\newtheorem{cor}[thm]{Corollary}

\newtheorem{defn}[thm]{Definition}
\newtheorem{rmk}[thm]{Remark}
\newtheorem{algo}[thm]{Algorithm}
\newtheorem{conj}[thm]{Conjecture}
\newtheorem{notation}[thm]{Notation}

\newcommand{\Spec}{\operatorname{Spec}}
\newcommand{\End}{\operatorname{End}}
\newcommand{\GL}{\operatorname{GL}}
\newcommand{\SL}{\operatorname{SL}}

\newcommand{\Hom}{\operatorname{Hom}}

\newcommand{\Frob}{\operatorname{Frob}}
\newcommand{\Gal}{\operatorname{Gal}}

\newcommand{\Jac}{\operatorname{Jac}}

\newcommand{\Nm}{\operatorname{Nm}}
\newcommand{\Sym}{\operatorname{Sym}}

\renewcommand{\dim}{\operatorname{dim}}
\newcommand{\pfrak}{\mathfrak{p}}
\newcommand{\iso}{\cong}
\newcommand{\tr}{\mathrm{tr}}

\newcommand{\Qbar}{{\overline{\mathbb{Q}}}}
\newcommand{\Q}{\mathbb{Q}}
\newcommand{\Z}{\mathbb{Z}}
\newcommand{\N}{\mathbb{N}}
\renewcommand{\C}{\mathbb{C}}
\newcommand{\F}{\mathbb{F}}
\renewcommand{\R}{\mathbb{R}}
\newcommand{\G}{\mathbb{G}}
\renewcommand{\o}{\mathfrak{o}}

\newcommand{\Aut}{\mathrm{Aut}}
\newcommand{\rank}{\mathop{\mathrm{rank}}}

\newcommand{\et}{\text{\'{e}t.}}

\renewcommand{\P}{\mathbb{P}}

\newcommand{\nfrak}{\mathfrak{n}}

\newcommand{\inj}{\hookrightarrow}
\newcommand{\Res}{\mathrm{Res}}
\newcommand{\Fbar}{\overline{\F}}
\newcommand{\Span}{\mathrm{span}}

\newcommand{\im}{\mathrm{im}}
\newcommand{\ur}{\mathrm{ur.}}
\newcommand{\sh}{\mathrm{s.H.}}
\newcommand{\rk}{\mathrm{rank}}
\newcommand{\id}{\mathrm{id}}
\newcommand{\diag}{\mathrm{diag}}
\newcommand{\eps}{\varepsilon}
\renewcommand{\H}{\mathbb{H}}

\newcommand{\sing}{\mathrm{sing.}}

\definecolor{turquoiseblue}{rgb}{0.0, 1.0, 0.94}

\newcommand{\levent}[1]{{\color{pink} #1}}
\newcommand{\brian}[1]{{\bf \color{blue} #1}}

\let\phi\varphi

\let\emptyset\varnothing

\newcommand\mnote[1]{\marginpar{\tiny #1}}

\makeatletter
\let\@@pmod\pmod
\DeclareRobustCommand{\pmod}{\@ifstar\@pmods\@@pmod}
\def\@pmods#1{\mkern4mu({\operator@font mod}\mkern 6mu#1)}
\makeatother

\DeclareSymbolFont{cyrillic}{T2A}{cmr}{m}{n}
\def\makecyrsymbol#1#2{%
  \begingroup\edef\temp{\endgroup
    \noexpand\DeclareMathSymbol{\noexpand#1}
    {\noexpand\mathalpha}{cyrillic}%
    {\expandafter\expandafter\expandafter
     \calccyr\expandafter\meaning\csname T2A\string#2\endcsname\end}}%
  \temp}
\expandafter\def\expandafter\calccyr\string\char#1\end{#1}

\makecyrsymbol\biga\CYRA
\makecyrsymbol\littlea\cyra
\makecyrsymbol\bigb\CYRB
\makecyrsymbol\littleb\cyrb
\makecyrsymbol\bigv\CYRV
\makecyrsymbol\littlev\cyrv
\makecyrsymbol\biggamma\CYRG
\makecyrsymbol\littlegamma\cyrg
\makecyrsymbol\bigd\CYRD
\makecyrsymbol\littled\cyrd
\makecyrsymbol\bige\CYRE
\makecyrsymbol\littlee\cyre
\makecyrsymbol\bigyo\CYRYO
\makecyrsymbol\littleyo\cyryo
\makecyrsymbol\bigzh\CYRZH
\makecyrsymbol\littlezh\cyrzh
\makecyrsymbol\bigz\CYRZ
\makecyrsymbol\littlez\cyrz
\makecyrsymbol\bigi\CYRI
\makecyrsymbol\littlei\cyri
\makecyrsymbol\bigy\CYRISHRT
\makecyrsymbol\littley\cyrishrt
\makecyrsymbol\bigk\CYRK
\makecyrsymbol\littlek\cyrk
\makecyrsymbol\bigell\CYRL
\makecyrsymbol\littleell\cyrl
\makecyrsymbol\bigem\CYRM
\makecyrsymbol\littleem\cyrm
\makecyrsymbol\bigen\CYRN
\makecyrsymbol\littleen\cyrn
\makecyrsymbol\bigoh\CYRO
\makecyrsymbol\littleoh\cyro
\makecyrsymbol\bigp\CYRP
\makecyrsymbol\littlep\cyrp
\makecyrsymbol\biger\CYRR
\makecyrsymbol\littleer\cyrr
\makecyrsymbol\bigs\CYRS
\makecyrsymbol\littles\cyrs
\makecyrsymbol\bigt\CYRT
\makecyrsymbol\littlet\cyrt
\makecyrsymbol\bigu\CYRU
\makecyrsymbol\littleu\cyru
\makecyrsymbol\bigf\CYRF
\makecyrsymbol\littlef\cyrf
\makecyrsymbol\bigx\CYRH
\makecyrsymbol\littlex\cyrh
\makecyrsymbol\bigch\CYRCH
\makecyrsymbol\littlech\cyrch
\makecyrsymbol\bigsh\CYRSH
\makecyrsymbol\littlesh\cyrsh
\makecyrsymbol\bigshya\CYRSHCH
\makecyrsymbol\littleshya\cyrshch
\makecyrsymbol\bighardsign\CYRHRDSN
\makecyrsymbol\littlehardsign\cyrhrdsn
\makecyrsymbol\bigyeri\CYRERY
\makecyrsymbol\littleyeri\cyrery
\makecyrsymbol\bigsoftsign\CYRSFTSN
\makecyrsymbol\littlesoftsign\cyrsftsn
\makecyrsymbol\bigerev\CYREREV
\makecyrsymbol\littleerev\cyrerev
\makecyrsymbol\bigyu\CYRYU
\makecyrsymbol\littleyu\cyryu
\makecyrsymbol\bigya\CYRYA
\makecyrsymbol\littleya\cyrya

\makecyrsymbol\Sha\CYRSH
\makecyrsymbol\hardsign\cyrhrdsn
\makecyrsymbol\bighardsign\CYRHRDSN
\makecyrsymbol\iy\cyrery
\makecyrsymbol\bigiy\CYRERY
\makecyrsymbol\softsign\cyrsftsn
\makecyrsymbol\bigsoftsign\CYRSFTSN

\def\hwemoji@insert#1{\scalerel*{\includegraphics[page=#1]{hwemoji-assets.pdf}}{X}}
\def\hwemoji@yfwry{\hwemoji@insert{444}}
\DeclareUnicodeCharacter{1F381}{\hwemoji@yfwry}
\newcommand{\giftemoji}{\scaleobj{1.5}{\scalerel*{\includegraphics[page=444]{emoji.pdf}}{X}}}

\title{\vspace{-1in}Conditional Algorithmic Mordell}
\author{Levent Alp\"{o}ge \& Brian Lawrence}
\date{}

\usepackage{titlesec}

\titleformat{\paragraph}
{\normalfont\normalsize\bfseries}{\theparagraph}{1em}{}
\titlespacing*{\paragraph}
{0pt}{3.25ex plus 1ex minus .2ex}{1.5ex plus .2ex}

\usepackage{moreenum}

\usepackage{environ}
\newenvironment{todolist}{}

\usepackage{utf8}
\setcode{utf8}

\begin{document}

\newif\ifhidecomments

\ifhidecomments
  \def\mnote#1{{}}
  \def\levent#1{{}}
  \def\brian#1{{}}
  \NewEnviron{hide}{}
  \let\todolist\hide
  \let\endtodolist\endhide
\fi

\maketitle

\input{problem}

\tableofcontents

\input{intro}

\newpage

\input{conjectural_characterization_of_galois_representations_attached_to_abelian_varieties}

\newpage

\input{T_Shafarevich}

\input{T_Mordell}  

\newpage

\input{proof_of_main_algorithmic_theorems}

\newpage

\input{computability}

\newpage

\input{fundamental_algorithms_for_galois_representations}

\newpage

\input{fundamental_algorithms_for_abelian_varieties_over_number_fields}

\newpage

\input{computability_of_polarizations}

\input{borel_harish_chandra} 

\newpage

\input{abelian}

\newpage

\input{algorithmic_decomposition_of_algebras_over_local_fields}

\newpage

\input{algorithmic_decomposition_of_semisimple_algebras_over_number_fields}

\newpage

\input{mumford}

\renewcommand\refname{References.}
\bibliography{references}

\end{document}

%% file: problem.tex
\section*{The problem.}

The problem is the following. Let $f\in \Z[x,y]$. It is known that the set of integral solutions of $f(x,y) = 0$ has a finitary description. The same goes for the set of rational such solutions. So: can one find such finitary descriptions automatically? --- that is, is there an algorithm which, when run on such an $f$, eventually returns a finite description of its set of integral (or rational) points?

~

Thanks to the negative solution of Hilbert's tenth problem \cite{matiyasevich, davis-putnam-robinson, sun} no such algorithm exists for the analogous problem of determining whether there is even a single integral solution to an equation $f(x_1, \ldots, x_n) = 0$ with $f\in \Z[x_1, \ldots, x_n]$ in $n\geq 11$ variables, and it is a conjecture of Baker, Matiyasevich, and Robinson \cite{suns-older-paper, gasarch} that the same should hold when $n\geq 3$, at least with "integral solution" replaced with "positive integral solution". Of course when $n = 1$ the problem is trivial. So the case of points on curves is, it seems, the only interesting one where one has a chance.

%% file: intro.tex
\section{Introduction.}

It is not currently known that the rank of an elliptic curve over $\Q$ is a computable quantity. In other words, it is not currently known that there is a finite-time algorithm, aka Turing machine that terminates on all inputs, that, on input an elliptic curve $E/\Q$, outputs $\rank{E(\Q)}$. However, what is known is that there is a Turing machine that, on input $E/\Q$, outputs $\rank{E(\Q)}$ \emph{if it terminates}, and moreover a standard conjecture\footnote{--- namely the finiteness of the corresponding Tate-Shafarevich group $\Sha(E/\Q)$, or even "just" that $(n, \#|n\cdot \Sha(E/\Q)|) = 1$ for some $n\in \Z^+$ ---} implies that said Turing machine indeed terminates on all inputs.

The purpose of this paper is to do the same\footnote{However: {\bf our algorithm is extraordinarily inefficient!} Our goal is only to produce \emph{some} algorithm, whereas the algorithm conjecturally computing the rank of an elliptic curve $E/\Q$ is efficient enough that it is used in practice.} for the age-old question of computing the rational points on hyperbolic curves. In other words, in this paper we will produce a Turing machine which, on input $C/K$ a smooth projective hyperbolic curve over a number field $K$, outputs $C(K)$ upon termination, and such that standard conjectures imply that said Turing machine terminates on all inputs.

Let us now describe the algorithm.

Let $K$ be a number field, $S$ a finite set of primes of $\mathcal{O}_K$,
and $g$ a positive integer.
The Shafarevich conjecture (a theorem of Faltings \cite[Satz 6]{Fal83})
asserts that there are only finitely many isomorphism classes of abelian varieties
of dimension $g$ over $K$ having good reduction at all primes outside $S$.
Faltings' proof is ineffective:
it does not provide a way to show that a list of abelian varieties is complete.

Assuming standard conjectures, we show that there exists an algorithm, terminating in finite time, that takes $K$, $S$ and $g$ as inputs, and returns a list of all isomorphism classes of abelian varieties of dimension $g$ over $K$ having good reduction outside $S$.

More precisely, in Section \ref{conjecture section} we formulate Conjecture \ref{the key conjecture} and prove that it follows from standard conjectures in arithmetic geometry.

\begin{thm}\label{the key implication}
The Hodge, Tate, and Fontaine-Mazur conjectures imply Conjecture \ref{the key conjecture}.
\end{thm}

Then we show that said conjecture has the following ramifications.

\begin{thm}\label{shafarevich main theorem}
There is a Turing machine $T_{\text{Shafarevich}}$ with the following properties.
\begin{itemize}
\item On input $(g,K,S, d)$, with $g, d \in \Z^+$, $K/\Q$ a number field, and $S$ a finite set of places of $K$, if $T_{\text{Shafarevich}}$ terminates, then it outputs the finitely many polarized $g$-dimensional abelian varieties $A/K$, with polarization of degree $d$, having good reduction outside $S$.
\item Conjecture \ref{the key conjecture} implies that $T_{\text{Shafarevich}}$ always terminates.
\end{itemize}
\end{thm}

\begin{thm}\label{mordell main theorem}
There is a Turing machine $T_{\text{Mordell}}$ with the following properties.
\begin{itemize}
\item On input $(K,C/K)$, with $K/\Q$ a number field and $C/K$ a smooth projective hyperbolic curve, if $T_{\text{Mordell}}$ terminates, then it outputs $C(K)$.
\item Conjecture \ref{the key conjecture} implies that $T_{\text{Mordell}}$ always terminates.
\end{itemize}
\end{thm}

We expect that both algorithms are too slow for practical use.

The algorithms rely on the following. For the sake of exposition we will be slightly imprecise (e.g.\ semisimplifications will be omitted, etc.).
\begin{itemize}
\item Bounds of Masser--W\"ustholz or Raynaud show that,
given an abelian variety over $K$,
one can effectively compute all other $K$-abelian varieties isogenous to it. 
\item \'Etale cohomology gives a correspondence between 
abelian varieties over $K$, with good reduction outside $S$, up to isogeny,
and certain $\ell$-adic representations of the absolute Galois group $\Gal_K$, unramified outside $S$.
\item Using a result of Faltings (a form of the Faltings--Serre method),
an $\ell$-adic Galois representation is determined by the Frobenius traces
at an explicitly computable finite list of primes.
\item Conditionally on the Fontaine--Mazur, Hodge, and Tate conjectures,
one can give local conditions under which an $\ell$-adic Galois representation must come from an abelian variety.
(A more precise but less general result of this form was recently proven by Patrikis, Voloch, and Zarhin \cite{PVZ16}.)
\item These conditions are the limits of conditions modulo $\ell^n$ for each $n\in \Z^+$. Moreover, given $n$, the mod-$\ell^n$ conditions can be checked algorithmically by explicit calculations involving finite flat group schemes.
\end{itemize}

Just for clarity we state the following immediate consequence of Theorem \ref{mordell main theorem}.

\begin{cor}
Assume the Hodge, Tate, and Fontaine-Mazur conjectures. Then there is a finite-time algorithm which computes the rational points on a given hyperbolic curve over a given number field.
\end{cor}

Let us also note in passing that Theorem \ref{mordell main theorem} implies the following.

\begin{thm}\label{hilbert}
There is a Turing machine $T_{\text{Hilbert}}$ with the following properties.
\begin{itemize}
\item On input $(\o,f)$, with $\o\subseteq \o_K$ an order in a number field $K/\Q$ and $f\in \o[x,y]$, if $T_{\text{Hilbert}}$ terminates, then it outputs $\{(a,b)\in \o\times \o : f(a,b) = 0\}$.
\item On input $(K,f)$, with $K/\Q$ a number field and $f\in K[x,y]$, if $T_{\text{Hilbert}}$ terminates, then it outputs a finite-length description\footnote{If the set is finite then the algorithm outputs it, else the geometric genus of $f = 0$ is at most $1$. Thus (the other cases being evident) we need only comment on the case of a smooth curve of genus $1$, where a finite-length "description" means (the set being infinite) a $K$-point along with a "$\Z$-basis" of the group of $K$-points of the curve's Jacobian.} of $\{(a,b)\in K\times K : f(a,b) = 0\}$.
\item Conjecture \ref{the key conjecture} and the "finiteness-of-$\Sha(E/K)$" conjecture imply that $T_{\text{Hilbert}}$ always terminates.
\end{itemize}
\end{thm}

By the "finiteness-of-$\Sha(E/K)$" conjecture we mean the conjecture that all Tate-Shafarevich groups of elliptic curves over $K$ are finite.

The deduction is immediate so we provide it immediately.

\begin{proof}
On input $(R, f)$ with $R\subseteq K$ and $K/\Q$ a number field, write $C_f / K$ (implicitly embedded in $\P^N_{/K}$) for the normalization of the scheme-theoretic closure of $f = 0$ in $\P^2_{/K}$ and $g$ for its genus (which is of course explicitly computable, e.g.\ for fun via point counting over finite fields). If $g\geq 2$ apply Theorem \ref{mordell main theorem}. If $g = 1$ and $R =: \o$ is an order in a number field, apply Baker's lower bound on linear forms in logarithms \cite{baker-linear-forms-in-logs} (see e.g.\ Baker-Coates \cite{baker-coates}). If $g = 0$ apply the usual finite-time algorithm depending on the divisor at infinity of $C_f$ (via an explicit form of Hasse-Minkowski or else Baker's effective solution of $S$-unit equations via his lower bounds on linear forms in logarithms \cite{baker-linear-forms-in-logs}). Else we are in the case $g = 1$ and $R = K$ --- then $C_f$ has no $K$-point if it is not everywhere locally soluble (an explicit finite check by e.g.\ the Hasse bound and Hensel lifting), else it is everywhere locally soluble and so, choosing an explicit $P\in C_f(\Qbar)$ and writing $L/K$ for the Galois closure of its field of definition, one obtains that $C_f$ induces a class $\alpha\in \mathrm{Sel}_{[L:K]}(E_f / K)$ with $E_f := \Jac{C_f}$. By hypothesis the usual compute-$\mathrm{Sel}_N(E_f / K)$-by-day / brute-force-search-to-lower-bound-$\rank\,{E_f(K)}$-by-night algorithm terminates with a "$\Z$-basis" of $E_f(K)$, whence in finite time we determine whether or not $\alpha$ lies in the image of $E_f(K) / [L:K]\to \mathrm{Sel}_{[L:K]}(E_f / K)$ --- if not then $C_f$ has no $K$-points, while if so a brute-force search will find a $K$-point $P\in C_f(K)$ aka a $K$-isomorphism $C_f\simeq E_f$, and we have already found a "$\Z$-basis" of $E_f(K)$.
\end{proof}

\subsection{Outline of the paper.}

In Section \ref{conjecture section}, we state the conjecture (Conjecture \ref{the key conjecture}) on which our conditional results rely, and we prove that Conjecture \ref{the key conjecture} is a consequence of the Hodge, Tate, and Fontaine--Mazur conjectures.

We state the main algorithms in Sections \ref{shafarevich algorithm section} and \ref{mordell algorithm section}; Section \ref{sec:algo_proofs} contains a proof that the algorithms achieve what we claim.

In Section \ref{sec:computability} we discuss how to work with various mathematical objects (number fields, abelian varieties, endomorphisms, etc.) at the level of (finite) byte representations, and the need to approximate objects that cannot be represented by finite bit strings (e.g.\ complex numbers, varieties and morphisms over $\mathbb{C}$, and $\ell$-adic Galois representations).
We also discuss brute-force search as a technique for finding various algebro-geometric objects.

The rest of the paper explains in some detail how to perform various calculations that are used in the main algorithms.  We expect that a significant proportion of this material is known to the experts; for lack of a suitable reference, we wrote the material in some level of detail.  Sections \ref{sec:fund_algo_gal_rep} and \ref{fundamental algorithms for abelian varieties section} present a large number of algorithms for fundamental calculations involving Galois representations and abelian varieties.  In Section \ref{comp_pol_section}, we explain how to find all polarizations (up to isomorphism) of given degree on a given abelian variety.  In Section \ref{sec:int_homology} we show how to compute the singular homology of an abelian variety in terms of differentials and the complex-analytic uniformization; this is needed for various exact calculations involving homology, endomorphisms, and polarizations.  Sections \ref{sec:alg_dec_ql_alg} and \ref{alg_dec_num_fld_alg} contain algorithms for working with semisimple algebras, which we will apply to the endomorphism ring of an abelian variety.  We conclude with Section \ref{sec:mum_coord}, where we explain Mumford's parametrization of a certain moduli space of abelian varieties with level structure.

\subsection{Acknowledgements.}

We would like to thank Manjul Bhargava, Frank Calegari, Ignacio Darago, Matt Emerton, Hao Lee, David Benjamin Lim,  Bjorn Poonen, Michael Stoll, Andrew Sutherland, Akshay Venkatesh, John Voight, and Felipe Voloch for helpful conversations.

LA thanks the Society of Fellows. BL would like to acknowledge support from a grant from the US National Science Foundation (2101985).

Work on this paper was performed at Schloss Schney, the Mathematical Sciences Research Institute, the University of Wisconsin--Madison, Sierra-at-Tahoe, and Kirkwood.
We are grateful to these institutions for their hospitality.

%% file: conjectural_characterization_of_galois_representations_attached_to_abelian_varieties.tex
\section{A characterization of Galois representations attached to abelian varieties.\label{conjecture section}}

The purpose of this section is to state Conjecture \ref{the key conjecture} and to prove (Theorem \ref{the key implication}) that Conjecture \ref{the key conjecture} is a consequence of the Hodge, Tate, and Fontaine--Mazur conjectures. Note that Theorem \ref{the key implication} is closely related to the main result of \cite{PVZ16}.

\begin{conj}\label{the key conjecture}
Let $K/\Q$ be a number field. Let
\[ \rho \colon \Gal(\Qbar/K) \rightarrow \GL_{2g}(\Q_\ell) \]
be a Galois representation such that
\begin{itemize}
\item $\rho$ is unramified outside $S$ and primes above $(\ell)$,
\item for every prime $\pfrak\nmid (\ell)$ of $K$ not in $S$, $\tr(\rho(\Frob_\pfrak))\in \Z$,
\item and, at each place of $K$ above $\ell$, the representation $\rho$ is de Rham, with Hodge--Tate weights $0$ and $1$, each appearing with multiplicity $g$.
\end{itemize}

Then there exists a Galois extension $L/K$, a CM field $E$, a degree one prime $\lambda\vert (\ell)$ of $E$, and an abelian variety $B/L$ admitting $\mathfrak{o}_E\hookrightarrow \End_L(B)$ with good reduction outside $S$ and primes above $\ell$ such that $B\sim_L B^\sigma$ for all $\sigma\in \Gal(L/K)$ and moreover the $E_\lambda\cong \Q_\ell$-adic rational Tate module $V_\lambda(B) := T_\lambda(B)\otimes_{\Z_\ell} \Q_\ell$ is isomorphic to $\rho$ as a $\Gal(\Qbar/K)$-representation: $$V_\lambda(B)\cong \rho.$$

In particular letting $A := \Res_K^L(B)$ we conclude that $V_\ell(A)\cong \rho^{\oplus [E : \Q]\cdot [L : K]}$ as $\Gal(\Qbar/K)$-representations.
\end{conj}

Standard examples of abelian surfaces with quaternionic multiplication demonstrate that we cannot hope to 
take the above $E = \mathbb{Q}$ in general (see e.g.\ Section $4$ of \cite{PVZ16}).

Now let us prove Theorem \ref{the key implication}.

\begin{proof}[Proof of Theorem \ref{the key implication}.]
By Fontaine-Mazur (\cite[Conjecture 1]{FM95}), there is a smooth projective variety $X/K$ and $i,j\in \Z$ such that $\rho$ is a subquotient of $H^i_\et(X/\Qbar,\Q_\ell)(j)$. By the Hodge and Tate conjectures \cite{moonen}, $H^i_\et(X/\Qbar,\Q_\ell)(j)$ is a semisimple $\Gal(\Qbar/K)$-representation, and so $\rho$ is in fact a summand thereof. Write $\pi\in \End_{\Q_\ell[\Gal(\Qbar/K)]}(H^i_\et(X/\Qbar,\Q_\ell)(j))$ for the corresponding projector. Let $M := H^i(X)(j)$, a pure motive with $\Q$-coefficients over $K$ --- thus $\End_{\Q_\ell[\Gal(\Qbar/K)]}(H^i_\et(X/\Qbar,\Q_\ell)(j))\simeq \End_K(M)\otimes_{\Q} \Q_\ell$. Thus $\End_K(M)$ is a semisimple $\Q$-algebra, whence it splits over $\Qbar$. In other words $\pi$ is in the image of $\End_K(M)\otimes_{\Q} \Qbar\inj \End_K(M)\otimes_{\Q} \Qbar_\ell\simeq \End_{\Qbar_\ell[\Gal(\Qbar/K)]}(H^i_\et(X/\Qbar,\Q_\ell)(j))$, where we have implicitly chosen a particular embedding $\Qbar\inj \Qbar_\ell$. Thus, because semisimple algebras over $\Q$ all split over the maximal CM extension of $\Q$, by the Tate conjecture it follows that there is a finite set $\mathscr{C}$ of $(\dim{X})$-dimensional correspondences $C\subseteq X\times X$ and $\alpha_C\in \Qbar$ lying in a CM field such that $\pi = \sum_{C\in \mathscr{C}} \alpha_C\cdot C_*$. Let $E := \Q(\{\alpha_C : C\in \mathscr{C}\})$, a (totally real or imaginary CM) number field because $\mathscr{C}$ is finite, equipped with a choice of embedding $E\inj \Qbar\inj \Qbar_\ell$ (and thus a chosen $\lambda\mid (\ell)$) from our preferred $\Qbar\inj \Qbar_\ell$ from before.

We conclude that $\pi$ is in the image of $\End_K(M)\otimes_{\Q} E$, and so $\rho\otimes_\Q E$ is the $\lambda$-adic realization of an object in the category of pure motives with $E$-coefficients over $K$. On restricting coefficients from $E$ to $\Q$ we conclude that $\Res_\Q^E(\rho\otimes_\Q E)$ is the $\ell$-adic realization of an object $\widetilde{M}$ in the category of pure motives with $\Q$-coefficients over $K$. But because $\rho$ has rational Frobenius traces it follows from Chebotarev that $\Res_\Q^E(\rho\otimes_\Q E)\cong \rho^{\oplus [E : \Q]}$.

So $\rho^{\oplus [E : \Q]}$ is the $\ell$-adic realization of $\widetilde{M}$. Now choose embeddings $K\inj \Qbar\inj \C$. Then the $\Q$-Hodge structure corresponding to (aka Betti realization of) $\widetilde{M}$ has weights $\{\underbrace{(1,0),\ldots,(1,0)}_{g\cdot [E : \Q]}, \underbrace{(0,1),\ldots,(0,1)}_{g\cdot [E : \Q]}\}$ by hypothesis on the weights of $\rho$. Moreover it is polarizable by our construction of $\widetilde{M}$ from the smooth projective $X/K$. It follows by Riemann that there is an isogeny class of abelian varieties over $\C$ with said $\Q$-Hodge structure, and it follows from the Hodge conjecture (applied to a product of one such abelian variety with a power of $X$) that said isogeny class is stable under the action of $\Aut(\C/\Qbar)$, whence it is an isogeny class of abelian varieties over $\Qbar$.

Let $B/\Qbar$ be one of the abelian varieties in said isogeny class, and $L/\Q$ a number field over which $B$ is defined (with an implicitly chosen embedding $L\inj \Qbar\inj \C$). Thus the $\Q$-Hodge structure of $B$ matches the Betti realization of $\widetilde{M}$. By enlarging $L/\Q$ if necessary, without loss of generality we may assume that $L/\Q$ is Galois and that the Hodge class providing the isomorphism between the $\Q$-Hodge structure of $B$ and that of $\widetilde{M}$ is invariant under $\Aut(\C/L)$ --- note that it then follows that $E\inj \End_L^0(B)$, and, by taking a Serre tensor product if necessary, without loss of generality that $\o_E\inj \End_L(B)$. By transporting said Hodge class along the comparison isomorphism between Betti and $\ell$-adic cohomology we find that $V_\ell(B)$ is isomorphic to the $\ell$-adic realization of $\widetilde{M}$ (aka $\rho^{\oplus [E : \Q]}$) as $\Gal(\Qbar/L)$-representations. Note that this immediately implies that $B\sim_L B^\sigma$ for all $\sigma\in \Gal(L/K)$.

Let then $A := \Res_K^L(B)$. It follows that $V_\ell(A)\iso \rho^{\oplus [E : \Q]\cdot [L : K]}$ as $\Gal(\Qbar/K)$-representations, so we are done.
\end{proof}

%% file: T_Shafarevich.tex
\section{The main algorithms.}

In this section we present the two main algorithms: $T_{\text{Shafarevich}}$, which finds all abelian varieties of dimension $g$ over a number field $K$, having good reduction outside $S$, and equipped with a polarization of degree $d$; and $T_{\text{Mordell}}$, which finds all rational points on a curve of genus at least $2$ over a number field $K$.
The algorithms make use of subroutines which are given in later sections.

\subsection{$T_{\text{Shafarevich}}$.}\label{shafarevich algorithm section}

\begin{algo}[$T_{\text{Shafarevich}}$]\label{the shafarevich algorithm}
On input $(g, K, S, d)$,
\begin{enumerate}
\item Check that the input specifies $g, d \in \Z^+$, a number field $K/\Q$, and a finite set $S$ of primes of $K$.
\item Initialize $\giftemoji := \emptyset$.  

(This will be the output set of abelian varieties.)
\item Choose a prime $\ell\in \Z^+$ not dividing $\prod_{\pfrak\in S} \Nm\,{\pfrak}$.

(We're going to work with $\ell$-adic Galois representations.)
\item Follow the proof of Lemma \ref{faltings lemma} on input $(g, K, S, \ell)$ to compute $T$ as guaranteed in Lemma \ref{faltings lemma}.  

($T$ is a finite set of primes of $K$ such that a semisimple rank-$2g$ Galois representation is determined by its Frobenius traces at primes in $T$.)
\item Let $C := \left\{(a_\pfrak)_{\pfrak\in T} : a_\pfrak\in \Z, |a_\pfrak|\leq 2g\cdot \sqrt{\Nm\,{\pfrak}}\right\}$, and initialize $ \text{\c{C}} = C$.

(This $C$ is a list of all possible tuples of Frobenius traces, at primes in $T$, of a Galois representation coming from an abelian variety.  Over the course of the algorithm we will remove elements from $\text{\c{C}}$; $\text{\c{C}}$ is the set of elements of $C$ that ``have not been processed yet.'')
\item Initialize $k_{\text{max}} := 1, H := 1, N := 1$.
\item While $\text{\c{C}} \neq \emptyset$:
\label{day_night}
\begin{enumerate}
\item Increment $k_{\text{max}}\mapsto k_{\text{max}} + 1, H\mapsto H + 1, N\mapsto N + 1$. 

(This loop is going to be executed over increasing tuples $(k_{\text{max}}, H, N)$ such that $N$ is sufficiently large with respect to $k_{\text{max}}$ (see step \ref{N_k}). If the loop were allowed to run forever, $k_{\text{max}}, H, N$ would go to infinity.)
\item \label{N_k}
For all $k \leq k_{\text{max}}$, follow the proof of Lemma \ref{faltings lemma} on input $(k\cdot g, K, S, \ell)$, and let $\widetilde{T}$ be the union of the outputs over all $k \leq k_{\text{max}}$.  Increase $N$ if necessary to ensure that all primes $\pfrak \in \widetilde{T}$ satisfy $\Nm\,{\pfrak} < \frac{\ell^{2N}}{16g^2}$.

(A rank-$kg$ Galois representation is determined by its Frobenius traces at all primes in $\widetilde{T}$; 
here we take $N$ large enough to ensure that if the representation comes from an abelian variety, those Frobenius traces are determined by their value modulo $\ell^N$.)
\item Let $\text{\c{C}}_N := \emptyset$.

(This will be a list of all possible tuples of Frobenius traces at primes in the extended set $\widetilde{T}$.  Each tuple in $(a_\pfrak)_{\pfrak\in T}\in \text{\c{C}}$ will extend to at most one tuple $(a_\pfrak)_{\pfrak\in {\widetilde{T}}}\in \text{\c{C}}_N$.)
\item \label{loop_small_cand} For each $(a_\pfrak)_{\pfrak\in T}\in \text{\c{C}}$, do the following.
\begin{enumerate}
\item Follow the proof of Theorem \ref{computability of the necessary conditions modulo ell to the N} on input $(g, K, S, T, (a_\pfrak)_{\pfrak\in T}, \ell, N)$, to compute $(\Phi, (a_\pfrak)_{\pfrak\in \widetilde{T}})$ as guaranteed in Theorem \ref{computability of the necessary conditions modulo ell to the N}.

(In other words: the trace tuple $(a_\pfrak)_{\pfrak\in T}$ comes from at most one good Galois representation; 
we attempt to lift that Galois representation to a mod-$\ell^N$ representation with good Frobenius traces at all $\pfrak\in \widetilde{T}$. 

If we succeed, $(a_\pfrak)_{\pfrak\in \widetilde{T}}$ is the (necessarily unique) tuple of Frobenius traces at all $\pfrak\in \widetilde{T}$, though there may be multiple mod-$\ell^N$ Galois representations $\rho \in \Phi$ having that same trace tuple.)

\item If $\Phi = \emptyset$ then remove $(a_\pfrak)_{\pfrak\in T}$ from $\text{\c{C}}$.

(If the Galois representation does not lift mod $\ell^N$, or if the lift does not satisfy the Weil bound at all primes in $\widetilde{T}$, we remove the tuple $(a_\pfrak)_{\pfrak\in T}$ from consideration.)
\item If $\Phi \neq \emptyset$, add $(a_\pfrak)_{\pfrak\in \widetilde{T}}$ to $\text{\c{C}}_N$.

(If the Galois representation lifts, add the extended trace tuple to $\text{\c{C}}_N$.  At the end of Step \ref{loop_small_cand}, $\text{\c{C}}_N$ will contain one extended trace tuple for each trace tuple in $\text{\c{C}}$.)
\end{enumerate}

\item Perform a brute-force search, with parameter $H$ (\S \ref{sec:brute_force} and Principle \ref{prin:brute_force}), for abelian varieties $A/K$ in Mumford form (Definition \ref{defn:mumford_form}), of dimension $\dim{A}\leq k\cdot g$ with $g\mid \dim{A}$, which are of good reduction outside $S$.
Let $\bighardsign_H$ denote the resulting finite set of abelian varieties.

Note that Lemma \ref{char_delta_lemma} gives an algorithmically verifiable condition to test whether $A/K$ is an abelian variety in Mumford form,
while Theorem \ref{conductors are computable} allows one to test whether $A$ is of good reduction outside $S$.

\item For each $A\in \bighardsign_H$, do the following.
\begin{enumerate}
\item Let $(a_\pfrak)_{\pfrak\in \widetilde{T}}\in \text{\c{C}}_N$ be the unique element such that $\tr(\rho_{A,\ell}(\Frob_\pfrak))\equiv \left (  \frac{\dim A}{g} \right )\cdot a_\pfrak\pmod*{\ell^N}$ for all $\pfrak\in \widetilde{T}$.  (Use Lemma \ref{comp_frob_trace} to compute Frobenius traces on $A$.)

\item Follow the proofs of Proposition \ref{the property of being a kth power is computable} and Theorem \ref{isogeny classes are computable} to determine whether or not there is a $B/K$ with $\dim{B} = g$ such that $A\sim_K B^{\times k}$, and to compute one such $B/K$ if one exists.

\item If there is such a $B/K$, follow the proof of Theorem \ref{isogeny classes are computable} to compute all polarized abelian varieties $(B', \mathcal{L})/K$ such that $B'\sim_K B$ and $\mathcal{L}$ is a polarization of degree $d$ on $B'$, and add all such $(B', \mathcal{L})$ to $\giftemoji$.

(In other words: using the Masser--Wüstholz isogeny estimates (\cite{MW}, \cite{Bost96}), determine all abelian varieties in the isogeny class of $B$.)
\item Remove $(a_\pfrak)_{\pfrak\in T}$ from $\text{\c{C}}$ if it has not been removed already.

(At this point we have found a single abelian variety $A$ whose Galois representation ``explains'' the tuple $(a_\pfrak)_{\pfrak\in T}$, and used $A$ to determine all abelian varieties with these Frobenius traces; we now remove the tuple $(a_\pfrak)_{\pfrak\in T}$ from consideration.)
\end{enumerate}
\end{enumerate}
\item Output $\giftemoji$.
\end{enumerate}
\end{algo}

%% file: T_Mordell.tex
\subsection{$T_{\text{Mordell}}$.}\label{mordell algorithm section}

\begin{algo}[$T_{\text{Mordell}}$]\label{the mordell algorithm}

On input $(K, C/K)$, 
\begin{enumerate}
\item Let ${\color{magenta} \text{\large \<ڭ>}} := \emptyset$.
\item Follow the proof of Theorem \ref{computability of a nonisotrivial family} on input $(K, C/K)$ to compute a nonisotrivial family $A \rightarrow C$ of abelian varieties parametrized by $C$ (with some projective embedding).

Let $d$ be the degree of the polarization of the given projective embedding of $A$.
\item Apply Theorem \ref{computability of an integral model} to $(K, C, \mathcal{A} \rightarrow C)$ to compute a finite set $S$ of primes of $K$, such that for every $x \in C(K)$, the fiber $\mathcal{A}_x$ has good reduction at all primes outside of $S$.
\item Let $g$ be the relative dimension of the abelian scheme $\mathcal{A} \rightarrow C$, and apply Algorithm \ref{the shafarevich algorithm} on input $(g, K, S)$ to determine (up to isomorphism) the set $\Sigma$ of all abelian varieties $A$ of dimension $g$ over $K$ having good reduction outside $S$.
\item For each abelian variety $A \in \Sigma$:
\begin{enumerate}
\item Run the algorithm of Theorem \ref{the set of polarizations of given multidegree is computable} to determine all polarizations $\mathcal{L}$ of degree $d$ on $A$.
\item For each such $\mathcal{L}$:
\begin{enumerate}
\item 
Apply Theorem \ref{the subfamily of abelian varieties in a family isomorphic to a given one is computable} to determine all $s \in C(\Qbar)$ such that $\mathcal{A}_s \cong A$ as polarized abelian varieties over $\Qbar$.
\item For each such $s$:
\begin{enumerate}
    \item Determine whether $s \in C(K)$.  If so, add $s$ to the set ${\color{magenta} \text{\large \<ڭ>}}$.
\end{enumerate}
\end{enumerate}
\end{enumerate}
\item Return ${\color{magenta} \text{\large \<ڭ>}}$.

\end{enumerate}
\end{algo}

%% file: proof_of_main_algorithmic_theorems.tex
\section{Proofs of Theorems  \ref{shafarevich main theorem} and \ref{mordell main theorem}.}
\label{sec:algo_proofs}

We now prove that the two algorithms just presented achieve what we claim.

\begin{proof}[Proof of Theorem \ref{shafarevich main theorem}.]
We claim that the Turing machine specified in Section \ref{shafarevich algorithm section} has the desired properties.

First, suppose it terminates on input $(g, K, S)$. We claim that its output is then exactly the set of $B/K$ with $\dim{B} = g$ which have good reduction outside $S$.

So let $\widetilde{B}/K$ be in the output, and let us show that it has good reduction outside $S$. Because, in the notation of Algorithm \ref{the shafarevich algorithm}, $\widetilde{B}^{\times k}\sim_K A$ with $A\in \bighardsign_N$ for some $N\in \Z^+$ and each $A/K$ has good reduction outside $S$, it follows that $\widetilde{B}/K$ also has good reduction outside $S$.

Now let $\widetilde{B}/K$ with $\dim{\widetilde{B}} = g$ have good reduction outside $S$, and let us show that it is the output. Certainly, in the notation of Algorithm \ref{the shafarevich algorithm}, $(\tr(\rho_{B,\ell}(\Frob_\pfrak))_{\pfrak\in T}\in \text{\c{C}}$, and it is never removed in Step $7.\text{(c)}.ii$ because $\rho_{\widetilde{B},\ell}\pmod*{\ell^N}\in \Phi$ as defined in Step $7.(c).i$. So, because $T_{\text{Shafarevich}}$ terminates on input $(g, K, S)$, it must be removed in Step $7.(f).iv$, which is to say that at the corresponding value of $N$ there is an $A\in \bighardsign$ such that $\tr(\rho_{A,\ell}(\Frob_\pfrak)) = k\cdot \tr(\rho_{\widetilde{B},\ell}(\Frob_\pfrak))$ for all $\pfrak\in \widetilde{T}$, whence by Lemma \ref{faltings lemma} $A\sim_K \widetilde{B}^{\times k}$. Since $\widetilde{B}^{\times k}\sim_K B^{\times k}$ implies that $\widetilde{B}\sim_K B$, it follows that $\widetilde{B}/K$ is then added to the output in Step $7.\text{(f)}.iv$, so indeed it occurs in the output as desired.

So the output is correct provided $T_{\text{Shafarevich}}$ terminates.

Now let us show that Conjecture \ref{the key conjecture} implies that $T_{\text{Shafarevich}}$ always terminates.

What we must show is that each $(a_\pfrak)_{\pfrak\in T}\in \text{\c{C}}$ as computed in Step $5$ is eventually removed, or in other words that each $(a_\pfrak)_{\pfrak\in T}\in \text{\c{C}}$ which is never removed in Step $7.\text{(c)}.ii$ is eventually removed in Step $7.\text{(f)}.iv$.

So let $(a_\pfrak)_{\pfrak\in T}\in \text{\c{C}}$ be such that it is never removed in Step $7.\text{(c)}.ii$. Thus for all $N\in \Z^+$ there is a $\rho_N: \Gal(\Qbar/K)\to \GL_{2g}(\Z/\ell^N)$ such that
\begin{itemize}
\item $\rho_N$ is unramified outside $S$ and the primes above $(\ell)$,
\item for all $\pfrak\in T$, $\tr(\rho_N(\Frob_\pfrak))\equiv a_\pfrak\pmod*{\ell^N}$,
\item $\det{\rho}\equiv \chi_\ell^g\pmod*{\ell^N}$,
\item for all $\lambda\mid (\ell)$, the Galois module corresponding to $\rho_N$ prolongs to a finite flat group scheme over $\o_{K,\lambda}$,
\item and, for all primes $\pfrak\subseteq \o_K$ not in $S$ and prime to $(\ell)$ with $\Nm\,{\pfrak} < \frac{\ell^{2N}}{16g^2}$, there is a unique $a_\pfrak\in \Z$ with $|a_\pfrak|\leq 2g\cdot \sqrt{\Nm\,{\pfrak}}$ and such that $\tr(\rho_N(\Frob_\pfrak))\equiv a_\pfrak\pmod*{\ell^N}$.
\end{itemize}

By K\HH{o}nig's Lemma\footnote{(There are only finitely many representations $\Gal(\Qbar/K)\to \GL_{2g}(\Z/\ell^N)$ which are unramified outside $S$ and the primes above $(\ell)$ by Hermite-Minkowski.)} it follows that we may assume without loss of generality that $\rho_{N+1}\pmod*{\ell^N} = \rho_N$. Let then $\rho := \varprojlim \rho_N$. Thus $\rho: \Gal(\Qbar/K)\to \GL_{2g}(\Z_\ell)$ is such that
\begin{itemize}
\item $\rho$ is unramified outside $S$ and the primes above $(\ell)$,
\item for all $\pfrak\in T$, $\tr(\rho(\Frob_\pfrak)) = a_\pfrak$,
\item $\det{\rho} = \chi_\ell^g$,
\item for all $\lambda\mid (\ell)$, the $\ell$-divisible group corresponding to $\rho$ prolongs to an $\ell$-divisible group over $\o_{K,\lambda}$,
\item and, for all $\pfrak$ not in $S$ and prime to $(\ell)$, $\tr(\rho(\Frob_\pfrak))\in \Z$.
\end{itemize}

The first property amounts to the statement that $\rho$ factors through a map $G_{K,S,\ell}\to \GL_{2g}(\Z_\ell)$. The fourth property amounts to the statement that the restriction of $\rho$ to an inertia group above $(\ell)$ is crystalline with all Hodge-Tate weights in $\{0, 1\}$, from which, via $\det{\rho} = \chi_\ell^g$, we deduce the second hypothesis of Conjecture \ref{the key conjecture}. The first hypothesis of Conjecture \ref{the key conjecture} is of course the fifth property.

Therefore by Conjecture \ref{the key conjecture} it follows that there is an abelian variety $A/K$ with good reduction outside $S$ such that $\rho_{A,\ell}\iso \rho^{\oplus \frac{\dim{A}}{g}}$ as $\Gal(\Qbar/K)$-representations. Therefore it must be the case that $(a_\pfrak)_{\pfrak\in T}$ is eventually selected in Step $7.\text{(f)}.i$. But then it is removed in Step $7.\text{(f)}.iv$, as desired.
\end{proof}

\begin{proof}[Proof of Theorem \ref{mordell main theorem}.]
Given Theorem \ref{shafarevich main theorem} the second part of Theorem \ref{mordell main theorem} is evident. So let us prove the first part.

Thus suppose $T_{\text{Mordell}}$ terminates on input $(K, C/K)$. Because Step $4.\text{(b)}$ only adds $K$-points of $C$ to the output it follows that the output is a subset of $C(K)$. So let $P\in C(K)$, whence our claim is that $P$ is in the output.

But because $\mathcal{C}/\o_{K,S}$ is proper, $\mathcal{C}(\o_{K,S}) = \mathcal{C}(K) = C(K)$, so that the fibre of $\mathcal{A}\to \mathcal{C}$ over $P\in C(K) = \mathcal{C}(\o_{K,S})$ is an $\o_{K,S}$-integral model of a $g$-dimensional abelian variety $A_P/K$ with good reduction outside $S$. But said abelian variety must then occur in the output of $T_{\text{Shafarevich}}$ on input $(g, K, S)$, whence the fibre of the family $(\mathcal{A}\to \mathcal{C}, \ldots)$ of Mumford data above $P\in C(K) = \mathcal{C}(\o_{K,S})$ must occur in $\mathcal{E}$ when computed in the Step $4.\text{(a)}$ corresponding to $A_P/K$. Hence $P\in C(K)$ must be added to the output in the corresponding Step $4.\text{(b)}$, and we conclude.
\end{proof}

%% file: computability.tex
\section{Some remarks on computability}
\label{sec:computability}

\subsection{Byte representations.}
\label{byte_reps}

In this paper we consider various types of arithmetic-geometric object that can be represented by a finite byte string.  For example:
\begin{itemize}
    \item an element of a number field $K$,
    \item a complex embedding of $K$,
    \item a projective variety over $K$,
    \item a morphism of projective varieties over $K$,
\end{itemize}
and so forth.

Any computer implementation will require some sort of standardized format for the byte string representing each such type of object that arises in the algorithm.
Where it is clear that such a format can be chosen (in the cases above, for instance),
we will not specify the format to be used.

On the other hand, certain types of objects, by nature, do \emph{not} in general admit descriptions by finite byte strings.  For example:
\begin{itemize}
    \item a complex number
    \item a variety over $\mathbb{C}$
    \item an $\ell$-adic Galois representation $\Gal(\Qbar/K)\to \GL_{2g}(\Z_\ell)$.
\end{itemize}
One cannot ask to compute such an object exactly, only to approximate it to some desired precision.

The case of Galois representations deserves special mention.  An $\ell$-adic Galois representation $\Gal(\Qbar/K)\to \GL_{2g}(\Z_\ell)$ is an inverse limit of finite (mod-$\ell^N$) representations
\[ \Gal(\Qbar/K)\to \GL_{2g}(\Z/\ell^N). \]
Any such mod-$\ell^N$ representation factors through some finite quotient group $\Gal(L / K)$, with $L/K$ some finite extension.
It follows that a mod-$\ell^N$ Galois representation can be represented by a finite byte string as follows:
First, specify the number field $L$ and give names to the finitely many elements of the Galois group $\Gal(L/K)$; then specify the function
\[ \Gal(L/K)\to \GL_{2g}(\Z/\ell^N) \]
(which is now a function from one finite set to another).
Our study of Galois representations is made possible by this ``$\ell$-adic approximation'' of $\ell$-adic representations by finitary objects.

\subsection{Abelian varieties and polarizations.}

Throughout the paper, an abelian variety over a number field $K$ will be presented as a projective variety $A$ over $K$, plus the usual structure morphisms (multiplication $A \times A \rightarrow A$, inverse $A \rightarrow A$ and identity $\operatorname{Spec} K \rightarrow A$).
Any such object is automatically a polarized abelian variety as well -- simply take the polarization to be the Neron--Severi class of $\mathcal{O}(1)$ in the given projective embedding.

For explicit computation with polarizations (e.g.\ as in Section \ref{comp_pol_section}), 
note that we can compute
the Chern class of a polarization in terms of an integral basis for $H^1(A)$, by Lemma \ref{find_polarization_in_H2}.

\subsection{The brute-force search principle.}
\label{sec:brute_force}

We will use the ``brute-force search principle'' repeatedly.

\begin{prin}
\label{prin:brute_force}
Suppose we wish to find an $O$ (some type of object) satisfying property $P$.  Assume:
\begin{itemize}
    \item Objects $O$ are represented by finite byte strings.
    \item Given any finite byte string, one can determine algorithmically whether it represents an object $O$ satisfying property $P$ (assuming formatting conventions have been chosen as in \S \ref{byte_reps}).
    \item At least one object $O$ with property $P$ is known to exist.
\end{itemize}
Then one can find an object $O$ with property $P$, by the following silly algorithm: iterate over all finite byte strings (in increasing order of length, say). Test each string to determine whether it represents the object sought, and halt when a string passes the test.
\end{prin}

In some situations we will say ``run brute-force search with parameter $H$'' (for some positive integer $H$) to mean ``search over all byte strings of length at most $H$''.  This has the following features:
\begin{itemize}
    \item For any fixed $H$, the brute-force search with parameter $H$ will terminate in finite time (though it may or may not find the object sought).
    \item If the search is run with larger and larger values of $H$, eventually (for sufficiently large $H$) the search will find the object.
\end{itemize}
This is useful in ``day-and-night searches'', where we wish to interleave a brute-force search with some other type of calculation.

See, for example, Step \ref{day_night} of Algorithm \ref{the shafarevich algorithm}, 
where a search for abelian varieties ``with parameter $H$'' is interleaved with a search for Galois representations.
(The search for Galois representations will give an upper bound on the number of abelian varieties that can arise.)

Some other places brute-force search is used:
\begin{itemize}
\item Lemma \ref{faltings lemma}, to find a set of primes satisfying the conclusion of Chebotarev's density theorem.

\item Lemma \ref{compute_torsion}, to find the $N$-torsion points on an abelian variety.

\item In computing the endomorphism ring of an abelian variety (Lemma \ref{algorithm computing the endomorphism ring}), to find endomorphisms.

\item Theorem \ref{computability of a nonisotrivial family}, to find a nonisotrivial family of abelian varieties over a given curve as base.

\item Lemma \ref{algo_mumf_family}, to find a Mumford form for a given abelian scheme.
\end{itemize}

\begin{rmk}
    In several of the above cases, there are more efficient algorithms than brute-force search.
    We do not attempt to describe them.
    
    Many of these algorithmic questions have been addressed in the literature; we hope this work will inspire further progress on explicit computational questions in arithmetic geometry.
\end{rmk}

%% file: fundamental_algorithms_for_galois_representations.tex
\section{Fundamental algorithms for Galois representations.}
\label{sec:fund_algo_gal_rep}

\subsection{Search for number fields.}

\begin{lem}
\label{hunter}
There is a finite-time algorithm which, on input $(d, K, S)$ with $d\in \Z^+$, $K/\Q$ a number field, and $S$ a finite set of places of $K$, outputs the set of extensions of $K$ of degree at most $d$ which are unramified at all primes outside $S$.
\end{lem}

\begin{proof}
The ramification hypothesis implies that the relative discriminant of $L/K$
is bounded by a constant $D$ that can be effectively computed.

Given the discriminant bound, one can find all such field extensions $L/K$ by a targeted Hunter search \cite[\S 9.3]{Coh00}
\end{proof}

\subsection{Faltings' Lemma.}

\begin{defn}
\label{determines_trace_functions}
Let $d\in \Z^+$, and let $\ell$ be a prime. Let $K/\Q$ be a number field and $S$ a finite set of places of $K$. 

Let $T$ be a set of primes of $K$.  We say that $T$ \emph{determines trace functions} (with parameters $(K, S, d, \ell)$) if it satisfies the following condition:
\begin{quote}
Let $R$ be either $\Q_{\ell}$, $\Z_{\ell}$, or $\Z / \ell^N$ for some $N$.  If two Galois representations
\[
\rho, \rho' \colon \Gal(\Qbar/K)\to \GL_d(R))
\]
satisfy $\tr(\rho(\Frob_\pfrak)) = \tr(\rho'(\Frob_\pfrak))$ for all $\pfrak\in T$, then $\tr\circ \rho = \tr\circ \rho'$ on $\Gal(\Qbar/K)$.
\end{quote}
\end{defn}

\begin{lem}[Faltings, cf.\ e.g.\ \cite{levent-thesis, levent-hypergeometric-paper}]\label{faltings lemma}
Let $d\in \Z^+$. 
Let $K/\Q$ be a number field and $S$ a finite set of places of $K$. 
One can compute a finite set $T_{K,S,d,\ell}$ of primes of $K$ that are prime to $\ell$, 
which is disjoint from $S$, such that for all Galois extensions $L/K$ of degree $[L:K]\leq {\ell}^{4d^2}$ that are unramified outside $S$ and the primes dividing $\ell$, the map $T_{K,S,d,\ell}\to \Gal(L/K)/\text{conj.}$ via $\pfrak\mapsto \Frob_\pfrak$ is surjective. 

Such a $T$ determines trace functions with parameters $(K, S, d, \ell)$.
\end{lem}\noindent

\begin{proof}
The desired set of primes is guaranteed to exist by the Chebotarev density theorem, so they it be found by brute-force search (see Section \ref{sec:brute_force}).

(Alternatively, the usual explicit form of the Chebotarev density theorem \cite{Chebotarev1, Chebotarev2} gives\footnote{Each such Galois extension $L/K$ has explicitly bounded discriminant, and Frobenii of norm bounded in terms of said discriminant represent all conjugacy classes of $\Gal(L/K)$ --- see e.g.\ Theorem $1.1$ of Lagarias-Montgomery-Odlyzko's \cite{lagarias-montgomery-odlyzko} for one such bound.} an explicit $T_{K,S,d,\ell}$ satisfying the above --- thus e.g.\ we may take $T_{K,S,d,\ell}$ satisfying the above and so that all $\pfrak\in T_{K,S,d,\ell}$ satisfy $\Nm\,{\pfrak}\ll_{K,S,d,N} 1$, where the implied constant is explicit.)

This $T_{K,S,d,\ell}$ has the required properties by \cite[proof of Satz 5]{Fal83}; we recall the argument here.

We may assume the coefficient ring $R$ is either $\Z_{\ell}$ or $\Z / \ell^N$.

It suffices to show that the $R$-span of $\im{(\rho\oplus \rho': \Gal(\Qbar/K)\to \GL_d(R)^{\times 2})}$ inside $M_d(R)^{\times 2}$ is in fact spanned by $\bigcup_{\pfrak\in T_{K,S,d,N}} (\rho\oplus \rho')(\Frob_\pfrak)$, where $\Frob_\pfrak\subseteq \Gal(\Qbar/K)$ is the Frobenius conjugacy class of $\pfrak$. To do this one uses Nakayama to reduce mod $\ell$, after which it follows from the hypothesis on $T_{K,S,d,N}$.
\end{proof}

\subsection{Computability of the property of finite flatness.}

\begin{lem}
\label{test_finite_flat}
There is a finite-time algorithm that, on input $(d, K, \lambda, \ell, N, \rho)$ with $K/\Q$ a number field, $\lambda\subseteq \o_K$ a prime of $K$, $\ell\in \Z^+$ a prime with $\lambda\mid (\ell)$, $N\in \Z^+$, and $\rho: \Gal(\Qbar/K)\to \GL_d(\Z/\ell^N)$ continuous, returns true if and only if the Galois module corresponding to $\rho$ prolongs to a finite flat group scheme over $\o_{K,\lambda}$.
\end{lem}

\begin{proof}
Write $\Gal(\Qbar_\ell/K_\lambda)\simeq D_\lambda\subseteq \Gal(\Qbar/K)$ for a decomposition group at $\lambda$. The restriction $\rho\vert_{D_\lambda}$ describes a group scheme $\Spec A$ over the field $K_\lambda$, where $A$ is a Hopf algebra.
To say that $\rho\vert_{D_\lambda}$ comes from a finite flat group scheme
means that there is a Hopf algebra $\mathcal{O}_A \subseteq A$ over $\o_{K,\lambda}$ such that
$\mathcal{O}_A \otimes_{\o_{K,\lambda}} K_\lambda \rightarrow A$ is an isomorphism.

More precisely, $A$ is a $K_\lambda$-vector space equipped with a multiplication map $A \otimes A \rightarrow A$ and a comultiplication $A \rightarrow A \otimes A$, subject to various axioms.
By duality, the comultiplication map defines a multiplication map (i.e.\ a ring structure) on $A^{\vee}$, the $K_\lambda$-linear dual of $A$.
We need to determine whether there exists an $\o_{K,\lambda}$-lattice $\mathcal{O}_A \subseteq A$,
such that $\mathcal{O}_A$ is stable under multiplication,
and its dual is stable under comultiplication.

We can compute $A$ concretely as follows.
The representation $\rho$ is a permutation representation of the Galois group on the finite set $(\Z/\ell^N)^{\oplus d}$;
by Galois theory, it corresponds to the $K$-algebra $A =  {\Hom}_{D_\lambda} ( (\Z/\ell^N)^{\oplus d}, \Qbar_\ell )$.
Writing $L := \Qbar_\lambda^{\ker{\rho\vert_{D_\lambda}}}$, we have $A =  {\Hom}_{\Gal(L/K_\lambda)} ( (\Z/\ell^N)^{\oplus d}, L )$; this allows us to compute $A$ (with its multiplication law) explicitly, working only with finite-dimensional $K_\lambda$-vector spaces.

The natural identification $A^{\vee} \otimes_{K_\lambda} L = L [(\Z/\ell^N)^{\oplus d}]$
gives a multiplication law on $A^{\vee} \otimes_{K_\lambda} L$,
which restricts to a multiplication law $A^{\vee} \times A^{\vee} \rightarrow A^{\vee}$.

Now let $\mathcal{O}_{\max}$ be the maximal order of $A$,
and let $\mathcal{O}_{\min}$ be the dual of the maximal order of $A^{\vee}$.
We can compute both maximal orders by Lemma \ref{maxl_order}.

We know that any Hopf algebra $\mathcal{O}_A \subseteq A$ must satisfy
$\mathcal{O}_{\min} \subseteq \mathcal{O}_A \subseteq \mathcal{O}_{\max}$.
Because the quotient $\mathcal{O}_{\max} / \mathcal{O}_{\min}$ is finite, 
we need only test the finitely many intermediate lattices $\mathcal{O}_A$
to determine whether any is stable under both multiplication and comultiplication.
\end{proof}

\subsection{Computability of the necessary conditions modulo $\ell^N$.}

\begin{thm}\label{computability of the necessary conditions modulo ell to the N}
There is a finite-time algorithm that, on input $(K, S, g, \ell, T, (a_\pfrak)_{\pfrak\in T}, \widetilde{T}, N)$, with
\begin{itemize}
\item $K/\Q$ a number field,
\item $S$ a finite set of places of $K$ not dividing $(\ell)$,
\item $g \in \Z^+$,
\item $\ell\in \Z^+$ a prime,
\item $T$ a finite set, disjoint from $S$, of places of $K$ not dividing $(\ell)$, which determines trace functions (Def.\ \ref{determines_trace_functions}) with parameters $(K, S, 2g, \ell)$,
\item $a_\pfrak\in \Z$ for all $\pfrak\in T$,
\item $\widetilde{T}$ a finite set, disjoint from $S$, of places of $K$ not dividing $(\ell)$, such that $T \subseteq \widetilde{T}$, such that $\Nm\,{\pfrak} < \frac{\ell^{2N}}{16g^2}$ for all $\pfrak\in \widetilde{T}$, and
\item $N\in \Z^+$,
\end{itemize}\noindent
returns $(\Phi, (a_\pfrak)_{\pfrak\in \widetilde{T}})$, where
\begin{itemize}
\item $\Phi$ is the (finite) set of all $\rho: \Gal(\Qbar/K)\to \GL_{2g}(\Z/\ell^N)$ such that
\begin{enumerate}
\item $\rho$ is unramified outside $S$ and the primes above $(\ell)$,
\item for all $\pfrak\in T$, $\tr(\rho(\Frob_\pfrak))\equiv a_\pfrak\pmod*{\ell^N}$,
\item $\det{\rho}\equiv \chi_\ell^g\pmod*{\ell^N}$,
\item for all $\lambda\mid (\ell)$, the Galois module corresponding to $\rho$ prolongs to a finite flat group scheme over $\o_{K,\lambda}$, and
\item \label{cond_ffgs} for all $\pfrak\in \widetilde{T}$, the Frobenius trace $\tr(\rho(\Frob_\pfrak))$ is congruent (modulo $\ell^N$) to some integer $a_{\pfrak, \rho}$ with $\left | a_{\pfrak, \rho} \right | \leq 2g\cdot \sqrt{\Nm\,{\pfrak}}$,
\end{enumerate}
\item and, for all primes $\pfrak\in \widetilde{T}$ and all $\rho \in \Phi$, we have $a_{\pfrak, \rho} = a_{\pfrak}$.
\end{itemize}

\begin{rmk}
    This theorem statement has two parts.  First, we can compute the set $\Phi$ of mod-$\ell^N$ Galois representations $\rho$ satisfying conditions 1-5.  Then, we claim that, for all such $\rho$, the Frobenius traces at primes in $\widetilde{T}$ agree (and can be computed).
\end{rmk}

\begin{proof}
    By Lemma \ref{hunter} one can compute all (finitely many) $\rho: \Gal(\Qbar/K)\to \GL_{2g}(\Z/\ell^N)$ that are unramified outside $S$ and the set of primes dividing $\ell$.

    Now it is simply a matter of testing which of these Galois representations satisfy conditions 2-5 above.  For condition \ref{cond_ffgs}, we use Lemma \ref{test_finite_flat}; the other conditions are straightforward.

    All these $\rho$ will have the same Frobenius traces since $T$ determines trace functions, and of course it is straightforward to compute those traces.
\end{proof}

\end{thm}

%% file: fundamental_algorithms_for_abelian_varieties_over_number_fields.tex
\section{Fundamental algorithms for abelian varieties over number fields.}\label{fundamental algorithms for abelian varieties section}

\subsection{Computability of torsion.}
\begin{lem}
\label{compute_torsion}
    There is a finite-time algorithm that, on input $(K, A/K, N)$ with $A/K$ an abelian variety over a number field $K / \Q$ and $N$ a positive integer, returns the finite set $A[N]$.  
    
    (Specifically, $A$ is input to the algorithm in a projective embedding; the algorithm then outputs the finitely many points $A[N]$ in the same projective embedding, with coordinates presented as exact elements of some number field $L$ extending $K$.)
\end{lem}

\begin{proof}
    Brute-force search (Principle \ref{prin:brute_force}).\footnote{Of course computing the kernel of multiplication of $N$ is more sensible, but we have released ourselves to brute force throughout.}  Note that:
    \begin{itemize}
        \item One can test whether any point of the ambient projective space is an $N$-torsion point of $A$, and
        \item One knows that the total number of $N$-torsion points is $N^{2 \dim A}$.
    \end{itemize}
\end{proof}

\subsection{Computability of Frobenius traces.}
\begin{thm}
\label{comp_frob_trace}
    There is an algorithm that takes as input a number field $K$, an abelian variety $A$ over $K$, and a prime $\pfrak$ of $K$ not dividing the conductor of $A$, and returns the trace of Frobenius acting on $H^1(A, \mathbb{Q}_{\ell})$ (this trace is independent of $\ell$ provided $\ell$ is not divisible by $\pfrak$).
\end{thm}

\begin{proof}
Choose a prime $\widetilde{\ell}\gg_{\Nm\,{\pfrak}} 1$ coprime to the conductor of $A$ and then use Lemma \ref{compute_torsion} to compute $A[\widetilde{\ell}]$ along with its action of $\Gal(\Qbar/K)$.
\end{proof}

\subsection{Computability of the conductor.}

We will in fact show that N\'{e}ron models are computable in Section \ref{computability of the neron model section}, but the (implicit) algorithm will repeatedly invoke the book of Bosch-L\"{u}tkebohmert-Raynaud \cite{bosch-lutkebohmert-raynaud} and thus will not be self-contained. Since we will only need to compute conductors for the main theorems of this paper, we will give a standalone algorithm to compute said conductors in this section.

\begin{algo}\label{algorithm computing the conductor}
On input $(K, A/K)$,
\begin{enumerate}
\item Compute a finite set $S$ of places of $K$ such that $A/K$ has good reduction outside $S$ (e.g.\ via collecting all denominators etc.\ present in Mumford data for the given $A/K$).

\item For each $\pfrak\in S$,
\begin{enumerate}
\item Let $p\in \Z^+$ be the prime such that $\pfrak\mid (p)$.
\item Let $\ell \neq p$ be another prime.
\item Compute the Swan conductor $w_\pfrak$ of $A[\ell]$ as a $\Gal(\Qbar_p/K_\pfrak)$-representation (using Lemma \ref{compute_torsion} to compute $A[\ell]$).
\item Let $N := 10^{10}$.
\begin{enumerate}
\item Let $M := N$.
\item Compute the set $\alpha_N$ of all smooth integral models $\mathcal{A}/\o_K$ of $A/K$ involving coefficients in $\o_K$ of height at most $N$.
\item For each $\mathcal{A}\in \alpha_N$,
\begin{enumerate}
\item Compute the exponent $e_{\mathcal{A}}$ of the $\ell$-part of the component group of its special fibre $\overline{\mathcal{A}} := \mathcal{A}\pmod*{\pfrak}$.
\item Increment $M\mapsto \max(M, e_{\mathcal{A}} + 1)$.
\end{enumerate}
\item Let $e_\pfrak^\uparrow(M) := \infty$.
\item For each $\mathcal{A}\in \alpha_N$,
\begin{enumerate}
\item If not all $P\in A[\ell^{M+1}]^{I_\pfrak}$ prolong over $\o_{K,\pfrak}^\ur$ then continue this loop to the next element of $\alpha_N$.
\item Compute $e_\pfrak(\mathcal{A}) := w_\pfrak + 2g - \dim_{\F_\ell}\ell^M\cdot\nolinebreak \overline{\mathcal{A}}\left(\Fbar_\ell\right)[\ell^{M+1}]$.
\item Replace $e_\pfrak^\uparrow(M)\mapsto \min(e_\pfrak^\uparrow(M), e_\pfrak(\mathcal{A}))$.
\end{enumerate}
\item If $e_\pfrak^\uparrow(M) = w_\pfrak + 2g - \dim_{\F_\ell}{\ell^M\cdot \left(A[\ell^{M+1}]^{I_\pfrak}\right)}$ then let $e_\pfrak := e_\pfrak^\uparrow(M)$ and break this loop.
\item Otherwise increment $N\mapsto M+1$.
\end{enumerate}
\end{enumerate}
\item Return $\nfrak := \prod_{\pfrak\not\in S} \pfrak^{e_\pfrak}$.
\end{enumerate}
\end{algo}

\begin{thm}\label{conductors are computable}
Algorithm \ref{algorithm computing the conductor} is a finite-time algorithm that, on input $(K, A/K)$ with $A/K$ an abelian variety over a number field $K/\Q$, returns its conductor $\nfrak\subseteq \o_K$.
\end{thm}

\begin{rmk}
    In particular, by factoring the conductor, one can determine algorithmically the set of primes of bad reduction of $A$.
\end{rmk}

\begin{proof}
We claim that the algorithm specified in Algorithm \ref{algorithm computing the conductor} terminates with the desired output on all inputs.

Let $A/K$ be an abelian variety over a number field $K/\Q$. Let $\pfrak\subseteq \o_K$ be a prime of $\o_K$. Let $p\in \Z^+$ be the prime such that $\pfrak\mid (p)$. Let $\ell\neq p$ be another prime. Let $N\in \Z^+$. Let $\mathcal{A}/\o_K$ be a smooth integral model of $A/K$ of height at most $N$. Let $M\in \Z^+$ be such that the component group of $\overline{\mathcal{A}} := \mathcal{A}\pmod*{\pfrak}$ is $\ell^M$-torsion. Let $w_\pfrak$ be the Swan conductor of $A[\ell]$ as a $\Gal(\Qbar_p/K_\pfrak)$-representation. Let $e_\pfrak(\mathcal{A}) := w_\pfrak + 2g - \dim_{\F_\ell} \ell^M\cdot \overline{\mathcal{A}}(\Fbar_p)[\ell^{M+1}]$. Let $G\subseteq A(K_\pfrak^\ur)[\ell^{M+1}]$ be the subgroup of points which prolong over $\o_{K,\pfrak}^\ur$.

We claim that $e_\pfrak(\mathcal{A})\geq w_\pfrak + 2g - \dim_{\F_\ell} \ell^M\cdot G$. Indeed this amounts to the claim that $\dim_{\F_\ell} \ell^M\cdot G\leq \dim_{\F_\ell}\ell^M\cdot\nolinebreak \overline{\mathcal{A}}\left(\Fbar_p\right)[\ell^{M+1}]$, which follows from the injection $G\inj \overline{\mathcal{A}}(\Fbar_p)[\ell^{M+1}]$ arising from prolongation and then reduction modulo $\pfrak$, since $G$ prolongs to a finite flat --- indeed \'{e}tale --- subgroup scheme of $\mathcal{A}$ over $\o_{K,\pfrak}^\ur$.

Note also that if $\mathcal{A}/\o_K$ is the N\'{e}ron model of $A/K$ then $G = A(K_\pfrak^\ur)[\ell^{M+1}] = A(\Qbar_p)[\ell^{M+1}]^{I_\pfrak}$ and moreover by smoothness and \'{e}taleness (i.e.\ Hensel) $G = \mathcal{A}(\o_{K,\pfrak}^\ur)[\ell^{M+1}]\to \overline{\mathcal{A}}(\Fbar_p)[\ell^{M+1}]$ is an isomorphism, whence it follows that $e_p(\mathcal{A}) = w_\pfrak + 2g - \dim_{\F_\ell} \ell^M\cdot G$.

So the bound is sharp with equality at least for the N\'{e}ron model, whence once $N$ becomes sufficiently large Step $2.(d).vi$ breaks the loop, which is to say that Algorithm \ref{algorithm computing the conductor} terminates on all inputs.

It remains to show that the corresponding $e_\pfrak^\uparrow(M)$ is the correct conductor exponent at $\pfrak$, which is to say that $\dim_{\F_\ell} \ell^M\cdot (A[\ell^{M+1}]^{I_\pfrak}) = t + 2b$, where $t = \dim{T}$ and $b = \dim{B}$ in the Chevalley decomposition $0\to T\times U\to \overline{\mathcal{A}}^\circ\to B\to 0$, with $\overline{\mathcal{A}}^\circ/(\o_K/\pfrak)$ the connected component of the identity of the mod-$\pfrak$ special fibre of the N\'{e}ron model $\mathcal{A}/\o_K$ of $A/K$, $B/(\o_K/\pfrak)$ an abelian variety, $T/(\o_K/\pfrak)$ a torus, and $U/(\o_K/\pfrak)$ a unipotent group. However as noted above by smoothness and \'{e}taleness we have that $A(\Qbar_p)[\ell^{M+1}]^{I_\pfrak} = A(K_\pfrak^\ur)[\ell^{M+1}]\to \overline{\mathcal{A}}(\Fbar_p)[\ell^{M+1}]$ is an isomorphism, and the latter satisfies $\dim_{\F_\ell} \ell^M\cdot \overline{\mathcal{A}}(\Fbar_p)[\ell^{M+1}] = t + 2b$ because $\ell^M\cdot \overline{\mathcal{A}}(\Fbar_p)\subseteq \overline{\mathcal{A}}^\circ(\Fbar_p)$, completing the proof.
\end{proof}

\subsection{Computability of the endomorphism ring.}

\begin{algo}\label{algorithm computing the endomorphism ring}
On input $(K, A/K)$,
\begin{enumerate}
\item Choose a prime $\ell\in \Z^+$.
\item Let $N := 1$.
\item Until this loop is broken (by the termination condition in (c) below),
\begin{enumerate}
\item Perform a brute-force search with parameter $N$ (Principle \ref{prin:brute_force}) for $K$-subvarieties of $A\times A$ of height at most $N$ which are the graphs of endomorphisms of $A$, and, for each such, compute its action on singular homology $H_1(A, \mathbb{Z})$ (using Lemma \ref{homology_action}).  Let $S_N \subseteq \End H_1(A, \mathbb{Z})$ be a basis for the algebra generated by all such endomorphisms.

\item Compute a $\Sigma_N\subseteq \End_K(A)$ such that $\Sigma_N$ is $\Z$-linearly independent and $\Span_\Z\, \Sigma_N$ is the saturation of $\Span_\Z\, S_N$ inside $\End_K(A)$.  To do this:

\begin{enumerate}
    \item Compute a basis $\Sigma_{N, \sing}$ for the saturation of $\Span_\Z\, S_N$ inside $\End H_1(A, \mathbb{Z})$.
    \item For each $f_\sing \in \Sigma_{N, \sing}$, find, by brute-force search, an endomorphism $f$ of $A$ such that $H_1(f) = f_\sing$.
\end{enumerate}

\item Using Lemma \ref{compute_torsion}, compute the $\ell^N$-torsion $A[\ell^N]$, and determine whether $\rank_\Z\Span_\Z\, \Sigma_N = \dim_{\F_\ell} \ell^{N-1}\cdot \End_{\Gal(\Qbar/K)}(A[\ell^N])$.
If so, return $\Sigma_N$. Else increment $N \mapsto N+1$ and return to the beginning of this loop.
\end{enumerate}
\end{enumerate}
\end{algo}

\begin{thm}\label{the endomorphism ring is computable}
Algorithm \ref{algorithm computing the endomorphism ring} is a finite-time algorithm that, on input $(K, A/K)$ an abelian variety over a number field $K/\Q$, outputs (a $\Z$-basis of) $\End_K(A)$.
\end{thm}

\begin{proof}
For all $N\in \Z^+$, $\rank_\Z \Span_\Z\, \Sigma_N\leq \rank_\Z \End_K(A)$ and $\rank_{\Z_\ell} \End_{\Gal(\Qbar/K)}(T_\ell(A))\leq \dim_{\F_\ell} \ell^{N-1}\cdot \End_{\Gal(\Qbar/K)}(A[\ell^N])$ --- the first is evident and the second follows because the canonical map $\End_{\Gal(\Qbar/K)}(T_\ell(A))/\ell\to \ell^{N-1}\cdot \End_{\Gal(\Qbar/K)}(A[\ell^N])$ is an injection because $(\ell^{N-1}\cdot \phi)(A[\ell^N]) = \phi(A[\ell])$ and so if the left-hand side vanishes then $\phi\in \End_{\Gal(\Qbar/K)}(T_\ell(A))$ is divisible by $\ell$. Moreover for $N$ sufficiently large both inequalities are sharp --- evident in the first case and a consequence of e.g.\ K\HH{o}nig's Lemma in the second case.

But now because $\End_K(A)\otimes_\Z \Z_\ell\to \End_{\Gal(\Qbar/K)}(T_\ell(A))$ is an isomorphism it follows that, for all $N$, $\rank_\Z \Span_\Z\, \Sigma_N\leq \dim_{\F_\ell} \ell^{N-1}\cdot \End_{\Gal(\Qbar/K)}(A[\ell^N])$, with equality if and only if both equal $\rank_\Z \End_K(A)$ (whence $\Span_\Z\, \Sigma_N = \End_K(A)$), and we have seen that equality holds when $N$ is sufficiently large, so that the algorithm always terminates with correct input, as desired.
\end{proof}

\subsection{Computability of the set of polarizations of given degree.}

\begin{thm}\label{the set of polarizations of given multidegree is computable}
There is a finite-time algorithm that, on input $(K, A/K, (d_1, \ldots, d_{\dim{A}}))$ with $A/K$ an abelian variety over a number field $K/\Q$ and $d_i\in \Z^+$, outputs a set of representatives of the (finitely many) $\Aut_K(A)$-orbits of $K$-polarizations of $A/K$ of multidegree $(d_1, \ldots, d_{\dim{A}})$.
\end{thm}

The proof of Theorem \ref{the set of polarizations of given multidegree is computable} is given in Section \ref{comp_pol_section}.

\subsection{Computability of the isomorphism relation.}

\begin{lem}
    \label{aut_of_pav}
    There is a finite-time algorithm that, on input $(K, (A, \mathcal{L}))$ with $K$ a number field and $(A, \mathcal{L})$ a polarized abelian variety over $K$, returns the (finite) set of all $K$-automorphisms of $A$ that preserve the Néron--Severi class of $\mathcal{L}$.
\end{lem}

\begin{proof}
    Compute (Theorem \ref{the endomorphism ring is computable}) the endomorphism ring of $A$, along with its action (Lemma \ref{homology_action}) on $H_1(A, \Z)$.

    Further compute (Lemma \ref{find_polarization_in_H2}) the Chern class of $\mathcal{L}$ as an alternating pairing $\phi(v, w)$ on $H_1(A, \Z)$, and the complex structure $v \mapsto I\cdot v$ (Lemma \ref{homology_basis}) on $H_1(A, \Z) \otimes \R$, in approximate coordinates.

    Compute approximate coordinates for the positive-definite pairing $(v, w) \mapsto \phi(v, Iw)$ on $H_1(A, \Z) \otimes \R$.  Any automorphism of $(A, \mathcal{L})$ must preserve this pairing; this gives us an explicit bound on the size of the coefficients of any such automorphism.  Test all endomorphisms in $\End_K(A)$ up to this bound to determine which fix the Chern class of $\mathcal{L}$.
\end{proof}

\begin{thm}\label{the polarized isomorphism relation is computable}
There is a finite-time algorithm that, on input $(K, (A, \mathcal{L}_1), (B, \mathcal{L}_2))$ with $(A, \mathcal{L}_1)$ and $(B, \mathcal{L}_2)$ two polarized abelian varieties over a number field $K$, returns true if and only if $(A, \mathcal{L}_1) \iso_K (B, \mathcal{L}_2)$.

If so, the algorithm also returns an explicit $K$-isomorphism.
\end{thm}

\begin{proof}
Find some Mumford embedding of $B$ over an extension $L$ of $K$ (Lemma \ref{algo_mumf_family}), polarized by $8 \mathcal{L}_2$.

Compute (Lemma \ref{all_mumf_coords}) all embeddings in Mumford form of $(A_L, 8 \mathcal{L}_1)$.  Check to see whether any of these embeddings agrees with the Mumford embedding of $B$ found above.  If no, return false.

If yes, let $f \colon A_L \rightarrow B_L$ be the isomorphism given by the equality between the Mumford embeddings of $A_L$ and $B_L$.  
Replace $L$ with its Galois closure over $K$.  
Compute (Theorem \ref{the endomorphism ring is computable}) the endomorphism ring $E$ of $A$ over $L$.  Then $\operatorname{Hom}_L(A, B)$ is a finitely-generated free abelian group, isomorphic to $E$ by $e \mapsto f \circ e$.

Compute the action of the finite Galois group $\Gal(L/K)$ on the finitely generated abelian group $\operatorname{Hom}_L(A, B)$.  Determine the fixed set $\operatorname{Hom}_K(A, B) = \operatorname{Hom}_L(A, B)^{\Gal(L/K)}$.

Apply Lemma \ref{aut_of_pav} to find all automorphisms of $A$ that respect the polarization $\mathcal{L}_1$.  For every such automorphism $e$, determine whether $f \circ e \in \operatorname{Hom}_L(A, B)$ is stable under the action of $\Gal(L/K)$.  If it is (for any $e$), return $f \circ e$; otherwise, return false.
\end{proof}

\begin{thm}\label{the isomorphism relation is computable}
There is a finite-time algorithm that, on input $(K, A, B)$ with $A, B / K$ two abelian varieties over a number field $K$, returns true if and only if $A\iso_K B$.

If so, the algorithm also returns an explicit $K$-isomorphism.
\end{thm}

\begin{proof}
Apply Theorems \ref{the set of polarizations of given multidegree is computable} (using e.g.\ the multidegree of the implicitly given polarization of $A$) and \ref{the polarized isomorphism relation is computable}.
\end{proof}

\subsection{Computability of the isogeny class.}

\begin{thm}\label{MW}
There is a finite-time algorithm that, on input $(K, A/K)$ an abelian variety over a number field $K$, returns an $N$ such that, if $B$ is any abelian variety $K$-isogenous to $A$, then there is a $K$-isogeny $A \rightarrow B$ of degree at most $N$.
\end{thm}

\begin{proof}
    This follows from \cite{MW} and \cite{Bost96}.\footnote{For example:
\begin{thm}[Theorem $1.4$ of Gaudron-R\'{e}mond's \cite{gaudron-remond}]
Let $A,A'/K$ be $K$-isogenous abelian varieties over a number field $K$. Write $g := \dim{A}$. Then: there is a $K$-isogeny $\phi: A\to A'$ of degree $$\deg{\phi}\leq \left((14g)^{64g^2}\cdot [K:\Q]\cdot \max(h(A), \log{[K:\Q]}, 1)^2\right)^{2^{10} g^3} =: \kappa(A),$$ where $h(A)$ is the Faltings height of $A$ using Faltings' original normalization.

Consequently, $$|h(A') - h(A)|\leq \frac{1}{2}\log{\kappa(A)}.$$
\end{thm}}
\end{proof}

\begin{thm}\label{isogeny classes are computable}
There is a finite-time algorithm that, on input $(d, K, A/K)$ with $d\in \Z^+$ and $A/K$ an abelian variety over a number field $K/\Q$, returns $\{(B, \mathcal{L})/K : B\sim_K A, \deg(\mathcal{L}) = d\}/\iso_K$.
\end{thm}

\begin{proof}
Compute $\End_K(A)$. Compute (Theorem \ref{MW}) an $N\in \Z^+$ such that $B\sim_K A$ implies there is a $K$-isogeny $\phi: A\to B$ of degree $\deg{\phi} \leq N$. Compute all subgroups $G\subseteq A[N']$ (with $N' \leq N$) which are $\Gal(\Qbar/K)$-stable, i.e.\ $\Gal(\Qbar/K)\cdot G = G$.
For each such $G$, determine (Theorem \ref{the set of polarizations of given multidegree is computable}) all polarizations $\mathcal{L}$ of degree $d$ on $A/G$.
Return the set (i.e.\ remove duplicates via Theorem \ref{the isomorphism relation is computable}) of all such pairs $(A, \mathcal{L})$.
\end{proof}

\subsection{Computability of a $k$-th root.}

\begin{prop}\label{the property of being a kth power is computable}
There is a finite-time algorithm that, on input $(K, A/K, k)$ with $A/K$ an abelian variety over a number field $K/\Q$ and $k\in \Z^+$, returns true if and only if there is an abelian variety $B/K$ such that $A\iso_K B^{\times k}$.

If so, the algorithm also returns such a $B/K$ along with an explicit $K$-isomorphism.
\end{prop}

Here there are a number of clear ways to proceed, e.g.\ by reading the answer off from the $K$-endomorphism ring, via brute force search, or e.g.\ the following.

\begin{proof}
Compute (Theorem \ref{the endomorphism ring is computable}) $\End_K(A)$. Compute (Proposition \ref{num_fld_algebra}) a decomposition
\[\End_K^0(A) \cong \bigoplus_{i=1}^s M_{n_i} (E_i), \]
with each $E_i$ a division algebra. Return false if at least one of the $n_i$'s is not divisible by $k$.

Else (via this isomorphism and suitable elements of the form $\left(\kappa\cdot \delta_{i, a}\cdot \left(\delta_{(j, k), (b, c)}\right)_{j, k = 1}^{n_i}\right)_{i = 1}^s$ with $\kappa\in \Z^+$ sufficiently divisible) compute an abelian subvariety $B/K$ with $B\subseteq A$ such that $A\sim_K B^{\times k}$.

Finally compute all $C\sim_K B$ (Theorem \ref{isogeny classes are computable}) and check if any have $A\iso_K C^{\times k}$ (Theorems \ref{the set of polarizations of given multidegree is computable} and \ref{the isomorphism relation is computable}).
\end{proof}

\subsection{Computability of the N\'{e}ron model.
\label{computability of the neron model section}}

This result is not used in the rest of the paper; we include it for our own amusement.

\begin{thm}\label{neron models are computable}
There is a finite-time algorithm that, on input $(K, A)$ with $K/\Q$ a number field and $A/K$ an abelian variety, outputs its N\'{e}ron model $\mathcal{A} / \o_K$.
\end{thm}

Let us show the local version of this statement first.

\begin{thm}\label{local neron models are computable}
There is a finite-time algorithm that, on input $(\ell, K, \lambda, A)$ with $K_\lambda / \Q_\ell$ a finite extension and $A / K_\lambda$ an abelian variety, outputs its N\'{e}ron model $\mathcal{A} / \o_{K, \lambda}$.
\end{thm}

Note that here we leave implicit the evident finitary statement in terms of $\ell$-adic approximations.

\begin{proof}[Proof of Theorem \ref{local neron models are computable}.]
Implicitly we are given a "desired precision" $N\in \Z^+$ and, to the extent that we can, we will leave implicit the fact that we are working at finite precision (i.e.\ all varieties are over $\o_K / \lambda^{\widetilde{N}}$ and the computations are repeated incrementing $\widetilde{N}\mapsto \widetilde{N} + 1$ until the desired precision is reached).

Let $\omega\neq 0$ be a left-invariant differential on $A / K_\lambda$.

$A / K_\lambda$ is equipped with a polarization, so let $\mathcal{B} / \o_{K, \lambda}$ be the scheme-theoretic closure of $A$ in the corresponding projective space over $\o_{K, \lambda}$. Compute $\mathcal{B}\pmod*{\lambda^N}$. Compute (via Hensel) an $n\in \Z^+$ such that a point $P\in (\mathcal{B}\bmod{\lambda})(\Fbar_\ell)$ lifts to $(\mathcal{B}\bmod{\lambda^n})(\o_K^\sh / \lambda^n)$ if and only if it lifts to $\mathcal{B}(\o_K^\sh)$. Without loss of generality (by replacing $N\mapsto \max(N, n)$) $N\geq n$, and indeed (by incrementing $N\mapsto N + 1$ and restarting the algorithm if necessary, or else via e.g.\ Hensel) $N$ is so large that all the below steps have errors contained in $\lambda^N$.

We will apply the "smoothening process" detailed in Chapter $3$ (and specifically on page $72$) of Bosch-L\"{u}tkebohmert-Raynaud's \cite{bosch-lutkebohmert-raynaud} to $\mathcal{B} / \o_{K, \lambda}$ to obtain a weak N\'{e}ron model $\widetilde{\mathcal{B}} / \o_{K, \lambda}$, exactly as in the first paragraph of their page $74$.

Write $\mathcal{B}_0 := \mathcal{B}$.

For each $i\in \Z^+$, we will inductively define $\mathcal{B}_i$ as follows.

Let $E^{(i, 0)} := \emptyset$.

For each $j\in \Z^+$, we will inductively define $E^{(i, j)}$, $F^{(i, j)}$, $Y_{/ (\o_K / \lambda)}^{(i, j)}$, and $U_{/ (\o_K / \lambda)}^{(i, j)}$ as follows.

Let $F^{(i, j)}\subseteq (\mathcal{B}\bmod{\lambda^N})(\o_K^\sh / \lambda^N)$ be the subset of points not in $E^{(i, j - 1)}$ which reduce into the singular locus of $(\mathcal{B}\bmod{\lambda}) / \Fbar_\ell$ and let $Y_{/ (\o_K / \lambda)}^{(i, j)}\subseteq \mathcal{B}\bmod{\lambda}$ be the scheme-theoretic closure of the reduction of $F^{(i, j)}$. (This latter closed subscheme is computable, again by Hensel.)

Let $U_{/(\o_K / \lambda)}^{(i, j)}\subseteq Y_{/ (\o_K / \lambda)}^{(i, j)}$ be the largest open subscheme which is smooth over $\o_K / \lambda$ and over which $\Omega^1_{\mathcal{B} / \o_K}\mid_{Y_{/ (\o_K / \lambda)}^{(i, j)}}$ is locally free. (This dense open subscheme is computable.)

Let $E^{(i, j)}\subseteq F^{(i, j)}$ be the points reducing into $U_{/ (\o_K / \lambda)}^{(i, j)}$.

Because the dimensions of the $Y_{/ (\o_K / \lambda)}^{(i, j)}$ are strictly decreasing it follows that for $j$ explicitly sufficiently large $Y_{/ (\o_K / \lambda)}^{(i, j)} = \emptyset$. Let $t_i\in \Z^+$ be minimal with this property.

If $t_i = 0$ then break the loop over $i$ and let $\widetilde{\mathcal{B}} := \mathcal{B}_i$.

Otherwise, let $\mathcal{B}_i$ be the blowup of $\mathcal{B}_{i - 1}$ at $Y_{/ (\o_K / \lambda)}^{(i, t_i - 1)}$ and continue the loop.

The proof of Theorem $2$ of Section $3.4$ of Bosch-L\"{u}tkebohmert-Raynaud's \cite{bosch-lutkebohmert-raynaud} implies that this procedure terminates and moreover (and indeed this is how it is used in the proof of Corollary $4$ of Section $3.1$ of the same) that the resulting $\widetilde{\mathcal{B}}$ is a weak N\'{e}ron model of $A / K_\lambda$.

Write $\widetilde{\mathcal{B}}_i$ for the preimages in $\widetilde{\mathcal{B}}$ of the irreducible components of $(\widetilde{\mathcal{B}}\bmod{\lambda}) / (\o_K / \lambda)$.

For each $i$, let $n_i$ be the valuation of $\omega$ at the generic point of the special fibre of $\widetilde{\mathcal{B}}_i$.

Let $n_* := \min_i n_i$ and let $S := \{i : n_i = n_*\}$. Then by Lemma $1$ and Proposition $2$ of Section $4.3$ of Bosch-L\"{u}tkebohmert-Raynaud's \cite{bosch-lutkebohmert-raynaud} it follows that an $\omega$-minimal (their nomenclature) $\o_{K, \lambda}$-model of $A / K_\lambda$ is equivalent (in the sense of page $105$ of Bosch-L\"{u}tkebohmert-Raynaud's \cite{bosch-lutkebohmert-raynaud}) to one of the $\widetilde{\mathcal{B}}_i$ with $i\in S$.

Now we may simply follow the proof of Proposition $4$ of Section $4.3$ of Bosch-L\"{u}tkebohmert-Raynaud's \cite{bosch-lutkebohmert-raynaud} to construct an $\o_{K, \lambda}$-model $\widetilde{\widetilde{\mathcal{B}}} / \o_{K, \lambda}$: shrink (via a finite computation on the special fibre) the special fibres of each of the $\widetilde{\mathcal{B}}_i$ for $i\in S$ to obtain $\widetilde{\mathcal{B}}_i'$, say, so that the diagonal of $A\times_{K_\lambda} A$ is Zariski closed inside $\widetilde{\mathcal{B}}_i'\times_{\o_{K, \lambda}} \widetilde{\mathcal{B}}_j'$ when $i\neq j$ and $i, j\in S$, and then glue the $\widetilde{\mathcal{B}}_i'$ for $i\in S$ along their generic fibres. The proof of Proposition $5$ of Section $4.3$ of Bosch-L\"{u}tkebohmert-Raynaud's \cite{bosch-lutkebohmert-raynaud} then constructs an explicit $\o_{K, \lambda}$-birational group law on the resulting model $\widetilde{\widetilde{\mathcal{B}}}$.

Finally it is a matter of executing a brute-force search for $\o_{K, \lambda}$-group scheme structure (i.e.\ for the multiplication and inversion maps) on $\widetilde{\widetilde{\mathcal{B}}}$ --- Theorem $6$ of Section $4.3$ of Bosch-L\"{u}tkebohmert-Raynaud's \cite{bosch-lutkebohmert-raynaud} amounts to the statement that this search terminates in finite time.

The resulting group scheme over $\o_{K, \lambda}$ is the N\'{e}ron model of $A / K_\lambda$ by Corollary $4$ of Section $4.4$ of Bosch-L\"{u}tkebohmert-Raynaud's \cite{bosch-lutkebohmert-raynaud}, as desired.
\end{proof}

\begin{proof}[Proof of Theorem \ref{neron models are computable}.]
Write $\mathcal{A} / \o_K$ for the N\'{e}ron model of $A/K$.

First note the following finite-time test of whether or not a smooth model $\mathcal{B} / \o_K$ is the N\'{e}ron model of $A$: compute (in the evident way) a finite set $T\supseteq S$ of primes of $K$ containing all ramified primes of $K/\Q$ such that $\mathcal{B}$ has good reduction outside $T$. For each prime $\lambda\in T$, compute (via Hensel) an $n\in \Z^+$ such that if the canonical map induced by the N\'{e}ron mapping property is an $\left(\o_K / \lambda^n\right)$-isomorphism between $\mathcal{B}\pmod*{\lambda^n}$ and the mod-$\lambda^n$ reduction of the N\'{e}ron model of $A / K_\lambda$, then said map is an $\o_{K, \lambda}$-isomorphism. Then check (via Theorem \ref{local neron models are computable}) in finite time whether this is the case. If this test passes for each $\lambda\in T$, then return $\mathcal{B} / \o_K$.

This test is indeed only passed by $\mathcal{A} / \o_K$ --- specifically, the N\'{e}ron mapping property produces a map $\mathcal{B}\to \mathcal{A}$ which is an isomorphism at all primes $\lambda$: at all $\lambda\in T$ by construction, and at all $\lambda\not\in T$ because of functoriality and the fact that abelian schemes are uniquely N\'{e}ron models of their generic fibres (alternatively consider $p$-divisible groups).

Now we simply brute-force search through smooth models $\mathcal{B} / \o_K$ of $A / K$ by enumerating integral models of abelian varieties, testing smoothness (a calculation of e.g.\ a Gr\"{o}bner basis), and then testing $K$-isomorphism of generic fibres (via Theorem \ref{the isomorphism relation is computable}).
\end{proof}

\subsection{Computability of a nonisotrivial family over a given curve.}

\begin{thm}\label{computability of a nonisotrivial family}
There is a finite-time algorithm that, on input $(K, C/K)$ with $C/K$ a smooth projective hyperbolic curve over a number field $K/\Q$, outputs a nonisotrivial family $A \rightarrow C$ of abelian varieties, defined over $K$.
\end{thm}

Note that valid outputs exist for all inputs, thanks to e.g.\ the Kodaira-Parshin family (see for example \cite[\S 7]{LV}).

\begin{proof}
Brute-force search (see Section \ref{sec:brute_force}).
\end{proof}

\subsection{Computability of an integral model.}

\begin{thm}\label{computability of an integral model}
    There is a finite-time algorithm that, on input $(K, C, A \rightarrow C)$, with $K/\Q$ a number field, $C$ a smooth projective hyperbolic curve over $K$, and $A \rightarrow C$ a nonisotrivial family of abelian varieties over $C$, outputs a finite set $S$ of places of $K$, such that for every $x \in C(K)$, the fiber $A_x$ has good reduction at all places of $K$ outside $S$.
\end{thm}

\begin{proof}
    Take any integral model $\mathcal{A} \rightarrow \mathcal{C}$ of $A \rightarrow C$ over $\mathcal{O}_K$.  (We can produce such a model by simply writing down defining equations for $A$ and $C$ in projective coordinates, and clearing denominators.)

    The locus in $\operatorname{Spec} \mathcal{O}_K$ over which $\mathcal{A} \rightarrow \mathcal{C}$ is a smooth morphism of smooth schemes is an effectively computable open set $U$; compute it, and take $S$ to be the set of points (i.e.\ primes) of $\mathcal{O}_K$ not contained in this $U$.

    Then every $K$-point $x \in C(K)$ extends to an $\mathcal{O}_{K, S}$-point $x \in \mathcal{C}(\mathcal{O}_{K, S})$, and the fiber $\mathcal{A}_x$ is a smooth abelian scheme over $\mathcal{O}_{K, S}$.
\end{proof}

%% file: computability_of_polarizations.tex
\section{Computability of the set of polarizations of given degree.}
\label{comp_pol_section}

In this and the following section we prove Theorem \ref{the set of polarizations of given multidegree is computable}: we show how to compute all polarizations, up to automorphism, of given degree, on a given abelian variety.  

That there are only finitely many such polarizations is a theorem of Narisimhan and Nori \cite{NN};
we need to prove that the finite list can be effectively computed.

Narasimhan and Nori reduce the statement about principal polarizations to a theorem of Borel and Harish-Chandra \cite[Theorem 6.9]{BHC} about orbits of arithmetic groups on lattices.  The reduction relies on \cite[Lemma 3.1]{NN}, showing that a certain algebraic group acting on a certain variety \emph{over $\mathbb{C}$} has only finitely many orbits.
We will explicitly enumerate these orbits.

Many of the results of \cite{BHC} are made algorithmic in a paper of Grunewald and Segal \cite{grunewald-segal}.
While \cite[Theorem 6.9]{BHC} is not treated in \cite{grunewald-segal}, 
it is quickly reduced to previous lemmas that are.

There is one piece of the argument of \cite{BHC} that we were unable to make effective in general: in order to determine all the integral orbits of $G$ on $W$, we need to determine which real orbits of $W$ contain a rational point, and identify a specific rational point in each.  In the special case relevant to abelian varieties we do so by an explicit calculation (Lemma \ref{the set of integral orbits in a relevant real orbit is computable}).

\subsection{Introduction and setup.}

Let us recall from \cite{NN} the connection between polarizations and the endomorphism algebra.

Let $A$ be an abelian variety, and fix an ample line bundle $\mathcal{L}_0$ on $A$.
This gives rise to a Rosati involution $\theta = \theta_{\mathcal{L}_0}$, defined by 
\[ \theta(f) = \phi_{\mathcal{L}_0}^{-1} \circ \hat{f} \circ \phi_{\mathcal{L}_0}. \]
For any other line bundle $\mathcal{L}$, the composition $\phi_{\mathcal{L}_0}^{-1} \circ \phi_{\mathcal{L}}$ is an endomorphism of $A$ that is stable under $\theta$;
the map 
\[\rho(\mathcal{L}) = \phi_{\mathcal{L}_0}^{-1} \circ \phi_{\mathcal{L}} \]
defines an injection
\[ \operatorname{NS}(A) \rightarrow (\operatorname{End}(A) \otimes \mathbb{Q})^{\theta}, \]
whose image is the lattice 
\[ \{\phi_{\mathcal{L}_0}^{-1} \circ f | f \in \operatorname{Hom}(A, A') \} \cap (\operatorname{End}(A) \otimes \mathbb{Q})^{\theta} \]
in the $\theta$-fixed subspace of $(\operatorname{End}(A) \otimes \mathbb{Q})$ (\cite[\S 20, Thm.\ 2 and \S 23, Thm.\ 3]{MumAV}).

Let $E := \End(A)$; write $E_{\mathbb{Q}} = \End(A) \otimes \mathbb{Q}$, and let $E_{\mathbb{Q}}^{\theta}$ be the subspace fixed by $\theta$.

We are interested in principal polarizations $\mathcal{L}$ on $A$, up to isomorphism of the pair $(A, \mathcal{L})$.  We consider two polarizations $\mathcal{L}_1$ and $\mathcal{L}_2$ equivalent if there exists an automorphism $f$ of $A$ such that $f^* \mathcal{L}_1$ agrees with $\mathcal{L}_2$ up to an element of $\operatorname{Pic}^0$.
On the level of the endomorphism algebra, we have \[ \rho(f^* \mathcal{L}) = \theta(f) \rho(\mathcal{L}) f, \]
giving an action of $E_{\mathbb{Q}}$ on $E_{\mathbb{Q}}^{\theta}$.

We want to compute a set of representatives for the (finitely many) orbits of the integral group $E$ on the integral lattice $NS(A) \subseteq E_{\mathbb{Q}}^{\theta}$.
The fact that these orbits are finite in number is \cite[Theorem 6.9]{BHC}.
Using results of \cite{grunewald-segal}, it is straightforward to compute these representatives, provided one is given a rational ``basepoint'' in each real orbit;
we do this in Section \ref{bhc_section} and Lemma \ref{computation of reps for integral orbits in a real orbit}.

On the other hand, we do not know of any procedure in the generality of \cite[Theorem 6.9]{BHC} to determine which real orbits contain rational points!  In our particular setting, we explicitly compute the real orbits (Proposition \ref{the set of integral orbits in a relevant real orbit is computable}), and we see that every real orbit contains a rational point -- which means that we can find the required basepoints by brute-force search.

\subsection{Computation of some polarization.}

First a definition.

\begin{defn}
Let $A,B/K$ be abelian varieties over $K$ and $\phi: A\to B$ a $K$-isogeny. Then: $\phi^\vee: B\to A$ is the $K$-homomorphism such that $\phi^\vee\circ \phi = \deg{\phi}$ as elements of $\End_K(A)$. (It follows then that $\phi\circ \phi^\vee = \deg{\phi}$ as elements of $\End_K(B)$.)
\end{defn}

\begin{prop}\label{some polarization is computable}
There is a finite-time algorithm that, on input $(K, A/K)$ with $A/K$ an abelian variety over a number field $K/\Q$, outputs $$(s, (e_i)_{i=1}^s, (n_i)_{i=1}^s, (E_i)_{i=1}^s, (\lambda_i)_{i=1}^s)$$ where $$\End_K^0(A)\iso \bigoplus_{i=1}^s M_{n_i}(E_i)$$ with $s\in \Z^+$, each $E_i$ a division algebra and each $e_i\in E$ the elementary idempotent projecting onto the $i$-th summand, and each $\lambda_i: A_i\to A_i^*$ a $K$-polarization, where $$A_i := \diag(\underbrace{1, 0, \ldots, 0}_{n_i})\cdot e_i\cdot A.$$
\end{prop}

\begin{proof}[Proof of Proposition \ref{some polarization is computable}.]
Compute, via Theorem \ref{the endomorphism ring is computable}, $\End_K^0(A)$. Compute, via Proposition \ref{num_fld_algebra}, $$\left(s, (e_i)_{i=1}^s, (n_i)_{i=1}^s, (E_i)_{i=1}^s\right)$$ where $$\End_K^0(A)\iso \bigoplus_{i=1}^s M_{n_i}(E_i)$$ with each $E_i$ a division algebra and each $e_i\in E$ the elementary idempotent projecting onto the $i$-th summand. (Note that then $e_i\cdot A\sim_K e_j\cdot A$ implies $i=j$.) Thus via $$\diag(\underbrace{1, 0, \ldots, 0}_{n_i})\in M_{n_i}(E_i),$$ i.e.\ via letting $A_i := \diag(\underbrace{1, 0, \ldots, 0}_{n_i})\cdot e_i\cdot A$, we obtain a $K$-isogeny decomposition $$A\sim_K \prod_{i=1}^s A_i^{\times n_i}$$ with each $A_i/K$ $K$-simple (indeed with $\End_K^0(A_i)\simeq E_i$) and $A_i\sim_K A_j$ only if $i = j$. Finally compute via brute force a $K$-polarization $\lambda_i: A_i\to A_i^*$ on each $A_i/K$.
\end{proof}

\subsection{Reduction to computing the set of symmetric endomorphisms of given degree modulo automorphisms.}

\begin{prop}\label{the set of endomorphisms of given degree modulo automorphisms is computable}
There is a finite-time algorithm that, on input $$(K, \widetilde{A}/K, \widetilde{\lambda}, n, d),$$ with $\widetilde{A}/K$ a $K$-simple abelian variety over a number field $K/\Q$, $\widetilde{\lambda}: \widetilde{A}\to \widetilde{A}^*$ a $K$-polarization of $A/K$, and $n, d\in \Z^+$, outputs a finite set $$\Phi\subseteq \End_K(\widetilde{A}^{\times n})$$ such that $\Aut_K(\widetilde{A}^{\times n})\cdot \Phi$ is the set of $K$-endomorphisms of $A/K$ of degree $d$ which are symmetric with respect to the Rosati involution of $\End_K^0(\widetilde{A}^{\times n})$ induced by $\widetilde{\lambda}^{\times n}: \widetilde{A}^{\times n}\to (\widetilde{A}^{\times n})^*$.
\end{prop}

Note that the Rosati involution induced by a $K$-polarization $\lambda: B\to B^*$ of an abelian variety $B/K$ is $$\phi\mapsto \tau_\lambda(\phi) := \frac{1}{\deg{\lambda}}\cdot (\lambda^\vee\circ \phi^*\circ \lambda),$$ with $\phi^*$ the dual $K$-endomorphism of $B^*$ induced by $\phi$ and $\lambda^\vee: B^*\to B$ the $K$-homomorphism such that $\lambda^\vee\circ \lambda = \deg{\lambda}$, i.e.\ such that $\ker{\lambda^\vee}\simeq B[\deg{\lambda}]/(\ker{\lambda})$ under the $K$-isomorphism $B^*\simeq B/(\ker{\lambda})$ furnished by $\lambda$.

Let us now show that Proposition \ref{the set of endomorphisms of given degree modulo automorphisms is computable} implies Theorem \ref{the set of polarizations of given multidegree is computable}.

\begin{proof}[Proof of Theorem \ref{the set of polarizations of given multidegree is computable} assuming Proposition \ref{the set of endomorphisms of given degree modulo automorphisms is computable}.]
First we reduce to the situation where $A$ has the form $A = \widetilde{A}^{\times n}$, with $\widetilde{A}$ a simple abelian variety.  Use Lemma \ref{the endomorphism ring is computable} to compute the endomorphism ring $E = \End_K(A)$, and use Proposition \ref{global_decomp} to decompose $E$ as a direct sum of matrix algebras $$E \cong \bigoplus_{i=1}^s M_{n_i} (E_i).$$  Compute $B_i$, the kernel of $1 - e_i$, where $e_i \in E$ is an idempotent projecting onto the $i$-th factor.  

Let $d_0$ be the degree of the natural map $$\bigoplus_{i=1}^s B_i \rightarrow A.$$  Then every polarization on of degree $d$ pulls back to a polarization of degree $d d_0^2$ on $\bigoplus_{i=1}^s B_i$; conversely, a polarization on $\bigoplus_{i=1}^s B_i$ descends to $A$ if its Chern class in $$H^2(A) \otimes \Q \cong \bigoplus_{i=1}^s H^2(B_i) \otimes \Q$$ is integral.  (Recall that we can compute Chern classes effectively in terms of an integral basis, Lemma \ref{find_polarization_in_H2}.)

So it suffices to find all polarizations of given degree on $\bigoplus_{i=1}^s B_i$.  Any such polarization is a direct sum of polarizations on the individual factors $B_i$, so it is enough to find all polarizations of given degree on a single factor $B_i$.

We now rename $B_i$ as $A$, and assume throughout that $A := \widetilde{A}^{\times n}$ is a power of a $K$-simple abelian variety.

Write $R := \End_K(\widetilde{A})$, $A = \widetilde{A}^{\times n}$, and $\lambda := \widetilde{\lambda}^{\times n}$. We find all $K$-polarizations of $A/K$ of given degree $d$ --- up to the action of $M_n(R)^\times\simeq \Aut_K(A)$ via $$\gamma\cdot \lambda' := \gamma\circ \lambda'\circ \gamma^*$$ --- as follows. Given a $K$-endomorphism $\phi: A\to A$ of degree $d^{\dim{A} - 1}\cdot \deg{\lambda}$ with kernel containing that of $\lambda$ and which is symmetric under the Rosati involution induced by $\lambda$, we let $\lambda': A\to A^*$ be the $K$-homomorphism such that $\phi = (\lambda')^\vee\circ \lambda$ (via the fact that $\phi$ factors through $A/(\ker{\lambda})\simeq A^*$), and then it is routine to determine in finite time whether or not this $\lambda'$ is indeed a $K$-polarization of $A/K$.

Thus we need only show that each such $K$-polarization (regarded as a symmetric $K$-homomorphism $\lambda': A\to A^*$) arises in this manner. But given such a $\lambda': A\to A^*$, the $K$-endomorphism $$(\lambda')^\vee\circ \lambda: A\to A$$ is of degree $d^{\dim{A} - 1}\cdot \deg{\lambda}$, has kernel containing that of $\lambda$, and is symmetric under the Rosati involution induced by $\lambda$. 
\end{proof}

\subsection{Reduction to the effectivization of certain cases of Borel--Harish-Chandra.}

\begin{prop}\label{the set of integral orbits in a relevant real orbit is computable}
(Determination of the integral orbits in a real orbit)

There is a finite-time algorithm that, on input $(K, A/K, \lambda, n, k, (\phi_i)_{i=1}^k)$, with $A/K$ a $K$-simple abelian variety over a number field $K/\Q$, $\widetilde{\lambda}: A\to A^*$ a $K$-polarization, $n\in \Z^+$, $k\in \Z^+$ the $\Z$-rank of $R := \End_K(A)$, and $(\phi_i)_{i=1}^k$ a $\Z$-basis of $R$, outputs a finite set $\Xi$ such that, writing:
\begin{itemize}
\item $\lambda := \widetilde{\lambda}^{\times n}: A^{\times n}\to (A^{\times n})^*$,
\item $i: M_n(R)\inj M_{k\cdot n}(\Z)$ for the map induced by the $\Z$-basis $(\phi_i)_{i=1}^k$,
\item $G$ the algebraic group over $\Q$ defined by $G(S) := M_n(R\otimes_\Z S)^\times$,
\item $G_\Z := G(\Q)\cap i^{-1}\left(\GL_{n\cdot \rk_{\Z} R}(\Z)\right) = M_n(R)^\times$,
\item $V$ the trivial vector bundle over $\Spec{\Q}$ defined by $V(S) := \End_K^0(A)\otimes_\Q S\simeq M_n(R\otimes_\Z S)$ (whence $G\inj \Aut_\Q(V)$ is an algebraic representation over $\Q$),
\item $V_\Z := \End_K(R)$ (thus $V_\Z$ is a lattice inside $V$ which is $G_\Z$-stable),
\item $\tau_\lambda: V\to V$ the Rosati involution induced by $\lambda$,
\item $V^{\mathrm{sym.}} := \ker(\tau_\lambda - \id)$, with $G$-action given by $\gamma\cdot \phi := \gamma\circ \phi\circ \tau_\lambda(\gamma)$ (whence $V^{\mathrm{sym.}}$ is an algebraic $G$-representation defined over $\Q$),
\item $V^{\mathrm{sym.}}_\Z := V^{\mathrm{sym.}}\cap V_\Z$ (thus $V^{\mathrm{sym.}}_\Z$ is a lattice inside $V$ which is $G_\Z$-stable),
\item $\det: G\to \G_m$ the evident map (defined using $i$), regarded as furnishing $\Q$ with the structure of an algebraic $G$-representation defined over $\Q$,
\item $W := V^{\mathrm{sym.}}\oplus \det\oplus \det^{-1}$ as $G$-representations defined over $\Q$,
\item and $W_\Z := V^{\mathrm{sym.}}_\Z\oplus \Z\oplus \Z$ (thus $W_\Z$ is a lattice inside $W$ which is $G_\Z$-stable),
\end{itemize}\noindent
it follows that $G_\Z\cdot \Xi = \{(\star, \pm 1, \pm 1)\in W_\Z\} = G(\R)\cdot \{(\star, \pm 1, \pm 1)\in W(\R)\}\cap W_\Z$.
\end{prop}

Let us now show that Proposition \ref{the set of integral orbits in a relevant real orbit is computable} implies Proposition \ref{the set of endomorphisms of given degree modulo automorphisms is computable} and thus Theorem \ref{the set of polarizations of given multidegree is computable}.

\begin{proof}[Proof of Proposition \ref{the set of integral orbits in a relevant real orbit is computable} assuming Proposition \ref{the set of endomorphisms of given degree modulo automorphisms is computable}.]
Compute such a $\Xi$ via Proposition \ref{the set of integral orbits in a relevant real orbit is computable}. For each $(\phi, \pm 1, \pm 1)\in \Xi$ we may check if there is a $\gamma\in G_\Z$ such that $$\ker{(\gamma\circ \phi\circ \tau_\lambda(\gamma))}\supseteq \ker{\lambda}$$ by computing a finite set of generators of $G_\Z$ \cite[Theorem B]{grunewald-segal}, then a finite set of representatives in $G_\Z$ of $G_\Z\pmod*{\!\!\deg{\phi}}$, and finally checking whether there is a $\gamma\in G_\Z$ among said representatives such that $$\ker{(\gamma\circ \phi\circ \tau_\lambda(\gamma))}\supseteq \ker{\lambda},$$ since both are subsets of $A[\deg{\phi}]$. If there is no such $\gamma\in G_\Z$ we may remove said representative, and if there is such a $\gamma\in G_\Z$ (thus already computed in terms of the generators provided by \cite[Theorem B]{grunewald-segal}) then without loss of generality $\gamma = \id$, and, writing the relevant representative as $\phi\in \End_K(A)$ and, letting $\lambda' : A\to A^*$ be such that $\phi = (\lambda')^\vee\circ \lambda$ (which is without loss of generality a $K$-polarization since this is easily checked in finite time), since $$\gamma\circ \phi\circ \tau_\lambda(\gamma) = (\gamma\circ \lambda'\circ \gamma^*)\circ \lambda$$ it follows that there is exactly one orbit of $K$-polarizations under $\Aut_K(A)$ corresponding to the $G_\Z$-orbit of $\phi$.
\end{proof}

So, in the notation of Proposition \ref{the set of integral orbits in a relevant real orbit is computable}, we have reduced to computing a set of representatives for the finitely many $G_\Z$-orbits inside the set $$\{(\phi, \pm 1, \pm 1)\in W_\Z\} = G(\R)\cdot \{(\phi, \pm 1, \pm 1)\in W(\R)\}\cap W_\Z.$$

\subsection{Characterization of the relevant real orbits.\label{characterization of the relevant real orbits section}}

Now let us invoke the Albert classification: by Albert (see for example \cite[\S 21]{MumAV}) the tuple $(R\otimes_\Z \Q, \tau_{\widetilde{\lambda}})$ falls into one of the following cases.

\subsubsection{Type I.\label{type i subsection}}

In the first case $R\otimes_\Z \Q$ is a totally real number field and $\tau_{\widetilde{\lambda}}$ is trivial. Hence $$\tau_\lambda: M_n(R\otimes_\Z \Q)\to M_n(\R\otimes_\Z \Q)$$ is the standard transpose and $$R\otimes_\Z \R\simeq \bigoplus_{R\hookrightarrow \R} \R,$$ whence $$M_n(R)\otimes_\Z \R\simeq M_n(R\otimes_\Z \R)\simeq \bigoplus_{R\hookrightarrow \R} M_n(\R),$$ and so $$V(\R)\simeq \bigoplus_{R\hookrightarrow \R} M_n(\R)$$ and similarly $$V^{\mathrm{sym.}}(\R)\simeq \bigoplus_{R\hookrightarrow \R} \Sym^2(\R^{\oplus n})$$ under this isomorphism.

Thus $G(\R)\simeq \prod_{R\hookrightarrow \R} \GL_n(\R)$, acting on each factor of $V^{\mathrm{sym.}}(\R)$ independently via the action of $\GL_n(\R)$ on $\Sym^n(\R^{\oplus n})\subseteq M_n(\R)$ via $(g,x)\mapsto g\cdot x\cdot g^t$, or in other words the action of $\GL_n(\R)$ on the quadratic form $v\mapsto v^t\cdot x\cdot v$.

Hence in particular from Sylvester's law of inertia we conclude that each $G(\R)$-orbit on $W(\R)$ intersects the following set at least once:

$$\left\{((\diag(x_i^{(v)})_{i=1}^n)_{v: R\hookrightarrow \R}, a, b) : a, b\in \R^\times, \forall i, v, x_i^{(v)}\in \{-1, 0, 1\}, \forall v, i < j, x_i^{(v)}\leq x_j^{(v)}\right\}.$$

In particular if $\phi\in V^{\mathrm{sym.}}(\R)$ has $\deg{\phi}\neq 0$ and $\delta, \eps\in \{\pm 1\}$, then there is a tuple $(k_v)_{v: R\hookrightarrow \R}$ with each $0\leq k_v\leq n$ such that $$G(\R)\cdot (\phi, \delta, \eps) = G(\R)\cdot \left(\diag\left(\underbrace{-1, \ldots, -1}_{k_v}, \underbrace{1, \ldots, 1}_{n-k_v}\right), \frac{\delta}{\sqrt{\deg{\phi}}}, \eps\cdot \sqrt{\deg{\phi}}\right).$$ Let us also note that the latter orbit is closed since $((X_v)_{v: R\hookrightarrow \R}, a, b)$ lies in said orbit if and only if $ab = \delta \eps$, $b^2\cdot \prod_{v: R\hookrightarrow \R} \det(X_v) = \deg{\phi}$, and each $X_v$ has signature $(n-k_v, k_v)$ (a closed condition since $\det(X_v)\neq 0$ via the previous equality) --- note that these conditions also allow us to check if a given $(\phi, a, b)\in W(\Q)$ with $ab\cdot \deg{\phi}\neq 0$ lies in a given $G(\R)$-orbit in finite time. Let us also note that $W(\Q)$ does intersect each such orbit since e.g.\ the locus of matrices in $\SL_n(\R)$ with given signature is both closed and open, and also that if $\phi\in V^{\mathrm{sym.}}(\C)$ has $\deg{\phi}\neq 0$ and $\delta, \eps\in \{\pm 1\}$, then $$G(\C)\cdot (\phi, \delta, \eps) = G(\C)\cdot \left(\id, \frac{\delta}{\sqrt{\deg{\phi}}}, \eps\cdot \sqrt{\deg{\phi}}\right),$$ and now the latter orbit is Zariski closed because $((X_v)_{v: R\hookrightarrow \R}, a, b)$ lies in said orbit if and only if $ab = \delta \eps$ and $b^2\cdot \prod_{v: R\hookrightarrow \R} \det(X_v) = \deg{\phi}$.

\subsubsection{Type II.\label{type ii subsection}}

In the second case $R\otimes_\Z \Q$ is a quaternion algebra over a totally real field $K/\Q$ which splits at all real places and $\tau_{\widetilde{\lambda}}$ is trivial on $K$ and such that there is an isomorphism of $\R$-algebras (computable by e.g.\ brute-force search for an element of $R\otimes_\Z \Q$ with totally real characteristic polynomial, i.e.\ for a totally real splitting field) $$R\otimes_\Z \R\simeq \bigoplus_{K\hookrightarrow \R} M_2(\R)$$ taking $\tau_{\widetilde{\lambda}}$ to the coordinatewise standard transpose on the right-hand side. Hence via the same isomorphism $$M_n(R\otimes_\Z \R)\simeq \bigoplus_{K\hookrightarrow \R} M_n(M_2(\R))\simeq \bigoplus_{K\hookrightarrow \R} M_{2n}(\R)$$ with $\tau_\lambda$ also taken to the coordinatewise transpose on the right-hand side. Thus exactly as in Section \ref{type i subsection} $$V^{\mathrm{sym.}}(\R)\simeq \bigoplus_{K\hookrightarrow \R} \Sym^2(\R^{\oplus 2n})$$ and $$G(\R)\simeq \prod_{K\hookrightarrow \R} \GL_{2n}(\R)$$ with the coordinatewise action again given by $(g,x)\mapsto g\cdot x\cdot g^t$, and we conclude in exactly the same way as in Section \ref{type i subsection} with explicit representatives for the $G(\R)$-orbits on $W(\R)$ as well as a finite-time algorithm to test membership of a given $(\phi, a, b)\in W(\Q)$ in each relevant orbit (all of which are seen to be closed --- with the analogous statements for $G(\C)$-orbits meant in the Zariski topology --- and to intersect $W(\Q)$).

\subsubsection{Type III.\label{type iii subsection}}

In the third case $R\otimes_\Z \Q$ is a quaternion algebra over a totally real field $K/\Q$ which is ramified at all real places and $\tau_{\widetilde{\lambda}}$ is trivial on $K$ and such that there is a computable isomorphism of $\R$-algebras $$R\otimes_\Z \R\simeq \bigoplus_{K\hookrightarrow \R} \H$$ taking $\tau_{\widetilde{\lambda}}$ to coordinatewise canonical involution, where by the canonical involution on the Hamilton quaternions $\H$ we mean $(1, i, j, k)\mapsto (1, -i, -j, -k)$. Hence via the same isomorphism $$M_n(R\otimes_\Z \R)\simeq \bigoplus_{K\hookrightarrow \R} M_n(\H)$$ with $\tau_\lambda$ taken to coordinatewise conjugate transpose $X\mapsto X^\dag$.

Thus now $$V^{\mathrm{sym.}}(\R)\simeq \bigoplus_{K\hookrightarrow \R} \mathrm{Herm}_{n\times n}(\H),$$ with $$\mathrm{Herm}_{n\times n}(\H) := \{X\in M_n(\H) : X^\dag = X\},$$ and $$G(\R)\simeq \prod_{K\hookrightarrow \R} M_n(\H)^\times$$ with the coordinatewise action given by $(g,x)\mapsto g\cdot x\cdot g^\dag$. Now instead of the Sylvester law of inertia we need a classification of Hermitian forms on $\H^{\oplus n}$, but this is routine --- the same Gram-Schmidt argument one uses to classify quadratic forms over $\R$ works to classify Hermitian forms on $\H^{\oplus n}$ by the same data: $(d_0, d_+, d_-)\in \N^{\times 3}$ with $d_0 + d_+ + d_- = n$, with each form in the orbit equivalent to $$(\alpha_i)_{i=1}^n\mapsto \sum_{i = d_0 + 1}^{d_0 + d_+} |\alpha_i|^2 - \sum_{i = d_0 + d_+ + 1}^n |\alpha_i|^2.$$

So again exactly as in Sections \ref{type i subsection} and \ref{type ii subsection} we find explicit representatives for the $G(\R)$-orbits on $W(\R)$ and a finite-time algorithm to test membership of a given element of $W(\Q)$ in each relevant orbit, all of which are seen to be closed (in the real topology, while again for $G(\C)$-orbits we mean in the Zariski topology) and to intersect $W(\Q)$.

\subsubsection{Type IV.\label{type iv subsection}}

In the fourth case $R\otimes_\Z \Q$ is a central simple algebra over an imaginary CM field $K/\Q$ (with maximal totally real subfield $K^+/\Q$, say) and $\tau_{\widetilde{\lambda}}$ restricts to complex conjugation on $K$ and is such that there is a (computable, in the same way as in e.g.\ Section \ref{type i subsection}) isomorphism of $\R$-algebras $$R\otimes_\Z \R\simeq \bigoplus_{K^+\hookrightarrow \R} M_k(\C)$$ (with $k$ the index of $R\otimes_\Z \Q$ over its centre $K$) taking $\tau_{\widetilde{\lambda}}$ to coordinatewise conjugate transpose $X\mapsto X^\dag$. Hence via the same isomorphism $$M_n(R\otimes_\Z \R)\simeq \bigoplus_{K^+\hookrightarrow \R} M_{kn}(\C)$$ with $\tau_\lambda$ taken to coordinatewise conjugate transpose as well.

Thus now $$V^{\mathrm{sym.}}(\R)\simeq \bigoplus_{K^+\hookrightarrow \R} \mathrm{Herm}_{kn\times kn}(\C),$$ with $$\mathrm{Herm}_{kn\times kn}(\C) := \{X\in M_{kn}(\C) : X^\dag = X\},$$ and $$G(\R)\simeq \prod_{K^+\hookrightarrow \R} \GL_{kn}(\C),$$ with the coordinatewise action given by $(g,x)\mapsto g\cdot x\cdot g^\dag$. So in this case we need a classification of Hermitian forms on $\C^{\oplus n}$, and again the usual Gram-Schmidt argument works to classify Hermitian forms on $\C^{\oplus n}$ by the same data: $(d_0, d_+, d_-)\in \N^{\times 3}$ with $d_0 + d_+ + d_- = n$, with each form in the orbit equivalent to $$(\alpha_i)_{i=1}^n\mapsto \sum_{i=d_0 + 1}^{d_0 + d_+} |\alpha_i|^2 - \sum_{i=d_0 + d_+ + 1}^n |\alpha_i|^2.$$

So again exactly as in Sections \ref{type i subsection}, \ref{type ii subsection}, and \ref{type iii subsection} we find explicit representatives for the $G(\R)$-orbits on $W(\R)$ and a finite-time algorithm to test membership of a given element of $W(\Q)$ in each relevant orbit, all of which are seen to be closed (in the real topology, while again for $G(\C)$-orbits we mean in the Zariski topology) and to intersect $W(\Q)$.

\subsection{Proof of Proposition \ref{the set of integral orbits in a relevant real orbit is computable}.\label{what to do once you have a rational point in a real orbit}}

Now we may prove Proposition \ref{the set of integral orbits in a relevant real orbit is computable} and thus Theorem \ref{the set of polarizations of given multidegree is computable}.

\begin{proof}[Proof of Proposition \ref{the set of integral orbits in a relevant real orbit is computable}.]
Thanks to Section \ref{characterization of the relevant real orbits section} we have an explicit description of the $G(\R)$-orbits in $$G(\R)\cdot \{(\phi, \pm 1, \pm 1)\in W(\R)\},$$ and we moreover may and will assume (via e.g.\ brute-force) that inside each such $G(\R)$-orbit we have chosen a representative lying in $W(\Q)$.

So let $v\in W(\Q)$ be said representative. It suffices then to explain how to compute representatives for the $G_\Z$-orbits in $$(G(\R)\cdot v)\cap W_\Z.$$ Let $q\in \Z^+$ be such that $q\cdot v\in W_\Z$. Let $\Gamma := \frac{1}{q}\cdot W_\Z$. Then certainly $\Gamma\subseteq W(\R)$ is a lattice invariant under $G_\Z$. Let $X := G(\R)\cdot v$ --- note that in Section \ref{characterization of the relevant real orbits section} we saw that $X\subseteq V(\R)$ is Zariski closed. But now we are done by Lemma \ref{computation of reps for integral orbits in a real orbit}, because a $G_\Z$-orbit intersects $W_\Z$ if and only if it lies entirely inside $W_\Z$.
\end{proof}

%% file: borel_harish_chandra.tex
\subsection{Making algorithmic a result of Borel and Harish-Chandra.}
\label{bhc_section}

In this section we prove Lemma \ref{computation of reps for integral orbits in a real orbit}, which shows how to find all integral orbits in a given real orbit of an algebraic group given a single rational point in the real orbit.  
This is a matter of effectivizing \cite[Theorem 6.9]{BHC}, using work of Grunewald and Segal \cite{grunewald-segal}.

We begin by recalling some notation.

\begin{defn}
\label{siegel_set}
Fix rational numbers $t > \frac{2}{\sqrt{3}}$ and $u > 1/2$.
The \emph{standard Siegel set} is the set
$\mathcal{S} \subseteq \operatorname{GL}_n(\mathbb{R})$ defined by
\[ \mathcal{S} = \mathcal{S}_{t, u} = K A_t N_u, \]
where $K = O_n(\mathbb{R})$, 
$A_t$ is the set of diagonal matrices with diagonal entries $a_i$ satisfying $a_i \leq t a_{i+1}$, 
and $N_u$ is the set of unipotent upper triangular matrices, with off-diagonal entries $x_{ij}$ satisfying $\left | x_{ij} \right| \leq u$.

(The definition of $\mathcal{S}$ depends on $t$ and $u$, but we will suppress this dependence in our discussion.
The reader is free to choose once and for all some particular values $t$ and $u$.)
\end{defn}

For convenience we record a simple lemma involving bounding products of matrices.
\begin{lem}
\label{combine_m}
For every $M>0$, let $$\Omega_M := \{ g = (g_{ij}) \in \GL_n(\R) \mid \left | g_{ij} \right | < M \text{ and } \left | \det g \right | > 1/M  \}.$$
Then: $\Omega_{M_1} \Omega_{M_2} \subseteq \Omega_{n M_1 M_2}$ and $\Omega_{M_1}^{-1} \subseteq \Omega_{n!\cdot M^n}$.
\end{lem}
\begin{proof}
Evident.
\end{proof}

\begin{lem}
\label{eff_5_3}
Let $G = \operatorname{GL}_n$, and
let $\pi \colon G \rightarrow \operatorname{GL}(V)$ be a representation of the algebraic group $G$, defined over $\mathbb{Q}$.

Let $\theta \colon g \mapsto (g^{\intercal})^{-1}$ be the standard Cartan involution (transpose inverse) on $G$.
Let $v \in V(\mathbb{R})$ be a point whose orbit under $G$ is closed
and whose isotropy group $G_v$ is stable under $\theta$.
Let $\mathcal{S}$ be a standard Siegel set (Definition \ref{siegel_set}) in $G$.
Let $\Gamma \subseteq V(\mathbb{Q})$ be a lattice.

Then one can compute algorithmically the finite set $v \cdot \pi(\mathcal{S}) \cap \Gamma$.
\end{lem}

\begin{rmk}
This is just a matter of going carefully through the proof of \cite[Lemma 5.4]{BHC} in the particular case of interest to us, i.e.\ $G = \operatorname{GL}_n$, and making every step effective.

(In fact, the particular case $G = \operatorname{GL}_n$ is the only situation in which Borel and Harish-Chandra use the result.)
\end{rmk}

\begin{proof}
Let $A$ be the torus of diagonal matrices in $G = \operatorname{GL}_n$.
The vector space $V$ splits (over $\mathbb{Q}$)
as a direct sum of eigenspaces $V_i$ for the action of $A$,
and we can explicitly compute the eigenspaces $V_i$ by linear algebra.
Choose a Euclidean norm $v \mapsto \left | v \right |$ on $V$ for which the $V_i$'s are mutually orthogonal.

For each eigenspace $V_i$, let $E_i \colon V \rightarrow V_i$ denote the projection of $V$ onto $V_i$.
Since $V$, $V_i$ and $E_i$ are all defined over $\mathbb{Q}$, the projection $E_i(\Gamma)$ is again a lattice in $V_i$, which can be computed effectively.
Hence, we can compute a constant $c$ such that if $w \in \Gamma$ and $E_i(w) \neq 0$, then $\left | E_i(w) \right | \geq c$.

Now consider variable $k, a, n$, with
$k \in K = \operatorname{O}(n)$,
$a \in A_t$,
$n \in N_u$,
so that
\[ x = kan \in \mathcal{S} \]
is an arbitrary element of $\mathcal{S}$.
(Here $A_t$ and $N_u$ are as in Definition \ref{siegel_set}.)

Let $w = v \cdot x$, and define
\begin{eqnarray*}
y & := & x a^{-1} = kana^{-1} \\
z & := & x a^{-2} = kana^{-2}.
\end{eqnarray*}

Note that $a n a^{-1}$ is a unipotent upper-triangular matrix whose off-diagonal entries are bounded in absolute value by $t^{n-1} u$.
Since $K$ is orthogonal, its entries are bounded as well (by $1$), 
so we can compute a bound on the absolute value of the matrix entries of $y$.
(In fact, by matrix multiplication, the entries of $y$ are clearly bounded by $n \operatorname{max}(1, t^{n-1} u)$.)
Since the action of $G$ on $V$ is explicitly presented,
we can compute a bound on $\left | v \cdot y \right | = \left | w \cdot a^{-1} \right |$, independent of $x$.  Write the bound as $\left | v \cdot y \right | \leq c'$.

It follows (see \cite{BHC}) that $\left | E_i(v \cdot z) \right | \leq (c')^2 / c$,
for all choices of $k, a, n$, and for all $i$.  Therefore, we can effectively compute a bound $\left | v \cdot z \right | \leq c''$.

Let $Q$ be the closed ball of radius $c''$ about the origin in $V$.
We want to produce a compact subset $\Omega \subseteq G$ such that 
\[ ( v \cdot \pi(G) ) \cap Q \subseteq v \cdot \pi(\Omega). \]
The existence of $\Omega$ is given by \cite[Lemma 5.2]{BHC}, which is not effective.

To produce $\Omega$, we will apply a theorem of Tarski (see \cite[\S 1.2]{grunewald-segal}).
We will consider only $\Omega \subseteq G$ of the form
\[ \Omega_M = \{ g = (g_{ij}) \in G \mid \left | g_{ij} \right | < M \text{ and } \left | \det g \right | > 1/M  \}. \]
Since every compact $\Omega \in G$ is contained in some $\Omega_M$, we know that $\Omega_M$ has the desired property for all sufficiently large $M$.
But now we have reduced to a standard quantifier elimination problem!  We simply need to find some $M$ for which the following statement holds:
\begin{quote}
For every $v' \in V$ and $g \in G$ such that $\left | v' \right | \leq c''$ and $v \cdot g = v'$, there exists $g' \in \Omega_M$ such that $v \cdot g' = v'$.
\end{quote}
Such an $M$ must exist by \cite[Proposition 5.2]{BHC}, so Tarski's algorithm (in the form of \cite[Algorithm 1.2.1]{grunewald-segal}) will find one.
(Literally, \cite[Algorithm 1.2.1]{grunewald-segal} only finds $M$ approximately: that is, it will give some interval $(M_{\mathrm{approx}}-\epsilon, M_{\mathrm{approx}}+\epsilon)$ which is guaranteed to contain an $M$ that works.
But if $M$ has the desired property, any $M' > M$ does as well, so in particular $M' = M_{\mathrm{approx}}+\epsilon$ will do the job.)

Returning to our effectivization of \cite[Lemma 5.3]{BHC}, we see that for any $k, a, n$ as above, the vector $v \cdot z$ satisfies
$v \cdot z \in v \cdot \Omega_M$,
so
$z \in G_v \Omega_M$.
In other words,
\[ k a^{-1} a^{2} n a^{-2} \in G_v \Omega_M. \]

We can compute an explicit bound on the coeffcients of $a^{2} n a^{-2}$ (and of course this matrix has determinant $1$), whence we obtain $M_1$ such that
\[ a^2 n a^{-2} \in \Omega_{M_1}. \]

Thus, we can compute (by Lemma \ref{combine_m}) $M_2$ such that
\[ k a^{-1} = k a^{-1} (a^{2} n a^{-2}) (a^{2} n a^{-2})^{-1} \in G_v \Omega_M \Omega_{M_1} \subseteq G_v \Omega_{M_2}. \]
Applying $\theta$, we note that $\theta(k) = k$, $\theta(a^{-1}) = a$, $G_v$ is invariant under $\theta$, and one can compute $M_3$ such that $\theta(\Omega_{M_2}) \subseteq \Omega_{M_3}$.  Hence, we obtain
\[ k a \in G_v \Omega_{M_3}. \]
Finally, the bounds on the coefficients of $n$ give an $M_4$ such that $n \in \Omega_{M_4}$, so we can find $M_5$ such that
\[ k a n \in G_v \Omega_{M_5}. \]

Thus $w = v \cdot x = v \cdot k a n \in v \cdot \Omega_{M_5}$,
so one can compute effective bounds on $\left | w \right |$.
This gives a computable finite list of possible vectors $w$, which is guaranteed to contain the finite set $v \cdot \pi(\mathcal{S}) \cap \Gamma$.

Finally, we test each vector $w$ on our list, 
using Tarski's algorithm \cite[Algorithm 1.2.1]{grunewald-segal}
to determine whether it lies in $\pi(\mathcal{S})$.
\end{proof}

\begin{lem}\label{computation of reps for integral orbits in a real orbit}
Let $G$ be a connected reductive group defined over $\mathbb{Q}$, 
$\pi \colon G \rightarrow \operatorname{GL}(V)$ a representation defined over $\mathbb{Q}$,
$\Gamma$ a $\mathbb{Z}$-lattice in $\mathbb{Q}$-vector space $V$ invariant under $G_{\mathbb{Z}}$, 
and $X$ a closed orbit of $G_{\mathbb{R}}$ on $V_{\mathbb{R}}$.
Suppose $v_0 \in \Gamma \cap X$ is a given point.

There is an algorithm that takes as input $G, \pi, \Gamma, v_0$, 
and returns a list of points 
\[ v_0, v_1, \ldots, v_n \in \Gamma \cap X \]
such that 
\[ \Gamma \cap X = \bigcup_{i=0}^n G_{\mathbb{Z}}\cdot v_i. \]
\end{lem}

\begin{rmk}
This is essentially just a matter of making effective the proof of \cite[6.9]{borel-harishchandra}.
That proof has one essential ineffectivity:
it does not say how to determine whether $\Gamma \cap X$ is empty.
We do not know how to solve this problem algorithmically, in general;
we solve it explicitly in the cases we need (Section \ref{characterization of the relevant real orbits section}).

In fact, \cite[6.5 (i)]{borel-harishchandra} is made algorithmic by \cite[5.3.1]{grunewald-segal},
so we need only explain how to reduce \cite[6.9]{borel-harishchandra} to \cite[6.5 (i)]{borel-harishchandra}.
\end{rmk}

\begin{proof}
$G$ is given as a subgroup, defined over $\mathbb{Q}$, of some $G' = GL_n$.
Let $H \subseteq G$ be the stabilizer of $v_0$.
By \cite[Algo.\ 5.1.1]{grunewald-segal}, 
we can find a rational representation
\[ \rho \colon G' \rightarrow GL(W), \]
and a vector $w \in W$, such that
\begin{enumerate}
    \item $H$ is the stabilizer of $w$ in $W$, and
    \item $w \cdot \rho(G')$ is closed in $W$.
\end{enumerate}

Let $X' = w \cdot \rho(G)$ be the orbit of $w$ under the subgroup $G \subseteq \operatorname{GL}_n$.
There is a unique $G$-equivariant isomorphism of $\mathbb{Q}$-schemes $\phi \colon X \rightarrow X'$ such that $\phi(v_0) = w$ (see the proof of \cite[6.9]{borel-harishchandra}).
In fact, since $X$ and $X'$ are both affine schemes, the coordinates of $\phi$ can be written explicitly as polynomials.
Choosing coordinates $x_i$ on $V$ and $y_i$ on $W$,
the map $\phi$ is given by polynomials $(P_1, \ldots, P_s)$ with rational coefficients such that $y_i(\phi(x_1, \ldots, x_r)) = P_i(x_1, \ldots, x_r)$.
The inverse morphism $\phi^{-1}$ is given by polynomials over $\mathbb{Q}$ as well.
The pair $(\phi, \phi^{-1})$ can be computed by brute-force search.

Let $\Gamma' \subseteq W$ be a $\mathbb{Z}$-lattice, containing $\phi(\Gamma)$ and invariant under the action of the group $G'_{\mathbb{Z}} = \operatorname{SL}_n(\mathbb{Z})$. 

(The next step effectivizes \cite[Lemma 6.8]{borel-harishchandra}.)
First, find (by brute-force search) an $a \in \operatorname{SL}_n(\mathbb{R} \cap \Qbar)$ 
such that $a G a^{-1}$ and $a H a^{-1}$ are self-adjoint
as subgroups of $\operatorname{SL}_n$.
Such an $a$ is guaranteed to exist: \cite[Lemma 1.9]{borel-harishchandra} gives the existence of an $a \in \operatorname{SL}_n(\mathbb{R})$ such that $a G a^{-1}$ and $a H a^{-1}$ are self-adjoint.
But the condition that $a G a^{-1}$ and $a H a^{-1}$ be self-adjoint
is an algebraic condition on $a$, 
so in fact we can take $a$ to have coefficients algebraic numbers.

By \cite[Algorithm 5.3.1]{grunewald-segal}, 
we can find finitely many elements $b_i \in \operatorname{GL}_n(\mathbb{Z})$ such that
\[ G(\mathbb{R}) = \bigcup_i \left ( G(\mathbb{R}) \cap a^{-1} \mathcal{S} b_i \right ) \dot G(\mathbb{Z}), \]
where $\mathcal{S}$ is the standard Siegel set
(Definition \ref{siegel_set}).

Now apply Lemma \ref{eff_5_3} (to the group $\operatorname{GL}_n$, the representation $W$, the lattice $\Gamma'$, and the vector $w \cdot a^{-1}$
to determine the finite set $w \cdot a^{-1} \mathcal{S} \cap \Gamma$.
(By the choice of $a$, the stabilizer $a H a^{-1}$ of $w \cdot a^{-1}$ is self-adjoint
under the standard Cartan involution on $\operatorname{GL}_n$; 
the other hypotheses of Lemma \ref{eff_5_3} are immediate.)

For each $i$, apply $b_i$ to $w \cdot a^{-1} \mathcal{S} \cap \Gamma$ to obtain the finite set $w \cdot a^{-1} \mathcal{S} b_i \cap \Gamma$.
Compute the union $F_0$ of these finite sets, over all the $b_i$ computed earlier.
Since $G(\mathbb{R}) = \bigcup_i \left ( G(\mathbb{R}) \cap a^{-1} \mathcal{S} b_i \right ) \dot G(\mathbb{Z})$, 
the set $F_0$ contains a representative for each orbit of $G(\mathbb{Z})$ on $X'$.
Test each point of $F_0$ to determine whether it lies in $X'$; let $F_1 = F_0 \cap X'$.

Finally, apply the isomorphism $\phi^{-1} \colon X' \rightarrow X$ to each point in $F_1$,
and test the resulting points to determine which lie in the lattice $\Gamma$.
Let $F$ be the set of point of $\phi^{-1}(F_1)$ lying in $\Gamma$; this $F$ is the desired set.
\end{proof}

%% file: abelian.tex
\section{The homology of an abelian variety.}
\label{sec:int_homology}

In this section we show how to explicitly compute the singular homology of an abelian variety over $\mathbb{C}$, which enables us to perform computations on endomorphisms and polarizations in exact integer coordinates.

Lemma \ref{homology_action} gives the action of an endomorphism on singular homology; we use it in the computation of the endomorphism ring.
Lemma \ref{find_polarization_in_H2} enables us to write the Chern class of a polarization in exact integer coordinates.

\begin{lem}
\label{basis_differentials}
There is an algorithm that takes as input an abelian variety $A$ over a number field $K$, and returns a $K$-basis for $\Gamma(A, \Omega^1_A)$.
\end{lem}

\begin{proof}
Let $e$ be the origin on $A$.
Let $u_1, \ldots, u_g$ for a basis for $(T_e A)^{\vee}$, the cotangent space to $A$ at $e$.

For each $1 \leq i \leq g$, compute $\omega_i$ by ``translating $u_i$ around $A$''.  
Concretely, let
\[\Delta = \{(x, -x) \mid x \in A\} \subseteq A \times A\] 
be the antidiagonal,
and let $m \colon A \times A \rightarrow A$
be the multiplication map, so that $\Delta = m^{-1}(e)$,
and $m^* u_i \in T(A\times A)^{\vee} | _{\Delta}$
is a differential vanishing along $\Delta$.

With the natural identification 
\[ T(A \times A)^{\vee} = \pi_1^* T_A \oplus \pi_2^* T_A, \] 
let $\iota_1$ denote the projection onto the first factor:
\[ \iota_1( \pi_1^* \omega_1 + \pi_2^* \omega_2 ) = \pi_1^* \omega_1. \]
Finally, let $\delta \colon A \rightarrow \Delta$ be the inclusion 
\[ \delta(x) = (x, -x). \]

Then (for $1 \leq i \leq g$)
\[ v_i = \delta^* \iota_1( m^* u_i ) \]
is a translation-invariant section of $\Omega_1(A)$
restricting to $u_1$ at $e$;
the sections $v_1, \ldots, v_g$ form the desired basis of $\Gamma(A, \Omega_1(A))$.
\end{proof}

\begin{notation}
\label{unif_exp}
Let $A$ be an abelian variety of dimension $g$ over a number field $K$, and suppose given some complex embedding of $K$.  Suppose $\omega_1, \ldots, \omega_g$ form a basis for $\Gamma(A, \Omega^1_A)$.

Then there is a unique complex-analytic map $\mathrm{exp} \colon \mathbb{C}^g \rightarrow A$
such that $\mathrm{exp}^* \omega_i = d z_i$ for each $i$.

(In other words: $A$ is uniformized by $\mathbb{C}^g$; 
the uniformizing map $\mathrm{exp} \colon \mathbb{C}^g \rightarrow A$ is well-defined up to linear automorphisms of $\mathbb{C}^g$, 
and we use the basis $\omega_i$ of $\Gamma(A, \Omega^1_A)$
to specify the map $\mathrm{exp}$.)

We call this $\mathrm{exp}$ the \emph{uniformization of $A$ corresponding to the basis $\omega_i$}.
\end{notation}

\begin{lem}
\label{tao_ift0}
(Effective inverse function theorem.)

Suppose $U \subseteq \mathbb{C}^n$ is a neighborhood of a point $e$, 
and $f \colon U \rightarrow \mathbb{C}^n$ is a holomorphic function.

Suppose the differential $df_e$ of $f$ at $e$ is invertible, 
and suppose $\epsilon$ and $c < 1$ are chosen such that,
for all $x \in B_\eps(e)$, 
\[ \| df_e^{-1} (df_x - df_e) \| \leq c. \]

Then:
\begin{itemize}
\item $f$ is injective on $B_\eps(e)$.
\item 
$f(e) + df_e ( B_{(1-c) \eps}(0))) \subseteq f( B_\eps(e) ) \subseteq
f(e) + df_e ( B_{(1+c) \eps}(0)))$.
\item 
The inverse function of $f$ can be computed effectively: for any $y \in f(e) + df_e (B_{(1-c) \eps}(0)))$, 
one can compute to arbitrary precision 
its unique inverse image $f^{-1}(y) \in B_\eps(e)$.
\end{itemize}
\end{lem}

\begin{proof}
This is a standard application of the contraction mapping principle 
(see for example \cite[Lemma 6.6.6]{Tao}).
To compute $f^{-1}(y)$, simply iterate the contraction mapping
\[ x \mapsto x + df_e^{-1} (y - f(x)). \]
\end{proof}

\begin{lem}
\label{tao_ift}
Let $X \subseteq \mathbb{A}^N$ be an algebraic variety 
of dimension $g$ over a number field $K$ (with fixed complex embedding)
and let $e$ be a smooth point of $X$.
Equip $\mathbb{A}^N(\mathbb{C}) = \mathbb{C}^N$ with the standard Euclidean metric. 

There is an algorithm to find some $g$ of the $N$ coordinates on $\mathbb{A}^N$ -- 
say $x_{i_1}, x_{i_2}, \ldots, x_{i_g}$ -- 
so that the resulting map $\pi \colon X \rightarrow \mathbb{A}^g$
is invertible in a neighborhood of $e$ (for the classical topology on $X^{an}$).

Furthermore, there is an algorithm to find an $\epsilon$
such that, if $B_\eps(e)$ is the ball of radius $\epsilon$ about $e$, then
\begin{itemize}
    \item the restriction of $\pi$ to $B_\eps(e) \cap X$ is invertible, and
    \item $B_\eps(e) \cap X$ is contained in a contractible subset of $X$.
\end{itemize}

Finally, there is an algorithm to find an $\epsilon_1$ such that $B_{\epsilon_1}(\pi(e)) \subseteq \pi(B_\eps(e))$.  The inverse function to $\pi$ can be computed algorithmically on $B_{\epsilon_1}(\pi(e))$ (to arbitrary precision).
\end{lem}

\begin{proof}
The algebraic variety $X$ comes presented as the vanishing locus of a collection of polynomials in $N$ variables.
That $X$ is smooth at $e$ implies that there are $N-g$ of these polynomials, say $f_1, \ldots, f_{N-g}$, 
whose differentials are linearly independent in $(T_e \mathbb{A}^N)^{\vee}$;
these polynomials can be found (among the finitely many defining polynomials for $X$)
by standard methods of linear algebra.

One can then choose indices $i_1, \ldots, i_g$ so that $df_1, \ldots, df_{N-g}, dx_{i_1}, \ldots, dx_{i_g}$ form a basis for $(T_e \mathbb{A}^N)^{\vee}$.

For dimension reasons, $X$ locally coincides which the vanishing locus $Z = Z(f_1, \ldots, f_{N-g})$; 
in other words, there is some open set $U \ni e$ such that 
\[ X \cap U = Z \cap U. \]

We will apply an effective form of the inverse function theorem 
to the map $F \colon \mathbb{C}^n \rightarrow \mathbb{C}^n$
given by
\[ F(x_1, \ldots, x_n) = (f_1(x_1, \ldots, x_n), \ldots, f_{N-g}(x_1, \ldots, x_n), x_{i_1}, \ldots, x_{i_g}). \]
We have arranged that $dF_e \colon \mathbb{C}^N \rightarrow \mathbb{C}^N$ is invertible.

Let $Y \subseteq \mathbb{C}^N$ be the set of points whose first $N-g$ coordinates vanish, so we have the relations
\[ X \subseteq Z = F^{-1}(Y). \]

Equip $\mathbb{C}^N$ with the standard Euclidean metric.
For variable $x \in \mathbb{C}^n$, consider the linear transformation
\[ G(x) = dF_e^{-1} \circ (dF_x - dF_e). \]
We have $G(e) = 0$.
Since $G$ is continuous, we can find $\epsilon$ such that,
whenever $\left | x-e  \right | < \epsilon$, all matrix entries of $G(x)$ are less than $1/2g^2$.  (In fact the matrix entries of $G$ are polynomials in $x$, and so one can find such $\epsilon$ algorithmically.)

Now by Lemma \ref{tao_ift0}
$F$ is invertible on $B_\eps(e)$.
Hence, $F$ maps $B_\eps(e)$ invertibly to some neighborhood of $F(e)$.
By Lemma \ref{tao_ift0} again,
\[ B_{\eps/4}(e) \subseteq F^{-1}B' \subseteq B_\eps(e), \]
where $B'$ is the image of $B_{\eps/2}(e)$ under the affine-linear map
\[ x \mapsto F(e) + dF_e(x-e). \]

Now $B'$ is an ellipsoid, so its intersection with any coordinate plane is also an ellipsoid.  In particular, $Y \cap B'$ is simply connected.  
Since $F$ is a homeomorphism on $F^{-1}(B')$, 
the set $$F^{-1}(Y \cap B') = Z \cap F^{-1}(B')$$ is simply connected as well.
In particular, $Z \cap F^{-1}(B')$ is contained in a single irreducible component of $Z$, so $Z \cap F^{-1}(B') \subseteq X$.

It now follows that $B_{\epsilon/4}(e)$ satisfies the conditions of the lemma: the restriction of $\pi$ to $B_{\epsilon/4}(e) \cap X$ is invertible, and $B_{\epsilon/4}(e) \cap X$ is contained in a contractible subset of $X$.

Let $\frac{1}{4} B'$ be the image of $B_{\eps/8}(e)$ under $x \mapsto F(e) + dF_e(x-e)$.
By Lemma \ref{tao_ift0} yet again, 
\[ \frac{1}{4} B' \subseteq \pi(B_{\epsilon/4}(e)). \]
Thus if we take $\epsilon_1$ small enough that $B_{\epsilon_1}(\pi(e)) \subseteq \frac{1}{4} B'$, we have $B_{\epsilon_1}(\pi(e)) \subseteq \pi(B_\eps(e))$.

Computability of the inverse function follows from Lemma \ref{tao_ift0}.
\end{proof}

\begin{lem}
\label{unif_ball}
Let $A$ be an abelian variety over a number field with chosen complex embedding.  Suppose a basis for $\Gamma(A, \Omega^1_A)$ has been chosen as in Lemma \ref{basis_differentials}, and let $\mathrm{exp} \colon \mathbb{C}^g \rightarrow A$ be the uniformizing map defined in Notation \ref{unif_exp}.

One can explicitly compute an open ball $B = B_{\eps}(0)$ about the origin in $\mathbb{C}^g$
and an open neighborhood $B'$ of $e$ in $A$
such that $\mathrm{exp}(B) \subseteq B'$, and $B'$ is contained in a simply connected subset of $A$.
Furthermore, the map $\mathrm{exp}$ can be computed algorithmically to arbitrary precision on $B$.
\end{lem}

\begin{proof}

Apply Lemma \ref{tao_ift} to the smooth point $e$ of the smooth variety $A$;
Lemma \ref{tao_ift} returns
a projection
\[\pi \colon U \rightarrow \mathbb{A}^g\]
from some open $U \subseteq A$
onto some $g$ of the $N$ standard coordinates on an affine coordinate patch $\mathbb{A}^N \subseteq \mathbb{P}^N$,
as well as an algorithmically-computable $\epsilon_0$
such that $B' = B_{\eps_0}(e) \cap A$ is contained in some simply-connected subset $\Omega$ of $A$.

Since $\mathrm{exp} \colon \mathbb{C}^g \rightarrow A$ is the universal cover of $A$,
we can invert $\mathrm{exp}$ canonically
on $\Omega$ by requiring that $\mathrm{exp}^{-1}(e) = 0$;
this gives a well-defined inverse $\mathrm{exp}^{-1}$ on $B'$,
which we will use without further comment.

\[
\xymatrix{
\mathrm{exp}^{-1}(B') \ar[r]^{\mathrm{exp}} \ar@{^{(}->}[d] & 
B' \ar[r]^{\gamma} \ar@{^{(}->}[d] & 
\mathbb{C}^g \\ \mathbb{C}^g \ar[r] 
& A & 
}
\]

Now we want to apply Lemma \ref{tao_ift0} to the map
\[ \mathrm{exp}^{-1} \circ \gamma^{-1} \colon \gamma(B') \rightarrow \mathbb{C}^g. \]
To estimate $d(\mathrm{exp}^{-1} \circ \gamma^{-1})$, we will work with differentials on $A$.
At any $a \in B \subseteq A$, 
we can express each $\omega_i$ 
as a linear combination of the differentials $\gamma^* dt_i$.
Specifically, one can compute (exactly) rational functions $M_{ij}$ on $A$ such that
\[ \omega_i = \sum_i M_{ij} \gamma^* t_j. \]

But now these $M_{ij}$ exactly give the differential of $\mathrm{exp}^{-1} \circ \gamma^{-1}$!
Specifically, if $M(x)$ denotes the matrix whose $i,j$ entry is $M_{ij}(x)$, we have
\[ d(\mathrm{exp}^{-1} \circ \gamma^{-1})_a = M(\gamma^{-1}(a)). \]

Now apply Lemma \ref{tao_ift0}.
We can find $\epsilon$ such that,
if $x \in \gamma^{-1} C$, then
\[ \| M(e)^{-1} (M(x) - M(e)) \| < 1/2. \]
And we are done!
Specifically: Lemma \ref{tao_ift0} tells us that $M(e)^{-1} B_{\epsilon/2}(0) \subseteq \mathrm{exp}^{-1} \circ \gamma^{-1} (\mathrm{exp}^{-1} \circ \gamma^{-1})$, so any ball $B$ contained in $M(e)^{-1} B_{\epsilon/2}(0)$ does the job.

Finally, note that both $\gamma \circ \mathrm{exp}$ and $\gamma^{-1}$ can be computed numerically by Lemma \ref{tao_ift}, so their composition $\mathrm{exp}$ can as well.
\end{proof}

\begin{lem}
\label{inv_exp}
There is an algorithm to compute the inverse of the uniformizing map $\mathrm{exp}$
on the open set $B' \subseteq A$ computed in Lemma \ref{unif_ball}. 
Specifically, the algorithm takes as input $x \in B' \subseteq A$, and returns an estimate to arbitrary precision
of the unique inverse image $\mathrm{exp}^{-1}(x) \in B \subseteq \mathbb{C}^g$.
\end{lem}

\begin{proof}
Numerical integration.

We have 
\[ \mathrm{exp}^{-1}(x) = \int_e^x (\omega_1, \ldots, \omega_g); \]
the path from $e$ to $x$ is uniquely determined because $B'$ is simply connected.
\end{proof}

\begin{lem}
\label{comp_exp}
(compute $\mathrm{exp}$ of a point)

There is an algorithm to compute $\mathrm{exp}(x) \in A$ to arbitrary precision, for any $x \in \mathbb{C}^g$.
\end{lem}

\begin{proof}
This follows from Lemma \ref{unif_ball} and the fact that $\mathrm{exp}$ is a group homomorphism.

Find some integer $n$ such that $\frac{1}{n} x \in B$, and compute
\[ \mathrm{exp}(x) = [n] \mathrm{exp}\left (\frac{1}{n} x\right ), \]
where $[n] \colon A \rightarrow A$ is multiplication by $n$.
\end{proof}

\begin{lem}
\label{lat_ball}
Equip $\mathbb{R}^n$ with the standard Euclidean metric, and let $\Lambda \subseteq \mathbb{R}^n$ be a lattice.  Suppose $r$ is large enough that $B_r(0) \cap \Lambda$ contains $n$ linearly independent vectors.  Then:
\begin{enumerate}
\item Let $a_n = \frac{\sqrt{n}}{2}$.  For any $w \in \mathbb{R}^n$, the intersection $B_{a_n}(w) \cap \Lambda$ is nonempty.
\item
Let $b_n = \operatorname{max}(1, a_n)$.  Then
$B_{b_n r}(0) \cap \Lambda$ 
contains a set of vectors that generates $\Lambda$.
\end{enumerate}
\end{lem}

\begin{proof}
Induction on $n$.  (The base case, either $n=0$ or $n=1$, is trivial.)

Let $v_1, \ldots, v_n$ be the $n$ independent vectors in $B_r(0) \cap \Lambda$.  For convenience, perform an orthogonal change of coordinates on $\mathbb{R}^n$ so that $v_1, \ldots, v_{n-1}$ have $x_n = 0$. The projection of $\Lambda$ onto the $x_n$-axis is a one-dimensional lattice, generated by, say, $z > 0$.  Note that $\left | z \right | < r$.

We begin with (1).  Translating $w$ by some element of $\Lambda$, we may assume that $w$ has $x_n$-coordinate $\left | x_n(v) \right | \leq z/2 < r/2$.  Apply the inductive hypothesis to $w'$, the (orthogonal) projection of $w$ onto the hyperplane $x_n = 0$; this produces some $\lambda \in \Lambda$ with
\[ \left | \lambda - w' \right | < r \sqrt{n-1} / 2, \]
whence
\[ \left | \lambda - w \right | < r \sqrt{n} / 2. \]
This proves (1).

We now turn to (2).
Let $\Lambda_0$ be the intersection of $\Lambda$ with the hyperplane $x_n = 0$.  By the inductive hypothesis, $\Lambda_0$ is generated by $B_{nr}(0) \cap \Lambda_0$.

The projection of $\Lambda$ onto the $x_n$-axis is a one-dimensional lattice, generated by, say, $z > 0$.  Suppose $v_n$ has $x_n$-coordinate $kz$, for some nonzero $k \in \mathbb{Z}$.  We need to show that $B_{b_n r}(0) \cap \Lambda$ contains some vector with $x_n$-coordinate $z$.

If $k = \pm 1$ then $v_n$ is the desired vector and there is nothing to prove.  Hence we may assume that $\left | k \right | \geq 2$, which implies that $\left | z \right | < r/2$.
Now the intersection of $\Lambda$ with the hyperplane $x_n = z$ is a torsor for $\Lambda$.  By part (1), this torsor contains some point $\lambda$ within $a_{n-1} r$ of the point $(0, 0, \ldots, 0, z)$.  Then this $\lambda$ has length at most $b_n r$, and the proof is complete.
\end{proof}

\begin{lem}
\label{homology_basis}
One can explicitly compute, to any desired precision, a basis for $H_1(A, \mathbb{Z})$ as a lattice in the uniformizing space $\mathbb{C}^g$.
\end{lem}

\begin{proof}
We can identify $H_1(A, \mathbb{Z})$ with the kernel of $\mathrm{exp}$.  The idea is simply to evaluate $\mathrm{exp}$ on points in a sufficiently fine mesh.

Let $B \subseteq \mathbb{C}^g$ and $B' \subseteq A$ be as in Lemma \ref{unif_ball}.
That is, $B \subseteq \mathbb{C}^g$ is a ball on which $\mathrm{exp}$ is invertible, and the inverse is explicitly computable (Lemmas \ref{unif_ball} and \ref{inv_exp}).  

Let $a_n$ and $b_n$ be the (explicit) constants in Lemma \ref{lat_ball}.
Suppose $B$ has radius $\epsilon$.  Consider the lattice
\[ \Lambda = \frac{\epsilon}{a_n} (\mathbb{Z} + i \mathbb{Z})^g \subseteq \mathbb{C}^g. \]
For any $x \in H_1(A, \mathbb{Z})$, the set $B_{\epsilon}(x) \cap \Lambda$ is nonempty by Lemma \ref{lat_ball}.  That is, there exists $\lambda$ ``near $x$'', in the lattice $\Lambda$, for which $\mathrm{exp}(\lambda) \in B'$.  Conversely, if $\mathrm{exp}(\lambda) \in B'$, by Lemma \ref{inv_exp}, we can compute (to arbitrary precision some $x$ ``near $\lambda$'' such that $x \in H_1(A, \mathbb{Z})$.

To finish, we simply perform a brute-force search over points $\lambda \in \Lambda$ (sorted by length).  For each $\lambda \in \Lambda$, we compute $\mathrm{exp}(\lambda)$; if $\mathrm{exp}(\lambda) \in B'$, we compute $x$ near $\lambda$ with $x \in H_1(A, \mathbb{Z})$.  Eventually, this procedure will generate a set of $g$ linearly-independent vectors $x$; let $r$ be the maximum of their lengths.  

These $g$ vectors will span $H_1(A, \mathbb{Z})$ as a $\mathbb{Q}$-vector space, but maybe not as a $\mathbb{Z}$-lattice.  To find generators as a $\mathbb{Z}$-lattice, search all vectors $\lambda \in \Lambda$ of length $\leq b_n r + \epsilon$ to find all $x \in H_1(A, \mathbb{Z})$ of length at most $b_n r$.  By Lemma \ref{lat_ball}, these vectors $x$ will generate $H_1(A, \mathbb{Z})$.
\end{proof}

\begin{lem}
\label{homology_action}
Let $f: A \rightarrow B$ be abelian varieties over a number field $K$, and suppose given some complex embedding of $K$.  Then one can explicitly compute the action of $f$ on homology, as a map $H_1(A, \mathbb{Z}) \rightarrow H_1(B, \mathbb{Z})$, in exact integers, in terms of a basis given by Lemma \ref{homology_basis}.
\end{lem}

\begin{proof}
Compute $f^* : \Gamma(B, \Omega^1(B)) \rightarrow \Gamma(A, \Omega^1(A))$ in exact coordinates over the number field $K$.  Lemma \ref{homology_basis} gives approximate coordinates (to arbitrary precision) for the isomorphism
\[ H_1(A, \mathbb{Z}) \cong \Gamma(A, \Omega^1(A))^{\vee}, \]
and similarly for $B$ (where $H_1(A, \mathbb{Z})$ is represented in terms of an integral basis, and $\Gamma(A, \Omega^1(A))^{\vee}$ in terms of Kähler differentials defined over $K$).  Composing the three maps, we obtain approximate coordinates for
\[ H_1(f) : H_1(A, \mathbb{Z}) \rightarrow H_1(B, \mathbb{Z}). \]
But since this is a morphism of $\mathbb{Z}$-modules, we can determine the coordinates as exact integers.
\end{proof}

\begin{lem}
\label{find_polarization_in_H2}
    Let $A$ be a projectively embedded abelian variety over a number field $K$ equipped with a choice of complex embedding $K \hookrightarrow \mathbb{C}$.

    One can compute (in terms of exact integers) the Chern class of $\mathcal{O}(1) | A$ as an element of $H^2(A, \mathbb{Z}) = \wedge^2 H_1(A, \mathbb{Z})^{\vee}$ (in terms of a basis for $H_1(A, \mathbb{Z})^{\vee}$ given e.g.\ by Lemma \ref{homology_basis}).
\end{lem}

\begin{proof}
    Numerical integration.

    Let $\omega$ be the $(1, 1)$-form on the ambient projective space given by the Fubini--Study metric. Concretely, let $\mathbb{P}^N$ be described by projective coordinates $[Z_0 : Z_1 : \cdots : Z_N]$, and on the affine patch given by points of the form $(1, z_1, z_2, \ldots, z_n)$, the differential form $\omega$ is given by
\[ \omega = \frac{1}{\pi} \frac{1}{\left (1 + \sum_i \left | z_i \right | ^2 \right )^2} \sum_i dz_i \wedge d\overline{z_i}. \]
(Make the obvious modifications to determine $\omega$ on the other $N$ standard affine patches that cover $\mathbb{P}^N$.)

To determine the class of $\omega$ in $H^2(A)$, it suffices to compute 
\[ \int_{\gamma_1 \times \gamma_2} \omega(x_1 + x_2) \]
as $\gamma_1, \gamma_2$ range over a basis for $H_1(A)$.  (In other words: we have a map $\gamma_1 \times \gamma_2 \rightarrow A$ given by $(x_1, x_2) \mapsto x_1 + x_2$, where addition represents the group law on $A$.  We pull $\omega$ back to the torus $\gamma_1 \times \gamma_2$ and integrate.)  By numerical integration we can determine this integral to arbitrary precision; in particular, since its value is guaranteed to be an integer, we can determine its value exactly.
\end{proof}

%% file: algorithmic_decomposition_of_algebras_over_local_fields.tex
\section{Algorithmic decomposition of $\mathbb{Q}_{\ell}$-algebras.}
\label{sec:alg_dec_ql_alg}

In this and the following section, all algebras and vector spaces considered are finite-dimensional.

\begin{prop}
\label{simple_reps}
There is a finite-time algorithm which, on input $(N, E_\Q, V_\Q, \rho_\Q)$ with $N\in \Z^+$, $E_\Q$ a semisimple algebra over $\Q$, $V_\Q$ a $\Q$-vector space, and $\rho_\Q: E_\Q\to \End_\Q(V_\Q)$, outputs $(s, (\widetilde{e}_i)_{i=1}^s, (n_i)_{i=1}^s, (V_i)_{i=1}^s, )$ with $s\in \N$, $\widetilde{e}_i\in E_\Q$, $n_i\in \Z^+$, and the $V_i$ pairwise distinct simple $E$-modules such that there is an isomorphism
\begin{equation}\label{decomp}
V\cong \bigoplus_{i=1}^s V_i^{\oplus n_i}
\end{equation}\noindent
of $E$-modules, and such that $\widetilde{e}_i\equiv e_i\pmod*{\ell^N}$ with each $e_i$ projection onto the $i$-th summand (i.e.\ $V_i^{\oplus n_i}$), where $E := E_\Q\otimes_\Q \Q_\ell$ and $V := V_\Q\otimes_\Q \Q_\ell$.
\end{prop}

\begin{prop}
\label{local_semisimple_algebras}
There is a finite-time algorithm which, on input $E$ a semisimple algebra over $\mathbb{Q}_{\ell}$, outputs $(s, (\widetilde{e}_i)_{i=1}^s, (n_i)_{i=1}^s, (K_i)_{i=1}^s, (q_i)_{i=1}^s)$, where
\begin{equation} \label{local_decomp}
E \cong \bigoplus_{i=1}^s M_{n_i} (E_i). \end{equation}
Here $E_i$ is the central simple algebra over the field $K_i$ with invariant $q_i$, and $\widetilde{e}_i$ is an approximate idempotent in $E$ projecting onto $M_{n_i}(E_i)$.
\end{prop}

We will begin with a general structure theorem.

\begin{prop}
\label{structure}
Let $E$ be a semisimple algebra over $\mathbb{Q}_{\ell}$.
Then we have a decomposition
$E \cong \bigoplus_i M_{n_i} (E_i)$,
where each $E_i$ is a division algebra over a finite extension of $\mathbb{Q}_{\ell}$.
(The $E_i$'s need not be distinct.)

For each summand $M_{n_i} (E_i)$ in the decomposition,
$E_i^{\oplus n_i}$ is a simple left $E$-module, and all simple left $E$-modules are of this form.

Each of the division algebras $E_i$ has a unique maximal order $\mathcal{O}_{E_i}$.
The maximal orders of $E$ are exactly the conjugates in $E$ of
$\bigoplus_i M_{n_i} (\mathcal{O}_{E_i})$.

The trace pairing
$(x, y) \mapsto \operatorname{Tr}_{E / \mathbb{Q}_{\ell}} (xy)$
defines a perfect $\mathbb{Q}_{\ell}$-bilinear pairing $E \times E \rightarrow \mathbb{Q}_{\ell}$.
\end{prop}

\begin{proof}
Let $Z(E)$ be the center of $E$.  
This $Z(E)$ decomposes as a direct sum of fields.
Let $e_1, \ldots, e_k \in Z(E)$ be elementary idempotents,
so each $e_i Z(E)$ is a field summand of $Z(E)$.
Then the $e_i$'s are orthogonal idempotents in $E$ as well,
so we have the decomposition
$E = \bigoplus_i e_i E$.

Passing to a direct summand, we may assume that $Z(E)$ is a field.
In other words, $E$ is a central simple algebra over the field $Z(E)$.
In this case,
it is well-known (see for example \cite[IX.1, Theorem $1$]{WeilBNT}) that $E$ must be of the form
$M_n(E_0)$, for $E_0$ a division ring over $Z(E)$,
and every simple left $E$-module is isomorphic to $E_0^{\oplus n}$.

Next we turn to the question of maximal orders.
It is well-known (\cite[I.4, Theorem $6$]{WeilBNT}) that each division algebra $E_i$ has a unique maximal order $\mathcal{O}_{E_i}$.
Consider first the situation $E \cong M_n(E_0)$;
we claim that any order $\mathcal{O}$ in $E$ is conjugate in $E$ to a suborder of $M_n(\mathcal{O}_{E_0})$.
Now $E_0^{\oplus n}$ has the structure of simple left $E$-module, whose ring of $E_0$-linear endomorphisms coincides with $E$.
It is enough to check that $E_0^{\oplus n}$ has an $\mathcal{O}$-stable lattice;
we can easily construct such a lattice by taking the $\mathcal{O}$-span of any basis for $E_0^{\oplus n}$.
This proves that every maximal order in $M_n(E_0)$ is conjugate to $M_n(\mathcal{O}_{E_0})$;
the claim for general $E$ follows easily.

Finally, the statement about the trace pairing
follows from the corresponding statement for division algebras, 
which is again standard.
\end{proof}

Next let us show how to algorithmically find maximal orders in $E$.

\begin{lem}
\label{order_step_by_step}
Let $E$ be a semisimple algebra over $\mathbb{Q}_{\ell}$,
and let $\mathcal{O} \subsetneq \mathcal{O}'$ be two orders in $E$.
Then there exists an order $\mathcal{O}''$ such that $\mathcal{O} \subsetneq \mathcal{O}'' \subseteq \mathcal{O}'$
and
$\mathcal{O}'' \subseteq \frac{1}{\ell} \mathcal{O}$.
\end{lem}

\begin{proof}
Let $r \geq 1$ be the smallest integer such that
$\ell^r \mathcal{O}' \subseteq \mathcal{O}$,
and let
$\mathcal{O}'' := \ell^{r-1} \mathcal{O}' + \mathcal{O}$. An order is just a $\mathbb{Z}_{\ell}$-lattice that is closed under multiplication. We only need to check that $\mathcal{O}''$ is closed under multiplication, 
and for this it is enough to note that
$\ell^{r-1} \mathcal{O}' \cdot \mathcal{O} \subseteq \ell^{r-1} \mathcal{O}' \cdot \mathcal{O}' \subseteq \ell^{r-1} \mathcal{O}'$.
\end{proof}

\begin{lem}
\label{maxl_order}
There is a finite-time algorithm to find a maximal order 
in a semisimple $\mathbb{Q}_{\ell}$-algebra $E$, presented as $E := E_\Q\otimes_\Q \Q_\ell$ with $E_\Q$ a semisimple $\Q$-algebra.

The algorithm takes as input a presentation of $E_\Q$ over $\Q$,
and outputs an integral basis for an order $\mathcal{O}_0\subseteq E_\Q$ such that $\mathcal{O}_0\otimes_\Z \Z_\ell\subseteq E$ is maximal.
\end{lem}

\begin{proof}
We begin by finding some order, 
and then enlarge it to a maximal order using Lemma \ref{order_step_by_step}.

First, choose some basis $e_1, \ldots, e_n$ for $E_\Q$ over $\Q$.
Write the multiplication rule as
$e_i e_j =: \sum_k \alpha_{i, j, k} e_k$.
If all the $\alpha$'s are $\ell$-integral, then $e_1, \ldots, e_n$ span an order in $E$.
If not, we can find some $r$ such that 
$\ell^{r} \alpha_{i, j, k}$ is $\ell$-integral for all $i$, $j$, $k$.
Then the span of $\ell^r e_1, \ldots, \ell^r e_n$ is an order in $E$.

Once we have found one order $\mathcal{O}$, 
we iteratively enlarge it.
Specifically, there are finitely many $\mathbb{Z}_{\ell}$-lattices $\mathcal{O}'$ with
$\mathcal{O} \subsetneq \mathcal{O}' \subseteq \frac{1}{\ell} \mathcal{O}$.
Indeed, these lattices are in bijection with 
the finitely many $\mathbb{F}_{\ell}$-vector subspaces of $\frac{1}{\ell} \mathcal{O} / \mathcal{O}$,
which can be enumerated algorithmically.
For each lattice, it is a finite computation to determine whether it is closed under multiplication:
simply choose a basis $e_i$, and compute the constants $\alpha_{i, j, k}$ as above.

At any time, if a larger order $\mathcal{O}'$ is found,
we replace $\mathcal{O}$ with $\mathcal{O}'$ and repeat.

We claim that this process must terminate.
If not, there is an infinite ascending chain of orders
$\mathcal{O}_1 \subsetneq \mathcal{O}_2 \subsetneq \mathcal{O}_3 \subsetneq \cdots$.
But the trace pairing must take integral values on any order,
so each $\mathcal{O}_n$ is contained in the dual under the trace pairing of $\mathcal{O}_1$,
a contradiction.

Hence, the algorithm terminates, giving some order $\mathcal{O}$ 
such that no order $\mathcal{O}'$ satisfies
$\mathcal{O} \subsetneq \mathcal{O}' \subseteq \frac{1}{\ell} \mathcal{O}$.
By Lemma \ref{order_step_by_step}, $\mathcal{O}$ is maximal.
\end{proof}

\begin{lem}
\label{field_decomp}
There is a finite-time algorithm that determines the decomposition 
of a commutative \'etale $\mathbb{Q}_{\ell}$-algebra as a direct sum of fields.

Specifically, the algorithm takes as input a commutative \'etale $\mathbb{Q}_{\ell}$-algebra $E$,
presented as $E := E_\Q\otimes_\Q \Q_\ell$ with $E_\Q$ a commutative \'etale $\Q$-algebra.
It outputs $\ell$-adic approximations to elementary idempotents $e_1, \ldots, e_n$,
such that
$E = \bigoplus_i e_i E$
is the decomposition of $E$ as a direct sum of fields;
these $e_i$'s can be computed to any desired $\ell$-adic precision.
\end{lem}

\begin{proof}
The algorithm is in two steps.
First, we determine elements $\widetilde{e}_i \in \mathcal{O}_{E}$ such that
$\widetilde{e}_i$ is a unit in $\mathcal{O}_{E_i}$,
but has strictly positive valuation in $\mathcal{O}_{E_j}$ for all $j \neq i$.
Then, we show how to produce successively better approximations to the true idempotents $e_i$.

First, use Lemma \ref{maxl_order} to compute a maximal order $\mathcal{O}_E$ in $E$.
Since $E$ is commutative, $\mathcal{O}_E$ is unique;
it is the direct sum of the maximal orders $\mathcal{O}_{E_j}$ in the number field summands of $E$.

The quotient $\mathcal{O}_E / \ell \mathcal{O}_E$ is a finite $\mathbb{F}_{\ell}$-algebra.
Let $\mathfrak{N}$ be its nilradical; 
this can be computed in finite time 
(for example, by computing characteristic polynomials).
The quotient of $\mathcal{O}_E / \ell \mathcal{O}_E$ by $\mathfrak{N}$
is the direct sum of the residue fields of the rings $\mathcal{O}_{E_i}$.
By a finite search, we can find all idempotent elements of this quotient ring.
There will be $2^n$ of them, where $n$ is the number of fields in the decomposition of $E$.
They form a lattice in which the join of two elements is given by their product;
the elementary idempotents are minimal nonzero elements in this lattice.
Clearly, these elementary idempotents can be computed in finite time.
Choosing lifts to $\mathcal{O}_E$ gives us elements $\widetilde{e}_1, \ldots, \widetilde{e}_n$,
such that each $\widetilde{e}_i$ is a unit in $\mathcal{O}_{E_i}$, congruent to $1$ modulo its maximal ideal,
and is contained in the maximal ideal of each $\mathcal{O}_{E_j}$ for $j \neq i$.

The absolute ramification degree of each $E_i$ is bounded above by $[E : \mathbb{Q}_{\ell}]$.
It follows (by a calculation in each field $E_i$) that $\widetilde{e}_i^{\ell^{r [E : \mathbb{Q}_{\ell}]}}$
is congruent to the $i$-th idempotent $e_i$ modulo $\ell^r$.
\end{proof}

\begin{defn}
Let
$E = \bigoplus_i E_i$
be a commutative \'etale $\mathbb{Q}_{\ell}$-algebra.
An \emph{approximate elementary idempotent} is an element $\widetilde{e}_i \in \mathcal{O}_E$
that is a nonunit in $E_j$ for all $j \neq i$, and that is congruent to $1$ modulo the maximal ideal of $\mathcal{O}_{E_i}$.
\end{defn}

The proof of Lemma \ref{field_decomp} shows that 
an approximate elementary idempotent determines the field $E_i$,
and given an approximate elementary idempotent $\widetilde{e}_i$,
we can compute arbitrarily accurate $\ell$-adic approximations to the idempotent that projects onto $E_i$.

\begin{lem}
\label{center}
There is a finite-time algorithm which, on input $E$ a semisimple algebra over $\mathbb{Q}_{\ell}$, 
presented as $E := E_\Q\otimes_\Q \Q_\ell$ with $E_\Q$ a semisimple algebra over $\Q$, outputs the center $Z(E) = Z(E_\Q)\otimes_\Q \Q_\ell$,
with output a basis for $Z(E_\Q)$.
\end{lem}

\begin{proof}
Let $a_1, \ldots, a_n$ be a basis for $E_\Q$ over $\Q$.
Then $Z(E_\Q)$ is the set of solutions $z$ to the system of simultaneous linear equations
$z a_i = a_i z$;
finding a basis for this solution set is a routine problem in linear algebra.
\end{proof}

\begin{lem}
Let $E$ be a semisimple algebra over $\mathbb{Q}_{\ell}$.
Let $Z(E)$ be the center of $E$, and let $e_i$ be an elementary idempotent in $Z(E)$.
Then $Z_i = e_i Z(E)$ is a field, and $e_i E$ is a central algebra over $Z_i$.
In particular, we have an isomorphism
$e_i E \cong M_{n_i} (D_i)$,
where $D_i$ is a central simple algebra over $Z_i$;
the integer $n_i$ and the isomorphism class of $D_i$ are uniquely determined.
\end{lem}

\begin{proof}
Apply Proposition \ref{structure}.
\end{proof}

\begin{lem}
\label{find_n}
There is a finite-time algorithm which, on input ($E$, $\widetilde{e}$), with $E$ a semisimple algebra over $\mathbb{Q}_{\ell}$, 
presented as $E := E_\Q\otimes_\Q \Q_\ell$ with $E_\Q$ a semisimple algebra over $\Q$, and $\widetilde{e}\in E$ an approximate elementary idempotent, outputs the integer $n_i$ such that 
$e_i E \cong M_{n_i} (D_i)$
for some division algebra $D_i$, and the invariant of the algebra $D_i$.
\end{lem}

\begin{proof}
Let $\mathcal{O}$ be a maximal order in $E$, let $e_i$ be the elementary idempotent corresponding to $\widetilde{e}$, and let $\mathcal{O}_i := e_i \mathcal{O}$.
We'll show that it is enough to compute the number of elements in the (unique)
maximal two-sided ideal of $\mathcal{O}_i / \ell \mathcal{O}_i\simeq (\widetilde{e} \mathcal{O}) / \ell (\widetilde{e} \mathcal{O})$.

For $D_i$, a central simple algebra over a $p$-adic field,
standard structure results imply that there is a valuation $v$ on $D_i$,
unique up to scaling,
and all ideals of $\mathcal{O}_{D_i}$ are powers of the maximal ideal $\mathfrak{m}$.
It follows by a routine calculation that all two-sided ideals of $M_{n_i}(\mathcal{O}_{D_i})$ 
are of the form $M_{n_i}(\mathfrak{m}^k)$.
In particular, $M_{n_i}(\mathcal{O}_{D_i})$ has a unique maximal ideal $M_{n_i}(\mathfrak{m})$,
which contains $\ell$.

Suppose $D_i$ has dimension $d_i^2$ over its center $Z(D_i) = e_i Z(E)$, 
so
$[e_i E : \mathbb{Q}_{\ell}] = d_i^2 n_i^2 [e_i Z(E) : \mathbb{Q}_{\ell}]$.
The structure theory of central simple algebras over $p$-adic fields gives that
$D_i$ has ramification degree $d_i$ over its center.
It follows that 
$\dim_{\mathbb{F}_{\ell}} M_{n_i}(\mathcal{O}_{D_i}) / M_{n_i}(\mathfrak{m}) = d_i n_i^2 [Z(D_i)^{\ur} : \mathbb{Q}_{\ell}],$
where $Z(D_i)^{\ur}$ is the maximal unramified subfield of $Z(D_i)$.

Similarly, taking $\mathfrak{m}_{Z(D_i)}$ to be the unique maximal ideal in $\mathcal{O}_{Z(D_i)}$, we have
$$\dim_{\mathbb{F}_{\ell}} (\mathcal{O}_{Z(D_i)} / \mathfrak{m}_{Z(D_i)}) = [Z(D_i)^{\ur} : \mathbb{Q}_{\ell}].$$

Finally we turn to the invariant of $D_i$.
Recall the theory of division algebras over a local field: 
The quotient $k_i = D_i / \mathfrak{m}_{Z(D_i)}$ is a finite field -- in fact, the finite field of order $p^{d_i}$.
There exist a uniformizer $\pi$ (i.e.\ a generator of $\mathfrak{m}_{Z(D_i)}$)
and a unique integer $r$ with $1 \leq r \leq d_i$ such that conjugation by $\pi^r$ induces the Frobenius endomorphism on $k_i$ -- in other words,
\[ \pi^r x \equiv x^p \pi^r (\text{mod } \pi^{r+1}) \]
for all $x \in \mathcal{O}_{D_i}$.
The invariant of $D_i$ is, by definition, $r / d_i$.

Now the quotient
\[ (\mathcal{O}_{D_i}) / M_{n_i}(\mathfrak{m}) \]
is isomorphic to $M_{n_i}(k_i)$,
so its center
\[ Z((\mathcal{O}_{D_i}) / M_{n_i}(\mathfrak{m})) \]
is isomorphic to $k_i$ (it is the set of diagonal matrices).
Choose some $x \in Z((\mathcal{O}_{D_i}) / M_{n_i}(\mathfrak{m}))$ that generates the field $k_i$.
Then a calculation shows that, for any $a \in \mathbb{Z}$ and any
\[ M \in \mathfrak{m}_{Z(D_i)}^a - \mathfrak{m}_{Z(D_i)}^{a+1} \]
we have
\[ M x \equiv x^{p^b} M (\text{mod }\mathfrak{m}_{Z(D_i)}^{a+1}),  \]
where $ra \equiv b$ (mod $d_i$)
and $r$ is the integer defined above (i.e.\ $r / d_i$ is the invariant of the algebra $D_i$).

Now we explain the algorithm.
Use Lemma \ref{maxl_order} to find a maximal order $\mathcal{O}$ in $E$.
The quotient ring $\mathcal{O}_i / \ell \mathcal{O}_i$ is finite, so we can compute a presentation for it.
It is a finite calculation to determine a maximal ideal $\mathfrak{m}$ in $\mathcal{O}_i / \ell \mathcal{O}_i$,
and hence compute $\dim_{\mathbb{F}_{\ell}} M_{n_i}(\mathcal{O}_{D_i}) / M_{n_i}(\mathfrak{m})$.

Similarly, let $Z$ be the center of $\mathcal{O}$; by Lemmas \ref{center} and \ref{maxl_order},
we can compute the maximal order $\mathcal{O}_Z$ in $Z$.
Again by a finite computation, we can compute the maximal ideal in 
$\widetilde{e} (\mathcal{O}_Z / \ell \mathcal{O}_Z)$; call it $\mathfrak{m}_{Z_i}$.
From this we determine $\dim_{\mathbb{F}_{\ell}} (\mathcal{O}_{Z(D_i)} / \mathfrak{m}_{Z(D_i)}) [Z(D_i)^{\ur} : \mathbb{Q}_{\ell}]$.

Having determined
$[Z(D_i)^{\ur} : \mathbb{Q}_{\ell}]$, $d_i n_i^2 [Z(D_i)^{\ur} : \mathbb{Q}_{\ell}]$, and $d_i^2 n_i^2 [e_i Z(E) : \mathbb{Q}_{\ell}]$,
it is straightforward to determine $n_i$.

Finally, to determine the invariant $r / d_i$ of $D_i$, we simply choose 
\[ M \in \mathfrak{m}_{Z(D_i)} - \mathfrak{m}_{Z(D_i)}^{2}, \]
take $x$ to generate the field $Z((\mathcal{O}_{D_i}) / M_{n_i}(\mathfrak{m}))$,
find the unique $b \in \mathbb{Z} / d_i \mathbb{Z}$ such that
\[ M x \equiv x M^b (\text{mod } \mathfrak{m}_{Z(D_i)}^{2}), \]
and take $r$ the multiplicative inverse of $b$ modulo $d_i$.
\end{proof}

We may now prove Propositions \ref{simple_reps} and \ref{local_semisimple_algebras}.

\begin{proof}[Proof of Proposition \ref{simple_reps}.]

We are given the semisimple algebra $E$ and its action on the vector space $V$.
By Lemmas \ref{center} and \ref{field_decomp}, we can find a system of approximate elementary idempotents for the center $Z(E)$;
by Lemma \ref{find_n} we can find the integers $n_i$ in the decomposition
$E \cong \bigoplus M_{n_i} (E_i)$.

Next let us show how to approximate the trace of any element $\alpha \in E$ on any simple $E$-module $E_i^{\oplus n_i}$.
This trace is simply $\operatorname{Tr}(e_i \alpha)$, 
where $e_i$ is the elementary idempotent of $E$ projecting onto $M_{n_i} (E_i)$.
By the remark following Lemma \ref{field_decomp}, we can approximate $e_i$ to any desired $\ell$-adic precision.
We need to control the precision on $\operatorname{Tr}(e_i \alpha)$;
scaling $\alpha$ by a power of $\ell$, we may assume that $\alpha$ has all eigenvalues integral.
In this setting, if
$\widetilde{e}_i \equiv e_i \pmod*{\ell^n Z(E)}$,
then
$\operatorname{Tr}(\widetilde{e}_i \alpha) \equiv \operatorname{Tr}(e_i \alpha) \pmod*{\ell^n}$.

All that remains is to determine the decomposition of a given $E$-module $V$ in terms of the simple $E$-modules;
for this it is enough to determine the multiplicity of a single $E_i^{\oplus n_i}$ as a factor of $V$.
For this, we consider the action of an approximate elementary idempotent $\widetilde{e}_i$ on $V$.
We know that $\widetilde{e}_i$ acts on $E_i^{\oplus n_i}$ with eigenvalues that are $\ell$-adic units,
while the eigenvalues of its actions on any other simple $E$-module are zero modulo $\ell$.
We can compute the number of nonzero eigenvalues of $\widetilde{e}_i$ on $V$;
this number will be exactly
$m_i \dim_{\mathbb{Q}_{\ell}} E_i^{\oplus n_i}$,
where $m_i$ is the multiplicity we want to find.
\end{proof}

\begin{proof}[Proof of Proposition \ref{local_semisimple_algebras}.]
Apply Lemmas \ref{center}, \ref{field_decomp}, and \ref{find_n}.
\end{proof}

%% file: algorithmic_decomposition_of_semisimple_algebras_over_number_fields.tex
\section{Algorithmic decomposition of semisimple algebras over number fields.}
\label{alg_dec_num_fld_alg}

\begin{prop}
\label{num_fld_algebra}
There is a finite-time algorithm which, on input $(K, E)$ with $E$ a semisimple algebra over a number field $K/\Q$, outputs $(s, (e_i)_{i=1}^s, (n_i)_{i=1}^s, (E_i)_{i=1}^s)$, where
\begin{equation} \label{global_decomp}
E \cong \bigoplus_{i=1}^s M_{n_i} (E_i), \end{equation}
with each $E_i$ a division algebra,
and each $e_i \in E$ the elementary idempotent projecting onto the $i$-th summand.
\end{prop}

\begin{lem}
\label{decomp_global_center}
There is a finite-time algorithm which, on input $(K, E)$ with $E$ a semisimple central algebra over a number field $K/\Q$, outputs the elementary idempotents $e_i$ projecting onto each summand in the decomposition \eqref{global_decomp}.
\end{lem}

\begin{proof}
By linear algebra we can compute a presentation for the center $Z(E)$.
This center decomposes as a sum of fields
$Z(E) = \bigoplus_i Z(E_i)$;
the idempotents we are looking for are exactly the projectors onto the summands.
Hence we are reduced to the problem of decomposing 
a commutative \'etale $K$-algebra as a sum of fields, which is standard.
\end{proof}

\begin{lem}
\label{real_inv}
There is a finite-time algorithm which, on input $E$ a semisimple central algebra over $\mathbb{R}$ (thus $E$ is either a matrix algebra over $\mathbb{R}$ or over the Hamilton quaternions $\mathbb{H}$), 
presented as $E := E_\Q\otimes_\Q \R$ with $E_\Q$ a semisimple central algebra over $\Q$, outputs true if and only if $E$ is a matrix algebra over $\mathbb{R}$.
\end{lem}

\begin{proof}
We may assume $\dim_{\mathbb{R}} E = (2n)^2$.
The trace pairing defines a quadratic form on $E$, 
whose signature is $(2n^2 + n, 2n^2 -n)$ if $E \cong M_{2n} (\mathbb{R})$,
and $(2n^2 - n, 2n^2 + n)$ if $E \cong M_n (\mathbb{H})$.
Given $E$, one can determine the signature of the trace pairing by the Gram-Schmidt process.
\end{proof}

\begin{lem}
\label{zero_inv}
There is a finite-time algorithm which, on input $(K, E)$ with $E$ a semisimple central algebra over a number field $K/\Q$, outputs a finite list of places of $K$ such that the local invariant of $E$ at any place not on the list is $0$.
\end{lem}

\begin{proof}
Choose any order $\mathcal{O}$ in $E$ (not necessarily maximal),
and compute its discriminant $N$ with respect to the trace pairing.
Return the set of finite places dividing $N$, plus all the archimedean places.
\end{proof}

\begin{lem}
\label{all_inv}
There is a finite-time algorithm which, on input $(K, E)$ with $E$ a semisimple central algebra over a number field $K/\Q$, outputs all the nonzero local invariants of $E$.
\end{lem}

\begin{proof}
The invariant of $E$ at any complex place is zero; 
at a real place it is $0$ or $1/2$, depending whether the associated division algebra is $\mathbb{R}$ or $\mathbb{Q}$.
By Lemma \ref{real_inv} we can determine which case holds at each real place.
This settles the archimedean places.

By Lemma \ref{zero_inv}, we only need to find the invariants of $E$ at finitely many nonarchimedean places;
to find the invariant of a nonarchimedean localization of $E$, we apply Proposition \ref{local_semisimple_algebras}.
\end{proof}

Now we may prove Proposition \ref{num_fld_algebra}.

\begin{proof}[Proof of Proposition \ref{num_fld_algebra}.]
The algorithm works as follows.
First we decompose $Z(E)$ into fields (Lemma \ref{decomp_global_center}); 
this determine the decomposition of $E$ into simple algebras.
Passing to one simple factor, we may assume that $E$ is simple, with center $K'$.
Then $E\cong M_{n_1}(E_1)$ is a matrix algebra over a division ring $E_1$.
If $\dim_{K'} E_1 = d_1^2$, then $\dim_K E = d_1^2 n_1^2$; 
since the latter dimension can be read off from our presentation,
we just need to compute the number $d_1$.

Global class field theory describes the structure of the central simple algebra $E_1$.
Its localization at each place (finite or infinite) of $K'$ 
is determined by an invariant in $\mathbb{Q} / \mathbb{Z}$.
All but finitely many of these invariants vanish, 
and their least common denominator is $d_1$.
We can compute the nonzero invariants by Lemma \ref{all_inv},
and we are done.
\end{proof}

%% file: mumford.tex
\section{Mumford coordinates.}
\label{sec:mum_coord}

Here we collect some results that enable us to work algorithmically with the moduli space of abelian varieties.  We use Mumford's explicit description of the moduli of abelian varieties with $\delta$-markings \cite{Mumford_coord1}.

A $\delta$-marking is a certain type of level structure, introduced by Mumford.  Any abelian variety with $\delta$-marking can be canonically embedded in projective space; we say that the resulting projectively embedded abelian variety is in \emph{Mumford form} (Definition \ref{defn:mumford_form}).  For fixed level $\delta$, any abelian variety admits only finitely many embeddings in Mumford form.  We will make algorithmic several basic tasks involving moduli spaces:

Lemma \ref{char_delta_lemma} allows us to test whether an abelian scheme is in Mumford form (which enables us to find Mumford forms for arbitrary abelian schemes, by brute-force search).

Lemma \ref{all_mumf_coords}
Given an abelian variety $A$ over a characteristic-zero field $K$, can compute (for given $\delta$) all Mumford embeddings of $A$.

Lemma \ref{the subfamily of abelian varieties in a family isomorphic to a given one is computable}
Given a family of abelian varieties and a fixed abelian variety $A$, one can compute all fibers of the family that are isomorphic to $A$.

\subsection{Introduction.}

We will start with Mumford's explicit description of abelian varieties,
and their moduli, in terms of (algebraic) theta functions and theta-zero values.
The material is classical -- it was known, in some form, to Riemann --
but we will use Mumford's presentation.
The material we need is in \cite[\S\S 1-3]{Mumford_coord1} and \cite[\S 6]{Mumford_coord2}.

Given a choice of ``level'' $\delta$, a polarization on $A$, and some additional discrete data (a ``$\delta$-marking''), Mumford defines a \emph{canonical} embedding of $A$ into a certain projective space depending only on $\delta$.
This embedding has the pleasant property that the line bundle $\mathcal{L} = \mathcal{O}(1)$ on $A$ is invariant under pullback both by the inversion $x \mapsto -x$ on $A$ and by a certain torsion subgroup (the ``$\delta$-torsion'') of $A$.
Inversion and these torsion translations on $A$ extend to linear maps of the ambient projective space, and in fact these linear maps can be described by universal equations independent of the abelian variety $A$.

Since the projective embedding is canonical, the coordinates $Q_x$ of the origin of $A$ define invariants of $A$ (with its $\delta$-marking).
One can write down equations for $A$ as a subscheme of projective space, whose coefficients depend only on the $Q_x$.
In other words, the coordinates $Q_x$ can be used as coordinates on the moduli space of abelian varieties (with suitable polarization and $\delta$-marking).  

However, the addition law on $A$ does not have a simple form in Mumford's coordinates.

\subsection{$\delta$-markings.}
\label{Delta markings}

(In \cite{Mumford_coord2} the constructions are carried out over $\mathbb{Z}[1 / \prod d_i]$,
but we will only need the results in characteristic $0$.)

We begin by recalling some notation from \cite[\S\S 1-3]{Mumford_coord1} and \cite[\S 6]{Mumford_coord2}.
Let $\delta = (d_1, \ldots, d_g)$ be a collection of elementary divisors; 
that is, they are positive integers such that $d_{i+1} | d_i$.
We also assume that all the $d_i$ are divisible by $8$.
(For the purposes of this paper, it is fine to assume $\delta = (8, 8, \ldots, 8).$)
We will write $2 \delta = (2d_1, \ldots, 2d_g)$
and 
\[ \left | \delta \right | := \prod_{i=1}^g d_i. \]

Let 
\[ K(\delta) := \bigoplus_{i=1}^k \mathbb{Z} / (d_i \mathbb{Z}).\]
Note that there is a natural inclusion $K(\delta) \subseteq K(2 \delta)$,
given by multiplication by two on the individual factors 
\[ \mathbb{Z} / (d_i \mathbb{Z})\rightarrow \mathbb{Z} / (2 d_i \mathbb{Z}) . \]
Also, let $Z_2 \subseteq K(\delta)$ be the subgroup of points that are divisible by 2.

Next we define the group scheme $\mathcal{G}(\delta)$.
As a set, the $S$-points of $\mathcal{G}(\delta)$ (for connected schemes $S$) are tuples $(t, a, \ell)$ such that:
\begin{enumerate}
\item $t$ is a section of $\mathbb{G}_{m, S}$,
\item $a \in K(\delta)$,
\item $\ell \in K(\delta)^{\vee} = \operatorname{Hom}(K(\delta), \mathbb{G}_m)$ -- or more concretely, $\ell = (\ell_1, \dots, \ell_g)$, where each $\ell_i$ is a $d_i$-th root of unity in $\mathbb{G}_{m, S}$.  (We will write the group law in $K(\delta)^{\vee}$ additively.)
\end{enumerate}
\begin{rmk}
    There is a natural pairing
    \[ K(\delta) \times K(\delta)^{\vee} \rightarrow \mathbb{G}_m, \]
    which we will write as
    \[ ( a, \ell ) \mapsto \langle a, \ell \rangle \]
    
    We will write the group law on $K(\delta)^{\vee}$ additively.
\end{rmk}

The group structure on $\mathcal{G}(\delta)$ is given by
\[ (t_1, a_1, \ell_1) \cdot (t_2, a_2, \ell_2) 
= (t_1 t_2 \langle a_2, \ell_1 \rangle , a_1 + a_2, \ell_1 + \ell_2).\]

We note that the injection $\mathbb{G}_m \mapsto \mathcal{G}(\delta)$ given by $t \mapsto (t, 0, 0)$
gives rise to an exact sequence
\[ 0 \rightarrow \mathbb{G}_m \rightarrow \mathcal{G}(\delta) \rightarrow K(\delta) \oplus K(\delta)^{\vee} \rightarrow 0. \]
Furthermore, the commutator on $G$ descends to the pairing $e(- , -)$ on $K(\delta) \oplus K(\delta)^{\vee}$ given by
\begin{equation} 
\label{pairing}
e((a_1, \ell_1), (a_2, \ell_2)) = 
\frac{\langle a_1, \ell_2 \rangle} {\langle a_2, \ell_1 \rangle}.
\end{equation}

Given an abelian scheme $A$ over a base $S$ of characteristic zero and a line bundle $\mathcal{L}$ on $A$, 
we define a group scheme $\mathcal{G}(\mathcal{L})$ over $S$ as follows.
For any $S$-scheme $T$, $\mathcal{G}(\mathcal{L})(T)$ is the group of pairs $(s, \alpha)$, where $s$ is a section of $A \times_S T \rightarrow T$, and 
\[ \alpha \colon T_{s}^* L \rightarrow L \]
is an isomorphism.  (Here $T_{s} \colon A \rightarrow A$ is translation by $s$.)

We note again that there is an injection $\mathbb{G}_m \mapsto \mathcal{G}(\mathcal{L})$, given by $\alpha \mapsto (0, \alpha)$;
that is, by the usual action of $\mathbb{G}_m$ by multiplication on the line bundle $\mathcal{L}$;
this injection gives rise to an exact sequence
\[ 0 \rightarrow \mathbb{G}_m \rightarrow \mathcal{G}(\mathcal{L}) \rightarrow A[\mathcal{L}] \rightarrow 0, \]
where $A[\mathcal{L}]$ is the image of the map $\mathcal{G}(\mathcal{L}) \rightarrow A$ given by $(s, \alpha) \mapsto s$.

\begin{defn}
Let $A$ be an abelian scheme of dimension $g$ over a $\mathbb{Q}$-scheme $S$.
A \emph{$\delta$-marking} on $A$ is the following data:
\begin{enumerate}
\item A symmetric, very ample line bundle $\mathcal{L}$ on $A$.  (``Symmetric'' means that there is an isomorphism $\mathcal{L} \cong \mathcal{L}^{-1}$.)
\item An isomorphism $\beta \colon \mathcal{G}(\mathcal{L}) \rightarrow \mathcal{G}(\delta)$ which is the identity on the subgroup $\mathbb{G}_{m, S}$.
\end{enumerate}

The \emph{N\'eron--Severi class} (or \emph{polarization}) of $\delta$ is simply the N\'eron--Severi class of $\mathcal{L}$ in $H^2(A)$.
\end{defn}

Given $A$ and $\delta$, there are finitely many $\delta$-markings on $A$ of given N\'eron--Severi class; see Lemmas \ref{constr_autos} and \ref{automorphisms}.

We note that any $\delta$-marking gives rise to an identification of the torsion subgroup $A[\delta]$ with $K(\delta) \oplus K(\delta)^{\vee}$.

\subsection{Let us count the ways (to put Mumford coordinates on an abelian variety).}

We begin with a short discussion of the structure of the group $\mathcal{G}(\delta)$.
\begin{defn}
    Let 
    \[ \operatorname{Aut}_{\mathbb{G}_m} \mathcal{G}(\delta) = \{ f \in \operatorname{Aut} \mathcal{G}(\delta) \mid f(x) = x \text{ for all } x \in \mathbb{G}_m \} \]
be the group of automorphisms of $\mathcal{G}(\delta)$ acting trivially on $\mathbb{G}_m$.
\end{defn}
Our goal is to compute all automorphisms in $\operatorname{Aut}_{\mathbb{G}_m} \mathcal{G}(\delta)$.

To start with, let $a_{(i)}$ be the standard generator of $\mathbb{Z} / (d_i \mathbb{Z}) \subseteq K(\delta)$ (that is, the image of $1$ under the inclusion of the $i$-th summand
\[ \mathbb{Z} / (d_i \mathbb{Z}) \hookrightarrow \bigoplus_{j=1}^k \mathbb{Z} / (d_j \mathbb{Z}) = K(\delta)). \]
Choose, for each $i$, a primitive $d_i$-th root of unity $\zeta_i$, and let $\ell_{(i)} \in (\mathbb{Z} / (d_i \mathbb{Z}))^{\vee} \subseteq K(\delta)^{\vee}$ be such that
\[ \langle a_{(i)}, \ell_{(i)} \rangle = \zeta_i. \]

Note that $\mathcal{G}(\delta)$ is generated by $\mathbb{G}_m$ and the elements $(1, a_{(i)}, 0)$ and $(1, 0, \ell_{(i)})$, subject to the relations that
\[ (1, a_{(i)}, 0)^{d_i} = (1, 0, \ell_{(i)})^{d_i} = (1, 0, 0), \]
\[ (1, 0, \ell_{(i)})(1, a_{(i)}, 0) = \zeta_i (1, a_{(i)}, 0)(1, 0, \ell_{(i)}), \]
and all other pairs of generators commute.

It follows that we can specify an automorphism in $\operatorname{Aut}_{\mathbb{G}_m} \mathcal{G}(\delta)$
by giving the image of $(1, a_{(i)}, 0)$ and $(1, 0, \ell_{(i)})$.

Let $\operatorname{Sp}(K(\delta) \oplus K(\delta)^{\vee})$ denote the set of automorphisms $\sigma$ of $K(\delta) \oplus K(\delta)^{\vee}$ that respect the pairing $e$, that is, that satisfy
\[ e(\sigma(x), \sigma(x')) = e(x, x') \]
for all $x, x' \in K(\delta) \oplus K(\delta)^{\vee}$.

\begin{lem}
\label{constr_autos}
There are finitely many automorphisms $f$ of $\mathcal{G}(\delta)$ that act as the identity on $\mathbb{G}_m$; each such $f$ can be written explicitly as follows.

Suppose given any
\[ \sigma \in \operatorname{Sp}(K(\delta) \oplus K(\delta)^{\vee}). \]
Let 
\[ \sigma(a_{(i)}) = (b_i, m_i), \]
and let $\psi_i$ be such that
\[ \psi_i^{d_i} = \langle b_i, m_i \rangle^{\frac{d_i(d_i-1)}{2}}. \]
Similarly, let
\[ \sigma(\ell_{(i)}) = (b_i', m_i'), \]
and let $\omega_i$ be such that
\[ \omega_i^{d_i} = \langle b_i', m_i' \rangle^{\frac{d_i(d_i-1)}{2}}. \]

Note that there are $d_i$ choices for each of $\psi_i$ and $\omega_i$, and that all choices can be computed explicitly in terms of roots of unity.

Then there is a unique automorphism $f$ of $\mathcal{G}(\delta)$, acting as the identity on $\mathbb{G}_m$ and as $\sigma$ on $\operatorname{Sp}(K(\delta) \oplus K(\delta)^{\vee})$, such that
\[ f((1, a_{(i)}, 0)) = (\psi_i, b_i, m_i) \]
and
\[ f((1, 0, \ell_{(i)})) = (\omega_i, b_i', m_i'). \]
\end{lem}

\begin{proof}
A calculation shows that $f((1, a_{(i)}, 0))$ and $f((1, 0, \ell_{(i)}))$ satisfy the requisite relations, and hence that $f$ extends to an automorphism of $\mathcal{G}(\delta)$.

To see that all $f$ arise in this way, note that any $f \in \operatorname{Aut}_{\mathbb{G}_m} \mathcal{G}(\delta)$ descends to an automorphism $\sigma$ of the quotient $\mathcal{G}(\delta) / \mathbb{G}_m = K(\delta) \oplus K(\delta)^{\vee}$.
The commutator of any two elements $(t, a, \ell), (t', a', \ell') \in \mathcal{G}(\delta)$ is given by 
\[ e( (a, \ell), (a', \ell') ) \in \mathbb{G}_m \subseteq \mathcal{G}(\delta), \]
so $\sigma$ must preserve the pairing $e$.

Finally, if we write
\[ f((1, a_{(i)}, 0)) = (\psi_i, b_i, m_i) \]
and
\[ f((1, 0, \ell_{(i)})) = (\omega_i, b_i', m_i'), \]
the relations
\[ (1, a_{(i)}, 0)^{d_i} = (1, 0, \ell_{(i)})^{d_i} = (1, 0, 0) \]
imply that 
\[ \psi_i^{d_i} = \langle b_i, m_i \rangle^{\frac{d_i(d_i-1)}{2}} \]
and
\[ \omega_i^{d_i} = \langle b_i', m_i' \rangle^{\frac{d_i(d_i-1)}{2}}. \]
\end{proof}

For the reader's convenience, we give a more abstract reformulation of Lemma \ref{constr_autos} below, though it is not used in the sequel.

\begin{lem}
\label{automorphisms}
Consider the group
\[ \operatorname{Aut}_{\mathbb{G}_m} \mathcal{G}(\delta) = \{ f \in \operatorname{Aut} \mathcal{G}(\delta) \mid f(x) = x \text{ for all } x \in \mathbb{G}_m \} \]
of automorphisms of $\mathcal{G}(\delta)$ acting trivially on $\mathbb{G}_m$.

This group sits in an exact sequence
\[ 0 \rightarrow (K(\delta) \oplus K(\delta)^{\vee})^{\vee} \rightarrow \operatorname{Aut} \mathcal{G}(\delta) \rightarrow \operatorname{Sp}(K(\delta) \oplus K(\delta)^{\vee}) \rightarrow 0.
\]
\end{lem}

\subsection{Representation theory of the group $\mathcal{G}(\delta)$.}

Mumford's projective embeddings of abelian varieties and their moduli are based on the following result, which describes the action of $\mathcal{G}(\mathcal{L})$ on the global sections $\Gamma(\mathcal{L})$.

\begin{lem}
\label{Mumford_rep}
\cite[\S 1, Prop.\ 3 and Thm.\ 2]{Mumford_coord1}
The group $\mathcal{G}(\delta)$ has (up to isomorphism) a unique irreducible representation $V(\delta)$ on which $\mathbb{G}_m \subseteq \mathcal{G}(\delta)$ acts via the identity (i.e.\ on which any $\lambda \in \mathbb{G}_m$ acts by multiplication by $\lambda$).

The representation $V(\delta)$ can be described (over any field $k$ of characteristic zero) as follows:
As a vector space, $V(\delta)$ is the space of $k$-valued functions on $K(\delta)$;
the action of $\mathcal{G}(\delta)$ is given by
\[ ((t, a, \ell) \cdot f)(x) = t \langle x, \ell \rangle f(x+a). \]

The group $\mathcal{G}(\mathcal{L})$ acts on the space of global sections $\Gamma(\mathcal{L})$; if there is an isomorphism $\mathcal{G}(\mathcal{L}) \cong \mathcal{G}(\delta)$, then $\Gamma(\mathcal{L})$ and $V(\delta)$ are isomorphic as representations of $\mathcal{G}(\delta)$.
\end{lem}

As in Lemma \ref{Mumford_rep}, let $V(\delta)$ be the vector space of $\mathbb{Q}$-valued functions on $K(\delta)$.
Let 
\[ \mathbb{P}(V(\delta)) = \operatorname{Proj} \Sym V(\delta); \]
this is a projective $\mathbb{Q}$-scheme whose points are in bijection with linear functionals on $V(\delta)$.

For all $a \in K(\delta)$, let $X_a \in V(\delta)$ be the function that is $1$ at $a$ and $0$ elsewhere;
this is naturally a section of $\mathcal{O}(1)$ on $\mathbb{P}(V(\delta))$,
and the sections $X_a$ (for $a \in K(\delta)$) give a set of projective coordinates on $\mathbb{P}(V(\delta))$.

Suppose $(\mathcal{L}, \beta)$ is a $\delta$-marking on an abelian scheme $A$ over some base $\mathbb{Q}$-scheme $S$.
By Lemma \ref{Mumford_rep}, we can identify $\Gamma(\mathcal{L})$ with $V(\delta)$ in a $\mathcal{G}({\delta})$-equivariant way.
This identification is unique up to scaling (by Schur's lemma),
so it gives rise to an embedding 
\[ A \hookrightarrow \mathbb{P}(V(\delta)). \]

This is quite a strong statement!
The embedding of $A$ is defined, not (as might be imagined) up to automorphisms of projective space, 
but as a specific map into the space $\mathbb{P}(V(\delta))$.
The projective coordinates $X_a(e)$ of the identity point of $A$ are an invariant of the abelian variety $A$, depending only on the discrete choices of the line bundle $\mathcal{L}$ and the $\delta$-marking $\beta$.

\begin{defn}
\label{defn:mumford_form}
An \emph{abelian scheme in Mumford form} is an abelian subscheme $A$
of $\mathbb{P}(V(\delta))_S$ arising from the construction above.

Its \emph{Mumford coordinates} $Q_a = Q_a(A) = Q_a(A, \mathcal{L}, \beta)$ are the projective coordinates $X_a(e)$ of the origin $e \in A$ in $\mathbb{P}(V(\delta))$, given as functions on $S$.
\end{defn}

\subsection{Characterizing $\delta$-markings.}
\label{char_delta}

Suppose $A$ is an abelian variety in Mumford form, coming from a $\delta$-marking $(A, \mathcal{L}, \beta)$.
Let $(a, \ell) \in K(\delta) \oplus K(\delta)^{\vee}$.
The $\delta$-marking gives an isomorphism $$K(\delta) \oplus K(\delta)^{\vee} \cong A[\mathcal{L}];$$ let $s$ be the image of $(a, \ell)$ under this isomorphism.

Translation by $s$ is an automorphism $T_s$ of the scheme $A$ (though not an automorphism of the abelian variety: $T_s$ does not fix the identity or respect the addition law).
Since $T_s^*(\mathcal{O}(1)) \cong \mathcal{O}(1)$, the automorphism $T_s$ of $A$ extends to a linear automorphism $f_s$ of the ambient projective space $\mathbb{P}(V(\delta))$.
This linear automorphism is determined by the action of $\mathcal{G}(\delta)$ on $\Gamma(\mathcal{L})$; its action on projective coordinates is given (up to scaling) by
\begin{equation} \label{translation_linear} f_s^*(X_b) = \langle b, \ell \rangle X_{a+b}. \end{equation}

\begin{lem}
\label{char_delta_lemma}
Let $A$ be an abelian scheme of dimension $g$ over a $\mathbb{Q}$-scheme $S$, presented as a subscheme $A \subseteq \mathbb{P}(V(\delta)) \times S$ of the projective space on $V(\delta)$.

This $A \subseteq \mathbb{P}(V(\delta)) \times S$ is an abelian scheme in Mumford form if and only if $A$ satisfies the following conditions.
\begin{enumerate}
\item 
$A$ is not contained in any hyperplane in $\mathbb{P}(V(\delta))$.

\item
The degree of the line bundle $\mathcal{L} = \mathcal{O}(1)|_A$ is $\left | \delta \right |$.

\item
The automorphisms $f_s$ of $\mathbb{P}(V(\delta))$ defined above (Equation \eqref{translation_linear}), 
for $s \in K(\delta) \oplus K(\delta)^{\vee}$, map $A$ into itself.

\item 
If $e$ is the identity section of $A$, then the map $f_s \colon A \rightarrow A$ coincides with translation by the section $f_s(e)$.

\item
The inverse map $A \rightarrow A$ extends to a linear automorphism of $\mathbb{P}(V(\delta))$.
\end{enumerate}
\end{lem}

\begin{proof}
For any line bundle $\mathcal{L}$ on $A$, the cokernel $H_{\mathcal{L}}$ of $\mathbb{G}_m \rightarrow \mathcal{G}(\mathcal{L})$ coincides with the kernel of the map $\phi_{\mathcal{L}} \colon A \rightarrow A^{\vee}$ defined by
\[ \phi_{\mathcal{L}} = T_x^* \mathcal{L} \otimes \mathcal{L}^{-1} \in \operatorname{Pic}^0(A) = A^{\vee}. \]
This group $H_{\mathcal{L}}$ has order
\[ \left | H_{\mathcal{L}} \right | = d(\mathcal{L})^2, \]
where $d(\mathcal{L})$ is the degree of $\mathcal{L}$. 
Furthermore, we have the equality
\[ \operatorname{dim} \Gamma(\mathcal{L}) = d(\mathcal{L}). \]

Returning our setting, 
suppose $A$ satisfies the list of conditions.
since the embedding of $A$ in $\mathbb{P}(V(\delta))$ has degree $\left | \delta \right |$, 
the line bundle $\mathcal{L}$ on $A$ has 
\[ \operatorname{dim} \Gamma(\mathcal{L}) = \left | \delta \right |. \]
Since $A$ is embedded in a $(\left | \delta \right | -1)$-dimensional projective space and does not lie in any hyperplane, 
the pullback map
\[ \Gamma(\mathbb{P}(V(\delta)), \mathcal{O}(1)) \rightarrow \Gamma(A, \mathcal{L}) \]
is an isomorphism.
In other words, the embedding of $A$ in $\mathbb{P}(V(\delta))$ is precisely the embedding given by the very ample line bundle $\mathcal{L}$.

The hypothesis on the inverse map implies that $\mathcal{L}$ is symmetric.

To finish, we just have to check that there is an isomorphism 
\[ \mathcal{G}(\delta) \cong \mathcal{G}(\mathcal{L}) \]
with respect to which
\[ \Gamma(\mathbb{P}(V(\delta)), \mathcal{O}(1)) \rightarrow \Gamma(A, \mathcal{L}) \]
is equivariant.
But this is clear: 
The action of $\mathcal{G}(\delta)$ on $\mathbb{P}(V(\delta))$ descends to $A$ by assumption.
Moreover, the action of $\mathcal{G}(\delta)$ on $\mathcal{O}(1)$ by linear automorphisms descends to an action on $\mathcal{L}$, equivariant over the action on $A$.
That is, our setup naturally gives rise to a homomorphism $\mathcal{G}(\delta) \cong \mathcal{G}(\mathcal{L})$.
This homomorphism is clearly injective and acts as the identity on the common subgroup $\mathbb{G}_m$;
since $\mathbb{G}_m$ has the same index $\left | \delta \right |^2$ in both $\mathcal{G}(\delta)$ and $\mathcal{G}(\mathcal{L}),$ 
the homomorphism $\mathcal{G}(\delta) \cong \mathcal{G}(\mathcal{L})$ is an isomorphism, 
so $A$ is indeed embedded in Mumford form.

The converse is easy:  Suppose $A$ is in Mumford form.  Then (1) follows because the projective embedding of $A$ is given by sections of a line bundle; (2) from the degree calculation above; (3) and (4) from the discussion preceding Equation \eqref{translation_linear}, and (5) from the symmetry of the line bundle $\mathcal{L}$.
\end{proof}

Next we recall a useful lemma from \cite{Mumford_coord1}: the Mumford coordinates of the origin determine the abelian variety.

\begin{lem}
\label{Mumford coords determine abelian variety}

Let $Z_2 \subseteq K(\delta)$ be the subgroup of points that are divisible by 2.

Suppose $A$ is an abelian variety in its Mumford embedding,
and let $Q_a$ be the coordinates of the origin on $A$
(as $a$ ranges over $K(\delta)$).
Then the variety $A$ in the projective space $\mathbb{P}(V(\delta))$ is cut out by the following quadratic equations (in variables $X_a$):
\begin{eqnarray*}
\left (  \sum_{\eta \in Z_2} l(\eta) Q_{c+d+\eta} Q_{c-d+\eta}  \right ) 
\left ( \sum_{\eta \in Z_2} l(\eta) X_{a+b+\eta} X_{a-b+\eta}   \right )    &  \\
- \left (  \sum_{\eta \in Z_2} l(\eta) Q_{c+b+\eta} Q_{c-b+\eta}  \right ) 
\left ( \sum_{\eta \in Z_2} l(\eta) X_{a+d+\eta} X_{a-d+\eta}   \right )   & 
\in \Gamma(\mathcal{O}(2)_\mathbb{P}(V(\delta)) \otimes  \mathcal{M}^2),
\end{eqnarray*}
for each tuple $(a, b, c, d, l)$, where $a, b, c, d \in K(2 \delta)$ are all congruent modulo $K(\delta)$,
and $l \colon Z_2 \rightarrow \{ \pm 1 \}$ is a group homomorphism.

In particular, knowledge of the projective coordinates $Q_a$ alone determines the abelian variety $A$.
\end{lem}

\subsection{Different Mumford coordinates on the same abelian variety.}

We want to use Mumford embeddings to detect whether two polarized abelian varieties are isomorphic.  To this end, we want to be able to find all Mumford embeddings of a given $A$ with given polarization.

\begin{lem}
\label{intertwining}
Let 
\[ \rho \colon \mathcal{G}(\delta) \rightarrow \operatorname{GL}(V(\delta)) \]
be the standard representation of $\mathcal{G}(\delta)$ on $V(\delta)$ (Lemma \ref{Mumford_rep}). 

For any automorphism $\phi$ of $\mathcal{G}(\delta)$, the two representations $\rho$ and $\phi \circ \rho$ of $\mathcal{G}(\delta)$ are isomorphic.
Furthermore, the isomorphism can be computed explicitly:
one can compute (working with exact arithmetic over a cyclotomic field) a $\mathcal{G}(\delta)$-equivariant map
\[ f \colon (V(\delta), \rho) \rightarrow (V(\delta), \phi \circ \rho) \]
from the representation $\rho$ to the representation $\phi \circ \rho$.
\end{lem}

\begin{proof}
The group $\mathcal{G}(\delta)$ has only the one irreducible representation (up to isomorphism) on which $\mathbb{G}_m$ acts as the identity (Lemma \ref{Mumford_rep}),
so $\phi \circ \rho$ must be isomorphic to this representation $\rho$.  This gives existence of $f$.

Once $f$ is known to exist, it can be computed by linear algebra.
\end{proof}

\begin{lem}\label{all mumford coordinates are computable from one}
Given $(A, \mathcal{L}, \beta)$ an abelian variety in Mumford embedding over a number field $K/\Q$,
one can compute all Mumford embeddings $(A, \mathcal{L}', \beta')$ over $K$ of the same abelian variety $A/K$ such that $\mathcal{L}'$ belongs to the same Neron--Severi class as $\mathcal{L}$.
\end{lem}

\begin{proof}
Given $(A, \mathcal{L}, \beta)$, the choice of another Mumford embedding belonging to the same Neron--Severi class
amounts to the choice of:
\begin{itemize}
    \item a symmetric line bundle $\mathcal{L}'$ belonging to the same Neron--Severi class as $\mathcal{L}$, and
    \item an isomorphism $\mathcal{G}(\delta) \cong \mathcal{G}(\mathcal{L})'$.
\end{itemize}

We will proceed in two steps: first we will describe how to fine one Mumford embedding for each line bundle $\mathcal{L}'$; 
then we will describe how, given a single Mumford embedding for $\mathcal{L}'$, to find all of them.

To say that $\mathcal{L}'$ is a symmetric line bundle in the Neron--Severi class of $\mathcal{L}$
means that $\mathcal{L}' \otimes \mathcal{L}^{-1}$
is a two-torsion point on the dual abelian variety $A^{\vee} = \operatorname{Pic}^0(A)$.
There are $2^{2g}$ such $\mathcal{L}'$;
they are all given as 
\[ T_x^* \mathcal{L} \]
for $x \in A(\mathbb{Q})$ such that $2x \in H_{\mathcal{L}}$, and we can compute them by Lemma \ref{compute_torsion}.

Given any such $x$, the tuple $(A, \mathcal{L}', \beta \circ T_x)$ is a Mumford embedding of $A$.

Now given $(A, \mathcal{L}', \beta_0)$, all other Mumford embeddings for this fixed $\mathcal{L}'$ are given by
\[ \beta' = \phi \circ \beta_0 \]
for some automorphism $\phi$ of $\mathcal{G}(\delta)$
acting as the identity on $\mathbb{G}_m$.
All such automorphisms $\phi$ are described explicitly in
Lemmas \ref{constr_autos} and \ref{automorphisms}.
For each such $\phi$, 
let $f_{\phi} \colon V(\delta) \rightarrow V(\delta)$ be 
an equivariant morphism from $V(\delta)$ (with $\mathcal{G}(\delta)$-action given by $\rho$) to itself (with $\mathcal{G}(\delta)$-action given by $\rho \circ \phi$).
(Note that $f_{\phi}$ can be computed exactly by Lemma \ref{intertwining}.)

Since $f_{\phi}$ is a linear transformation of $V(\delta)$, 
it acts on the projective space $\mathbb{P}(V(\delta))$,
the Mumford embedding 
\[ (A, \mathcal{L}', \beta') \]
is given by postcomposing $A \rightarrow \mathbb{P}(V(\delta))$ with $f_{\phi}$.
\end{proof}

\begin{rmk}
The choice of $x$ such that $T_x^* \mathcal{L} \cong \mathcal{L}'$ is not unique: it is only unique up to $H_{\mathcal{L}}$.
Translation by $H_{\mathcal{L}}$ on the Mumford-embedded $A$ is given by linear transformations $f_s$ of the ambient projective space (see Equation \eqref{translation_linear} in Section \ref{char_delta}).
These are precisely the $f_{\phi}$ arising from $\phi$ an inner automorphism of $\mathcal{G(\delta)}$.
\end{rmk}

\subsection{Detecting an abelian variety in a family.}

\begin{lem}
\label{algo_mumf_family}
Let $\mathcal{A} \rightarrow S$ be an abelian scheme, where $S$ is quasiprojective over a number field $K$, and the map $\mathcal{A} \rightarrow S$ and the structure maps of $\mathcal{A}$ are defined over $K$.

There is an algorithm that produces an \'etale cover $S'$ of $S$, an abelian scheme $\mathcal{A}' / S'$ in Mumford embedding, and an isomorphism $\mathcal{A}' \cong \mathcal{A} \times_S S'$ of $S'$-schemes.
\end{lem}

\begin{proof}
Brute-force search (see Section \ref{sec:brute_force}).  By Lemma \ref{char_delta_lemma}, we can detect whether a given $A \subseteq \mathbb{P}(V(\delta)) \times S$ is an abelian variety in Mumford form.
\end{proof}

\begin{lem}\label{all mumford coordinates are computable}
\label{all_mumf_coords}
Let $(A, \nu)$ be an abelian variety with given polarization over a number field $K$, and fix a $\delta$ as in \S \ref{Delta markings}.  (Assume $\delta$ is the invariants of the polarization.) 

There is an algorithm that takes $A$ and $\delta$ as input, and returns the Mumford coordinates of all $\delta$-markings on $A$ of polarization $\nu$.
\end{lem}

\begin{proof}
Find a single Mumford form by Lemma \ref{algo_mumf_family}, and then find all the others by Lemma \ref{all mumford coordinates are computable from one}.
\end{proof}

\begin{lem}\label{the subfamily of abelian varieties in a family isomorphic to a given one is computable}
There is an algorithm that takes as input a polarized abelian scheme $\mathcal{A} \rightarrow S$, where $S$ is quasiprojective over a number field $K$, and a polarized abelian variety $A$ over $K$, and determines all $s \in S(\overline{\mathbb{Q}})$ such that $\mathcal{A}_s$ and $A$ are geometrically isomorphic (by an isomorphism that respects the polarization).
\end{lem}

\begin{proof}

By multiplying the polarization on $\mathcal{A}$ by $8$, we may assume that its invariant $\delta = (d_1, \ldots, d_g)$ is such that $8 \mid d_i$ for each $i$; similarly for the polarization on $A$.  We may assume that both polarizations have the same invariant; otherwise $\mathcal{A}_s$ and $A$ cannot ever be isomorphic as polarized abelian varieties.

Apply Lemma \ref{algo_mumf_family} to $\mathcal{A} \rightarrow S$ to get a family $\mathcal{A}' \rightarrow S'$ in Mumford embedding, 
and let
\[ Q_{\mathcal{A}'} \colon S' \rightarrow \mathbb{P}(V(\delta)) \]
be the corresponding Mumford coordinates.

Note that $\mathcal{A}_s$ and $A$ are geometrically isomorphic
if and only if there is some choice of $\delta$-marking on $A$ for which its Mumford coordinates agree with $Q_{\mathcal{A}'}(s)$ (Lemma \ref{Mumford coords determine abelian variety}).

By Lemmas \ref{algo_mumf_family} (applied with $S = \operatorname{Spec} K$) and \ref{all_mumf_coords}, we can find all Mumford coordinates $Q_{(A, \mathcal{L}, \beta)}$ on $A$ (for the given polarization).
Then we simply compute, for each $(\mathcal{L}, \beta)$, the inverse image of $Q_{(A, \mathcal{L}, \beta)} \in \mathbb{P}(V(\delta))$ under $Q_{\mathcal{A}'}$.
\end{proof}